\documentclass[10pt]{amsart}
\usepackage{enumerate}
\usepackage{ulem}
\usepackage{cite}
\usepackage{algorithmic,algorithm,amsfonts,amsmath,amssymb,amstext,bbm,
amsthm}
\usepackage{amsmath,amssymb,amsthm,graphicx,epstopdf,mathrsfs,url}
\usepackage[usenames,dvipsnames]{color}
\usepackage{dsfont}   
\usepackage[matha,mathx]{mathabx}
\usepackage{booktabs}
\definecolor{darkred}{rgb}{0.4,0.1,0.1}
\definecolor{darkblue}{rgb}{0.1,0.1,0.4}
\definecolor{darkgrey}{rgb}{0.5,0.5,0.5}
\usepackage[colorlinks=true,linkcolor=darkred,citecolor=Blue]{hyperref}

\numberwithin{equation}{section}
\theoremstyle{plain}
\newtheorem{thm}{Theorem}[section]
\newtheorem{lem}[thm]{Lemma}

\newtheorem{prop}[thm]{Proposition}

\theoremstyle{remark}
\newtheorem{remark}[thm]{Remark}

\theoremstyle{plain}

\newcommand{\hyp}[1]{$C^{2}$-hypersurface as in Definition~\ref{definition_hypersurface}}

\DeclareMathOperator\ran{ran}

\newcommand{\dom}{\mathrm{dom}\,}

\begin{document}
\title[]{Boundary integral formulations of eigenvalue problems for elliptic differential operators with singular interactions and their numerical approximation by boundary element methods}

\author[M. Holzmann]{Markus Holzmann}
\address{Institut f\"{u}r Angewandte Mathematik\\
Technische Universit\"{a}t Graz\\
 Steyrergasse~30, 8010 Graz, Austria\\
E-mail: {holzmann@math.tugraz.at}}

\author[G. Unger]{Gerhard Unger}
\address{
Institut f\"{u}r Angewandte Mathematik\\
Technische Universit\"{a}t Graz\\
 Steyrergasse~30, 8010 Graz, Austria\\
E-mail: {gunger@math.tugraz.at}
} 
\begin{abstract}
 In this paper the discrete eigenvalues of elliptic second order differential operators in $L^2(\mathbb{R}^n)$, $n \in \mathbb{N}$, with singular $\delta$- and $\delta'$-interactions are studied. We show the self-adjointness of these operators and derive equivalent formulations for the eigenvalue problems involving boundary integral operators. 
These formulations are suitable for the numerical computations of  the discrete eigenvalues and the corresponding eigenfunctions  by boundary element methods. We provide convergence results and  show numerical examples.
\end{abstract}


\maketitle

\section{Introduction}

Schr\"odinger operators with singular interactions supported on sets of measure zero play an important role in mathematical physics. The simplest example are Schr\"odinger operators with point interactions, which were already introduced in the beginnings of quantum mechanics \cite{KP31, T35}. The importance of these models comes from the fact that they reflect the physical reality still to a reasonable exactness and that they are explicitly solvable.
The point interactions are used as idealized replacements for regular potentials, which are strongly localized close to those points supporting the interactions, and the eigenvalues can be computed explicitly via an algebraic equation involving the values of the fundamental solution corresponding to the unperturbed operator evaluated at the interaction support, cf. the monograph \cite{AGHH05} and the references therein.

Inspired by this idea, Schr\"odinger operators with singular $\delta$- and $\delta'$-interactions supported on hypersurfaces (i.e. manifolds of codimension one like curves in $\mathbb{R}^2$ or surfaces in $\mathbb{R}^3$) where introduced. Such interactions are used as idealized replacements of regular potentials which are strongly localized in neighborhoods of these hypersurfaces e.g. in the mathematical analysis of leaky quantum graphs, cf. the review \cite{E08} and the references therein, and in the theory of photonic crystals \cite{FK}. Note that in the case of $\delta$-potentials this idealized replacement is rigorously justified by an approximation procedure \cite{BEHL17}.
The self-adjointness and qualitative spectral properties of Schr\"odinger operators with $\delta$- and $\delta'$-interactions are well understood, see e.g. \cite{BGLL15, BLL13b, BEKS, E08, EK15, MPS16} and the references therein, and the discrete eigenvalues can be characterized via an abstract version of the Birman Schwinger principle. However, following the strategy from the point interaction model one arrives,  instead of an algebraic equation, at a boundary integral equation involving  the fundamental solution for the unperturbed operator.

In this paper we suggest  boundary element methods for the numerical approximations of these boundary integral equations. With this idea of computing the eigenvalues of the differential operators with singular interactions numerically, we give a link of these models to the original explicitly solvable models with point interactions. As theoretical framework for the description of the eigenvalues in terms of boundary integral equations we use the theory of eigenvalue problems for holomorphic and meromorphic Fredholm operator-valued functions~\cite{GohbergGoldberg1990, GohbergSigal:1971, KozlovMazya:1999}. For the approximation of this kind of eigenvalue problems by the Galerkin method there exists  a complete convergence analysis in the case that the operator-valued function is holomorphic~\cite{Karma1, Karma2, SU12}. This analysis provides error estimates for the eigenvalues and eigenfunctions as well as  results which guarantee that the approximation of the eigenvalue problems does not have so-called spurious eigenvalues, i.~e., additional eigenvalues which are not related to the original problem.

Other approaches for the numerical approximation of eigenvalues of differential operators with singular interactions are based on finite element methods, where $\mathbb{R}^n$ is replaced by a big ball, whose size can be estimated with the help of Agmon type estimates. Moreover, in \cite{BFT98, EN03, O06} it is shown in various settings in space dimensions $n \in \{2, 3\}$ that Schr\"odinger operators with $\delta$-potentials supported on curves (for $n=2$) or surfaces (for $n=3$) can be approximated in the strong resolvent sense by Hamiltonians with point interactions. An improvement of this approach is presented in~\cite{BO07}. This allows also to compute numerically the eigenvalues of the limit operator.

Let us introduce our problem setting and give an overview of the main results. Consider a strongly elliptic and formally symmetric partial differential operator in $\mathbb{R}^n$, $n \in \mathbb{N}$, of the form
\begin{equation*} 
  \mathcal{P} := -\sum_{j, k=1}^n \partial_k a_{j k} \partial_j + \sum_{j=1}^n \big(a_j \partial_j - \partial_j \overline{a_j} \big) + a,
\end{equation*}
see Section~\ref{section_P} for details. Moreover, let $\Omega_\text{i}$ be a bounded Lipschitz domain with boundary $\Sigma := \partial \Omega_\text{i}$, let $\Omega_\text{e} := \mathbb{R}^n \setminus \overline{\Omega_\text{i}}$, and let $\nu$ be the unit normal to $\Omega_\text{i}$. Eventually, let $\gamma$ be the Dirichlet trace and $\mathcal{B}_\nu$  the conormal derivative at $\Sigma$ (see \eqref{def_B_nu} for the definition). We are interested in the eigenvalues of two kinds of perturbations of $\mathcal{P}$ as self-adjoint operators in $L^2(\mathbb{R}^n)$ which are formally given by
\begin{equation*}
  A_{\alpha} := \mathcal{P} + \alpha \delta_\Sigma \quad \text{and} \quad B_\beta := \mathcal{P} + \beta \langle \delta_\Sigma', \cdot \rangle \delta'_\Sigma,
\end{equation*}
where $\delta_\Sigma$ is the Dirac $\delta$-distribution supported on $\Sigma$ and the interaction strengths $\alpha, \beta$ are real valued functions defined on $\Sigma$ with $\alpha, \beta^{-1} \in L^\infty(\Sigma)$. For $\mathcal{P} = -\Delta$ these operators have been intensively studied e.g. in \cite{BLL13b, BEKS, E08, EK15}, for certain strongly elliptic operators and smooth surfaces several properties of $A_\alpha$ and $B_\beta$ have been investigated in \cite{BGLL15, MPS16}. 
For the realization of $A_\alpha$ as an operator in $L^2(\mathbb{R}^n)$ we remark that if the distribution $A_\alpha f$ is generated by an $L^2$-function, then  $f_{\text{i/e}} := f \upharpoonright \Omega_{\text{i/e}}$ has to fulfill 
\begin{equation} \label{jump_condition_delta}
  \gamma f_\text{i} = \gamma f_\text{e} \quad \text{and} \quad \mathcal{B}_\nu f_\text{e} - \mathcal{B}_\nu f_\text{i} = \alpha \gamma f \quad \text{on } \Sigma,
\end{equation}
as then the singularities at $\Sigma$ compensate, cf. \cite{BLL13b}. In a similar manner, if the distribution $B_\beta f$ is generated by an $L^2$-function, then $f$ has to fulfill 
\begin{equation} \label{jump_condition_delta_prime}
  \mathcal{B}_\nu f_\text{i} = \mathcal{B}_\nu f_\text{e} \quad \text{and} \quad \gamma f_\text{e} - \gamma f_\text{i} = \beta \mathcal{B}_\nu f\quad \text{on } \Sigma.
\end{equation}
Hence, the relations~\eqref{jump_condition_delta} and~\eqref{jump_condition_delta_prime} are necessary conditions for a function $f$ to belong to the domain of definition of $A_\alpha$ and $B_\beta$, respectively.
Our aims  are to show the self-adjointness of $A_\alpha$ and $B_\beta$ in $L^2(\mathbb{R}^n)$ and to fully characterize their discrete spectra in terms of boundary integral operators. We pay particular attention to establish formulations which fit in the framework of eigenvalue problems for holomorphic and meromorphic Fredholm operator-valued functions  and which are accessible for boundary element methods. This requires a thorough analysis of the involved boundary integral operators.

When using   boundary element methods for the approximations of discrete eigenvalues of $A_\alpha$ and $B_\beta$ it is convenient to consider the related transmission problems. A value 
$\lambda$ belongs to the point spectrum of $A_\alpha$ if and only if there exists a nontrivial $f \in L^2(\mathbb{R}^n)$ satisfying
\begin{equation} \label{transmission_delta}
  \begin{split}
    (\mathcal{P} - \lambda) f = 0 \quad &\text{in } \mathbb{R}^n \setminus \Sigma, \\
    \gamma f_\text{i} = \gamma f_\text{e}, \quad &\text{on } \Sigma, \\
    \mathcal{B}_\nu f_\text{e} - \mathcal{B}_\nu f_\text{i} = \alpha \gamma f \quad& \text{on } \Sigma.
  \end{split}
\end{equation}
Similarly, $\lambda$ belongs to the point spectrum of $B_\beta$ if and only if there exists a nontrivial $f \in L^2(\mathbb{R}^n)$ satisfying
\begin{equation} \label{transmission_delta_prime}
  \begin{split}
    (\mathcal{P} - \lambda) f = 0 \quad &\text{in } \mathbb{R}^n \setminus \Sigma, \\
    \mathcal{B}_\nu f_\text{i} = \mathcal{B}_\nu f_\text{e} \quad &\text{on } \Sigma, \\
    \gamma f_\text{e} - \gamma f_\text{i} = \beta \mathcal{B}_\nu f \quad &\text{on } \Sigma.
  \end{split}
\end{equation}
This shows that the eigenvalue problems for $A_\alpha$ and $B_\beta$ are closely related to transmission problems for $\mathcal{P}-\lambda$, as they were treated in \cite{Kleefeld:2013, KSU}, and the strategies presented there are useful for the numerical calculation of the eigenvalues of $A_\alpha$ and $B_\beta$.

For the analysis of the spectra of $A_\alpha$ and $B_\beta$ a good understanding of the unperturbed operator $A_0$ being the self-adjoint realization of $\mathcal{P}$ with no jump condition at $\Sigma$ and some operators related to the fundamental solution of $\mathcal{P} - \lambda$ are necessary. Assume for $\lambda \in \rho(A_0) \cup \sigma_\text{disc}(A_0)$ that $G(\lambda; x, y)$ is the integral kernel of a suitable paramatrix associated to $\mathcal{P}-\lambda$ which is explained in detail in Section~\ref{section_P}; for $\lambda$ in the resolvent set $\rho(A_0)$ it is in fact a fundamental solution for $\mathcal{P} - \lambda$. We remark that the knowledge of $G(\lambda; x, y)$ or at least a good approximation of this function is essential for our numerical considerations. We introduce the single layer potential $\text{SL}(\lambda)$ and the double layer potential $\text{DL}(\lambda)$ acting on sufficiently smooth functions $\varphi: \Sigma \rightarrow \mathbb{C}$ and $x \in \mathbb{R}^n \setminus \Sigma$ as
\begin{equation*}
  \text{SL}(\lambda) \varphi(x) := \int_\Sigma G(\lambda; x,y) \varphi(y) \text{d} \sigma(y) 
\end{equation*}
and
\begin{equation*}
\text{DL}(\lambda) \, \varphi(x) := \int_\Sigma (\mathcal{B}_{\nu, y} G(\lambda; x,y)) \varphi(y) \text{d} \sigma(y).
\end{equation*}
As we will see, all solutions of $(\mathcal{P}-\lambda) f = 0$ will be closely related to the ranges of $\text{SL}(\lambda)$ and $\text{DL}(\lambda)$. Moreover, the boundary integral operators which are formally given by
\begin{equation*}
  \mathcal{S}(\lambda) \varphi := \gamma \text{SL}(\lambda) \varphi, \quad \mathcal{T}(\lambda)' \varphi := \mathcal{B}_\nu (\text{SL}(\lambda) \varphi)_\text{i} + \mathcal{B}_\nu (\text{SL}(\lambda) \varphi)_\text{e},
\end{equation*}
and 
\begin{equation*}
  \mathcal{T}(\lambda) \varphi := \gamma (\text{DL}(\lambda) \varphi)_\text{i} + \gamma (\text{DL}(\lambda) \varphi)_\text{e}, \quad \mathcal{R}(\lambda) \varphi := - \mathcal{B}_\nu \text{DL}(\lambda) \varphi,
\end{equation*}
will play an important role. While the properties of the above operators are well-known for many special cases, e.g. for $\mathcal{P} = -\Delta$, the corresponding results are, to the best of the authors' knowledge, not easily accessible in the literature for general $\mathcal{P}$. Hence, for completeness we spend some efforts in Section~\ref{section_surface_potentials} to provide those properties of the above integral operators which are needed for our considerations.
Eventually, following a strategy from \cite{BR15}, we show that the discrete eigenvalues of $A_0$ can be characterized as the poles of an operator-valued function which is built up by the operators $\mathcal{S}(\lambda)$, $\mathcal{T}(\lambda)$, $\mathcal{T}(\lambda)'$, and $\mathcal{R}(\lambda)$; see also \cite{BMNW08} for related results. Compared to \cite{BR15} our formulation is particularly useful for the application of boundary element methods to compute the discrete eigenvalues of $A_0$ numerically, as the appearing operators are easily accessible for numerical computations.

In order to introduce $A_\alpha$ and $B_\beta$ rigorously, consider the Sobolev spaces
\begin{equation*}
  H^1_\mathcal{P}(\Omega) := \{ f \in H^1(\Omega): \mathcal{P} f \in L^2(\Omega) \}.
\end{equation*}
Inspired by~\eqref{jump_condition_delta} and~\eqref{jump_condition_delta_prime} we define $A_\alpha$ as the partial differential operator in $L^2(\mathbb{R}^n)$ given by
\begin{equation}  \label{def_A_alpha_intro}
  \begin{split}
    A_\alpha f &:= \mathcal{P} f_\text{i} \oplus \mathcal{P} f_\text{e}, \\
    \dom A_\alpha &:= \big\{ f = f_\text{i} \oplus f_\text{e} \in H^1_\mathcal{P}(\Omega_\text{i}) \oplus H^1_\mathcal{P}(\Omega_\text{e}): \gamma f_\text{i} = \gamma f_\text{e}, \, \mathcal{B}_\nu f_\text{e} - \mathcal{B}_\nu f_\text{i} = \alpha \gamma f \big\},
  \end{split}
\end{equation}
and $B_\beta$ by
\begin{equation} \label{def_B_beta_intro} 
  \begin{split}
    B_\beta f &:= \mathcal{P} f_\text{i} \oplus \mathcal{P} f_\text{e}, \\
    \dom B_\beta &:= \big\{ f = f_\text{i} \oplus f_\text{e} \in H^1_\mathcal{P}(\Omega_\text{i}) \oplus H^1_\mathcal{P}(\Omega_\text{e}): \mathcal{B}_\nu f_\text{i} = \mathcal{B}_\nu f_\text{e}, \, \gamma f_\text{e} - \gamma f_\text{i} = \beta \mathcal{B}_\nu f \big\}.
  \end{split}
\end{equation}
In Section~\ref{section_delta} and~\ref{section_delta_prime} we show the self-adjointness of these operators in $L^2(\mathbb{R}^n)$ and via the Weyl theorem that the essential spectra of $A_\alpha$ and $B_\beta$ coincide with the essential spectrum of the unperturbed operator $A_0$. Hence, to know the spectral profile of $A_\alpha$ and $B_\beta$ we have to understand the discrete eigenvalues of these operators.
The characterization of the discrete eigenvalues of $A_\alpha$ and $B_\beta$ in terms of boundary integral equations depends on the discrete spectrum of the unperturbed operator $A_0$ being empty or not. Let us consider the first case. It turns out that $\lambda\in\rho(A_0)$ is a discrete eigenvalue of $A_\alpha$ if and only if there exists a nontrivial $\varphi\in L^2(\Sigma)$ such that 
\begin{equation} \label{Birmal_Schwinger_delta_intro}
  (I + \alpha \mathcal{S}(\lambda)) \varphi = 0.
\end{equation}
Similarly, the existence of  a discrete eigenvalue  $\lambda\in\rho(A_0)$ of $B_\beta$ is  equivalent to the existence of a corresponding nontrivial $\psi\in H^{1/2}(\Sigma)$ which satisfies
\begin{equation} \label{Birmal_Schwinger_delta_prime_intro}
  (\beta^{-1} + \mathcal{R}(\lambda)) \psi = 0.
\end{equation}
As shown in Sections~\ref{section_delta} and~\ref{section_delta_prime} the boundary integral formulations in~\eqref{Birmal_Schwinger_delta_intro} and ~\eqref{Birmal_Schwinger_delta_prime_intro} are eigenvalue problems for holomorphic Fredholm operator-valued functions. These eigenvalue problems can be approximated by standard boundary element methods. The convergence of the approximations follows from well-known abstract convergence results~\cite{Karma1, Karma2, SU12}, which are summarized in Section~\ref{SectionGalerkinApprox}.   
In the case that $\sigma_\text{disc}(A_0)$ is not empty, still all eigenvalues of $A_\alpha$ and $B_\beta$ in $\rho(A_0)$ can be characterized and computed using~\eqref{Birmal_Schwinger_delta_intro} and~\eqref{Birmal_Schwinger_delta_prime_intro}, respectively. For the  possible eigenvalues $A_\alpha$ and $B_\beta$ which lie in $\sigma_\text{disc}(A_0)$ also boundary integral formulations are provided which are accessible by boundary element methods and discussed in detail in Section~\ref{section_delta} and~\ref{section_delta_prime}.

Finally, let us note that our model also contains certain classes of magnetic Schr\"odinger operators with singular interactions with rather strong limitations for the magnetic field. Nevertheless, one could use our strategy and the Birman-Schwinger principle for magnetic Schr\"odinger operators with more general magnetic fields provided in \cite{BEHL19, O06} to compute the discrete eigenvalues of such Hamiltonians numerically. Also, an extension of our results to Dirac operators with $\delta$-shell interactions~\cite{BEHL19QS} would be of interest, but this seems to be a rather challenging problem.

Let us shortly describe the structure of the paper. In Section~\ref{SectionGalerkinApprox} we recall some basic facts on eigenvalue problems of holomorphic Fredholm operator-valued functions and on the approximation of this kind of eigenvalue problems by the Galerkin method. In Section~\ref{section_P} we introduce the elliptic differential operator $\mathcal{P}$ and the associated integral operators and investigate the properties of the unperturbed operator $A_0$. Sections~\ref{section_delta} and~\ref{section_delta_prime} are devoted to the analysis of $A_\alpha$ and $B_\beta$, respectively. We introduce these operators as partial differential operators in $L^2(\mathbb{R}^n)$, show their self-adjointness and derive boundary integral formulations to characterize their discrete eigenvalues. Moreover, we discuss how these boundary integral equations can be solved numerically by boundary element methods, provide convergence results,  and give some numerical examples.

\subsection*{Notations}

Let $X$ and $Y$ be complex Hilbert spaces. The set of all anti-linear bounded functionals on $X$ and $Y$ are denoted by $X^*$ and $Y^*$, respectively, and the sesquilinear duality product in $X^* \times X$, which is linear in the first and anti-linear in the second argument, is $( \cdot, \cdot)$; the underlying spaces of the duality product will be clear from the context. Next, the set of all bounded and everywhere defined linear operators from $X$ to $Y$ is $\mathcal{B}(X, Y)$; if $X=Y$, then we simply write $\mathcal{B}(X) := \mathcal{B}(X,X)$. For $A \in \mathcal{B}(X,Y)$ the adjoint $A^* \in \mathcal{B}(Y^*, X^*)$ is uniquely determined by the relation $(A x, y)  = (x, A^* y)$ for all $x \in X$ and $y \in Y^*$.

If $A$ is a self-adjoint operator in a Hilbert space, then its domain, range, and kernel are denoted by $\dom A$, $\ran A$, and $\ker A$. The resolvent set, spectrum, discrete, essential, and point spectrum are $\rho(A)$, $\sigma(A)$, $\sigma_\text{disc}(A)$, $\sigma_\text{ess}(A)$, and $\sigma_\text{p}(A)$, respectively. Finally, if $\Lambda$ is an open subset of $\mathbb{C}$ and $\mathcal{A}: \Lambda \rightarrow \mathcal{B}(X, X^*)$, then we say that $\lambda \in \Lambda$  is an eigenvalue of $\mathcal{A}(\cdot)$, if $\ker \mathcal{A}(\lambda) \neq \{ 0 \}$.

\subsection*{Acknowledgements}

We are specially grateful to O.~Steinbach for encouraging us to work on this project. Moreover, we thank J.~Behrndt and J.~Rohleder for helpful discussions and literature hints.

\newcommand{\opA}{\mathcal{F}}
\newcommand{\opC}{\mathcal{C}}
\newcommand{\R}{\mathbb{R}}
\newcommand{\N}{\mathbb{N}}
\newcommand{\C}{\mathbb{C}}
\section{Galerkin approximation of eigenvalue problems for holomorphic Fredholm operator-valued functions}\label{SectionGalerkinApprox}

In this section we present basic results of the theory of eigenvalue problems for holomorphic Fredholm operator-valued functions \cite{ GohbergGoldberg1990, KozlovMazya:1999} and  summarize main results of the convergence analysis of the Galerkin approximation of such eigenvalue problems \cite{Karma1,Karma2,Unger:2009}. These results build the abstract framework which we will utilize in order to show the  convergence of the boundary element method  for the approximation of the discrete eigenvalues of $A_\alpha$ as well as of $B_\beta$ which lie in $\rho(A_0)$. Under specified conditions  the convergence for discrete eigenvalues of $A_\alpha$ and $B_\beta$ in $\sigma_\text{disc}(A_0)$ is also guaranteed.    
 
Let $X$ be a Hilbert space and let $\Lambda \subset \C$ be an open and connected subset of $\C$. We consider 
an operator-valued function $\opA:\Lambda \to \mathcal{B}(X,X^*)$ which depends holomorphically on $\lambda \in \Lambda$, i.e., for each    $\lambda_0 \in \Lambda$ the derivative $\opA'(\lambda_0):=\lim_{\lambda \to \lambda_0} \frac{1}{\lambda-\lambda_0} \left\|\opA(\lambda) - \opA(\lambda_0)\right\|_{L( X,X^*)}$ exists.
Moreover, we assume that $\opA(\lambda)$ is a Fredholm operator of index zero for all $\lambda\in\Lambda$ and that it satisfies a so-called G\r{a}rding's  inequality, i.~e., there exists a compact operator $\opC(\lambda):X\rightarrow X^*$ and a constant $c(\lambda)>0$ for all $\lambda\in\Lambda$ such that
\begin{equation}\label{Inequality:Garding}
 \vert \left(\left(\opA(\lambda)+\opC(\lambda)\right)u,u\right)\vert\geq c(\lambda)\|u\|_X^2\quad\text{for all }u\in X.
\end{equation}

We consider the nonlinear eigenvalue problem for the operator-valued function~$\opA(\cdot)$  of the  form: find eigenvalues $\lambda \in \Lambda$ and corresponding eigenelements $u\in  X\setminus \{0\}$ such that 
 \begin{equation} \label{Eq:abstractEVP}
   \opA(\lambda)   u = 0.
 \end{equation}
In the following we assume that the set $\{\lambda\in\Lambda: \exists \opA(\lambda)^{-1}\in \mathcal{B}(X^*,X)\}$    
is not empty. Then the set of eigenvalues in $\Lambda$ has no accumulation points inside of $\Lambda$ \cite[Cor. XI 8.4]{GohbergGoldberg1990}. The dimension of the null space $\ker \opA(\lambda)$ of an 
eigenvalue $\lambda$ is called the geometric
multiplicity of $\lambda$.  
An ordered collection of elements  $u_0, u_1,\ldots,u_{m-1}$ in $X$ is
called  a Jordan chain of $(\lambda,u_0)$, if it is an eigenpair  
 and if
\begin{equation*}
\sum_{j=0}^{n}\frac{1}{j!}\opA^{(j)}(\lambda)u_{n-j}=0 \quad\text{for all }
n=0,1,\ldots, m-1
\end{equation*} 
is satisfied, where $\opA^{(j)}$ denotes the $j$th derivative. 
The length of any Jordan chain of an eigenvalue is
finite \cite[Lem.~A.8.3]{KozlovMazya:1999}.
Elements of any Jordan chain of an eigenvalue $\lambda$ are called generalized
eigenelements of $\lambda$. The closed linear space of all generalized 
eigenelements
of an
eigenvalue $\lambda$ is called generalized eigenspace of $\lambda$ and is 
denoted by
$G(\opA,\lambda)$. The dimension of the generalized eigenspace 
$G(\opA,\lambda)$ is
finite \cite[Prop.~A.8.4]{KozlovMazya:1999} and it is 
referred to as algebraic  multiplicity of $\lambda$.
 
\subsection{Galerkin approximation}
For the approximation of the eigenvalue problem~\eqref{Eq:abstractEVP} we consider  a conforming Galerkin approximation. We assume that $\left( X_N\right)_{N\in \N}$ is a sequence of finite-dimensional subspaces of $ X$ such that the orthogonal projection $ P_N: X \to  X_N$ converges pointwise to the identity $I:X\rightarrow X$, i.e., for all $u \in  X$ we have
 \begin{equation}\label{Eq:ApproxSpace}
  \| P_N u - u\|_ X=\inf_{v_N \in  X_N} \| v_N - u\|_ X \to 0 \quad \text{ as } N \to \infty.
 \end{equation}
The Galerkin approximation of the  eigenvalue problem reads as: find eigenpairs $(\lambda_N,u_{N}) \in \Lambda \times  X_N\setminus \{0\}$ such that 
 \begin{equation}\label{Eq:GalerkinEVP}
  \left( \opA(\lambda_N)   u_N, {v}_N\right) = 0\quad\text{for all }v_n\in X_N.
 \end{equation}

For the formulation of the convergence results  we need the definition of the gap $\delta_V(V_1,V_2)$ of two subspaces $V_1,V_2$ of a normed space $V$:
\begin{equation*}
\delta_V(V_1,V_2):=\sup_{\substack{v_1\in V_1\\
\|v_1\|_V= 1}} \inf_{v_2\in V_2}\|v_1-v_2\|_V.
\end{equation*}
\begin{thm}\label{Theorm:AbstrctConvergenceResults}
Let $\opA:\Lambda\rightarrow \mathcal{B}(X,X^*)$ be a holomorphic operator-valued function and assume that for each $\lambda\in\Lambda$ there exist a compact operator $\opC(\lambda):X\rightarrow X^*$ and a constant $c(\lambda)>0$ such that inequality~\eqref{Inequality:Garding} is satisfied. Further, suppose that $(X_N)_{N\in\N}$ is a sequence of finite-dimensional subspaces of $X$ which fulfills  the property~\eqref{Eq:ApproxSpace}. Then the following holds true:
\begin{itemize}
\item[(i)] (Completeness of the spectrum of the Galerkin eigenvalue problem) For each eigenvalue $\lambda\in\Lambda$ of the operator-valued function $\opA(\cdot)$ there exists a sequence $(\lambda_N)_\N$ of eigenvalues of the Galerkin eigenvalue problem~\eqref{Eq:GalerkinEVP} such that
 \begin{equation*}
  \lambda_N\rightarrow \lambda\quad\text{as }N\rightarrow\infty.
\end{equation*}
\item[(ii)] (Non-pollution of the spectrum of the Galerkin eigenvalue problem) Let $K\subset\Lambda$ be a compact and connected set such that $\partial K$ is a simple 
rectifiable curve. Suppose that  there is no eigenvalue of $\opA(\cdot)$ 
 in $K$. Then there exists an $N_0\in\N$ such that for all 
$N\geq N_0$ the Galerkin eigenvalue problem \eqref{Eq:GalerkinEVP} has 
no eigenvalues in~$K$.  
\item[(iii)] Let $D\subset\Lambda$ be a compact  and    connected set such that $\partial D$ is a simple 
rectifiable curve. Suppose that $\lambda\in
\mathring{D}$ is the only eigenvalue of $\opA$ in $D$.  
Then
there exist an
$N_0\in\N$ and a constant  $c>0$ such that for all $N\geq N_0$  we
have:
\begin{itemize}
\item[(a)] For all eigenvalues $\lambda_N$ of the Galerkin eigenvalue problem~\eqref{Eq:GalerkinEVP} in $D$ 
\begin{equation*}
 \vert \lambda-\lambda_N\vert \leq
c\delta_X(G(\opA,\lambda),X_N)^{1/\ell}\delta_{X}(G(\opA^*,\lambda),X_N)^{1/\ell}
\end{equation*}
holds, where $\opA^*(\cdot):=(\opA(\overline{\cdot}))^*$ is the adjoint function with respect to the pairing $(\cdot,\cdot)$ for $X^*\times X$ and $\ell$ is the maximal length of a Jordan chain corresponding to~
$\lambda$.
\item[(b)] If  $(\lambda_N,u_N)$ is an eigenpair of 
\eqref{Eq:GalerkinEVP} with $\lambda_N\in D$ and  
$\|u_N\|_{X}=1$, then 
\begin{equation*}
 \inf_{u\in \ker(\opA,\lambda)}\| u-u_N\|_{X}
 \leq 
 c 
 \left(\vert \lambda_N-\lambda \vert + \delta_X(\ker(\opA,\lambda),X_N)\right).
\end{equation*}
\end{itemize}
\end{itemize}
\end{thm}
\begin{proof}
The Galerkin method fulfills the required properties in order to apply the abstract convergence results in \cite{Karma1, Karma2, Unger:2009} to eigenvalue problems for  holomorphic operator-valued functions which satisfy inequality~\eqref{Inequality:Garding}, see \cite[Lem.~4.1]{SU12}. We refer to  \cite[Thm.~2]{Karma1} for assertion~(i) and~(ii), and to~\cite[Thm.~3]{Karma2} for~(iii)a). The error estimate in~(iii)b) is a consequence of~\cite[Thm.~4.3.7]{Unger:2009}.
\end{proof}

\section{Strongly elliptic differential operators and associated integral operators} \label{section_P}

In this section we introduce the class of elliptic differential operators which will be perturbed by the singular $\delta$- and $\delta'$-interactions supported on a hypersurface $\Sigma$, and we introduce the integral operators $\mathcal{S}(\lambda)$, $\mathcal{T}(\lambda)$, $\mathcal{T}(\lambda)'$, and $\mathcal{R}(\lambda)$ in Section~\ref{section_surface_potentials} in a mathematically rigorous way and recall their properties, which will be of importance for our further studies. Eventually, in Section~\ref{section_ev_A_0} we show how the discrete eigenvalues of $A_0$ can be characterized with the help of these boundary integral operators. But first, we introduce our notations for function spaces which we use in this paper.

\subsection{Function spaces}

For an open set $\Omega \subset \mathbb{R}^n$, $n \in \mathbb{N}$, and $k \in \mathbb{N}_0 \cup \{ \infty \}$ we write $C^k(\Omega)$ for the set of all $k$-times continuously differentiable functions and 
\begin{equation*}
  C_b^\infty(\Omega) := \{ f \in C^\infty(\Omega): f, \nabla f \text{ are bounded} \}.
\end{equation*}Moreover, the Sobolev spaces of order $s \in \mathbb{R}$ are denoted by $H^s(\Omega)$, see \cite[Chapter~3]{M00} for their definition.

In the following we assume that $\Omega \subset \mathbb{R}^n$ is a Lipschitz domain in the sense of \cite[Definition~3.28]{M00}. We emphasize that $\Omega$ can be bounded or unbounded, but $\partial \Omega$ has to be compact. Note that in this case we can identify $H^s(\mathbb{R}^n \setminus \partial \Omega)$ with $H^s(\Omega) \oplus H^s(\mathbb{R}^n \setminus \overline{\Omega})$. With the help of the integral on $\partial \Omega$ with respect to the Hausdorff measure we get in a natural way the definition of $L^2(\partial \Omega)$. In a similar flavor, we denote the Sobolev spaces on $\partial \Omega$ of order $s \in [0, 1]$ by $H^s(\partial \Omega)$, see \cite{M00} for details on their definition. For $s \in [-1,0]$ we define $H^s(\partial \Omega) := (H^{-s}(\partial \Omega))^*$ as the anti-dual space of $H^{-s}(\partial \Omega)$.

Finally, we recall that the Dirichlet trace operator $C^\infty(\overline{\Omega}) \ni f \mapsto f|_{\partial \Omega}$ can be extended for any $s \in (\frac{1}{2}, \frac{3}{2})$ to a bounded and surjective operator
\begin{equation} \label{Dirichlet_trace}
  \gamma: H^s(\Omega) \rightarrow H^{s-1/2}(\partial \Omega);
\end{equation}
cf. \cite[Theorem~3.38]{M00}.

\subsection{Strongly elliptic differential operators}

Let   $a_{j k}, a_j, a \in C_b^\infty(\mathbb{R}^n)$,  $n \in \mathbb{N}$, and $j, k \in \{ 1, \dots, n\}$, and define the differential operator 
\begin{equation} \label{def_P}
  \mathcal{P} f := -\sum_{j, k=1}^n \partial_k(a_{j k} \partial_j f) + \sum_{j=1}^n \big(a_j \partial_j f - \partial_j(\overline{a_j} f)\big) + a f 
\end{equation}
in the sense of distributions.
We assume that $a_{j k} = \overline{a_{k j}}$ and that $a$ is real valued; then $\mathcal{P}$ is formally symmetric. Moreover, we assume that $\mathcal{P}$ is strongly elliptic, that means there exists a constant $C>0$ independent of $x$ such that
\begin{equation*} 
  \sum_{j, k=1}^n a_{j k}(x) \xi_j \overline{\xi_k} \geq C |\xi|^2
\end{equation*}
holds for all $x \in \mathbb{R}^n$ and all $\xi \in \mathbb{C}^n$.  

Next, define for an open subset $\Omega \subset \mathbb{R}^n$ the sesquilinear form $\Phi_\Omega: H^1(\Omega) \times H^1(\Omega)$ by
\begin{equation} \label{def_Phi_Omega}
  \Phi_\Omega[f, g] := \int_\Omega \left[ \sum_{j, k=1}^n a_{j k} \partial_j f \overline{\partial_k g} + \sum_{j=1}^n \big(a_j (\partial_j f) \overline{g} + f (\overline{a_j \partial_j g}) \big) + a f \overline{g} \right] \text{d} x.
\end{equation}

In the following assume that $\Omega \subset \mathbb{R}^n$ is a Lipschitz set, let $\nu$ be the unit normal vector field at $\partial \Omega$ pointing outwards $\Omega$, denote by $\gamma$ the Dirichlet trace operator, see~\eqref{Dirichlet_trace}, and introduce for $f \in H^2(\Omega)$ the conormal derivative $\mathcal{B}_\nu f$ by
\begin{equation} \label{def_B_nu}
  \mathcal{B}_\nu f := \sum_{k=1}^n \nu_k \sum_{j=1}^n \gamma (a_{j k} \partial_j f) + \sum_{j=1}^n \nu_j \gamma (\overline{a_j} f).
\end{equation}
Then one can show that 
\begin{equation} \label{Green_classic}
  (\mathcal{P} f, g)_{L^2(\Omega)} = \Phi_\Omega[f, g] - (\mathcal{B}_\nu f, \gamma g)_{L^2(\partial \Omega)}, \quad f \in H^2(\Omega), ~g \in H^1(\Omega),
\end{equation}
holds.
Next, we introduce the Sobolev space 
\begin{equation} \label{def_H_P}
  H_{\mathcal{P}}^1(\Omega) := \big\{ f \in H^1(\Omega): \mathcal{P} f \in L^2(\Omega) \big\},
\end{equation}
where $\mathcal{P} f$ is understood in the distributional sense. It is well known that the conormal derivative $\mathcal{B}_\nu$ has a bounded extension
\begin{equation} \label{conormal_derivative_extension}
  \mathcal{B}_\nu: H^1_\mathcal{P}(\Omega) \rightarrow H^{-1/2}(\partial \Omega),
\end{equation}
such that~\eqref{Green_classic} extends to
\begin{equation} \label{Green_extended}
  (\mathcal{P} f, g)_{L^2(\Omega)} = \Phi_\Omega[f, g] - ( \mathcal{B}_\nu f, \gamma g), \quad f \in H^1_\mathcal{P}(\Omega),~g \in H^1(\Omega),
\end{equation}
where the term on the boundary in~\eqref{Green_classic} is replaced by the duality product in $H^{-1/2}(\Sigma)$ and $H^{1/2}(\Sigma)$, see \cite[Lemma~4.3]{M00}. We remark that this formula also holds for $\Omega = \mathbb{R}^n$, then the term on the boundary is not present.

Our first goal is to construct the unperturbed self-adjoint operator $A_0$ in $L^2(\mathbb{R}^n)$ associated to $\mathcal{P}$. With the help of \cite[Theorem~4.7]{M00} it is not difficult to show that the sesquilinear form $\Phi_{\mathbb{R}^n}$ fulfills the assumptions of the first representation theorem \cite[Theorem~VI~2.1]{K95}, so we can define $A_0$ as the self-adjoint operator corresponding to $\Phi_{\mathbb{R}^n}$. The following result is well-known, the simple proof is left to the reader.

\begin{lem} \label{lemma_free_Op}
  Let $\mathcal{P}$ be given by~\eqref{def_P} and let the form $\Phi_{\mathbb{R}^n}$ be defined by~\eqref{def_Phi_Omega}. Then $\Phi_{\mathbb{R}^n}$ is densely defined, symmetric, bounded from below, and closed. The self-adjoint operator $A_0$ in $L^2(\mathbb{R}^n)$ associated to $\Phi_{\mathbb{R}^n}$ is
  \begin{equation} \label{def_free_Op}
    A_0 f  =\mathcal{P} f, \quad \dom A_0  = H^2(\mathbb{R}^n).
  \end{equation}
\end{lem}

Assume that $\Omega_\text{i}$ is a bounded Lipschitz domain in $\mathbb{R}^n$ with boundary $\Sigma := \partial \Omega_\text{i}$, let $\nu$ be the unit normal to $\Omega_\text{i}$, and set $\Omega_\text{e} := \mathbb{R}^n \setminus \overline{\Omega_\text{i}}$.
Then it follows from \cite[Theorem~4.20]{M00} that a function $f = f_\text{i} \oplus f_\text{e} \in H^1_\mathcal{P}(\Omega_\text{i}) \oplus H^1_\mathcal{P}(\Omega_\text{e})$ fulfills
\begin{equation} \label{condition_H2}
  f \in \dom A_0 = H^2(\mathbb{R}^n) \quad \Longleftrightarrow \quad \gamma f_\text{i} = \gamma f_\text{e} \text{ and } \mathcal{B}_\nu f_\text{i} = \mathcal{B}_\nu f_\text{e}.
\end{equation}

Next, we review some properties of the resolvent of $A_0$ which are needed later. In the following, let $\lambda \in \rho(A_0) \cup \sigma_\text{disc}(A_0)$ be fixed. Recall that a map $\mathcal{G}$ is called a {\it paramatrix} for $\mathcal{P}-\lambda$ in the sense of \cite[Chapter~6]{M00}, if there exist integral operators $\mathcal{K}_1, \mathcal{K}_2$ with $C^\infty$-smooth integral kernels such that
\begin{equation*}
  \mathcal{G} (\mathcal{P} - \lambda) u = u - \mathcal{K}_1 u \quad \text{and} \quad (\mathcal{P} - \lambda) \mathcal{G} u = u - \mathcal{K}_2 u
\end{equation*}
holds for all $u \in \mathcal{E}^*(\mathbb{R}^n)$, where $\mathcal{E}^*(\mathbb{R}^n)$ is the set of all distributions with compact support, cf. \cite{M00}. A paramatrix is a {\it fundamental solution} for $\mathcal{P}-\lambda$, if the above equation holds with $\mathcal{K}_1 = \mathcal{K}_2 = 0$.

Let us denote the orthogonal projection onto $\ker(A_0 - \lambda)$ by $\widehat{P}_\lambda$ and set 
\begin{equation} \label{def_P_lambda}
  P_\lambda := I - \widehat{P}_\lambda.
\end{equation}
Note that $P_\lambda = I$ for $\lambda \in \rho(A_0)$ and if $\{ e_1, \dots e_N \}$, $N := \dim \ker (A_0 - \lambda)$, is a basis of $\ker (A_0-\lambda)$ for $\lambda \in \sigma_\text{disc}(A_0)$, then
\begin{equation*}
  \widehat{P}_\lambda f = \sum_{k=1}^N (f, e_k)_{L^2(\mathbb{R}^n)} e_k = \int_{\mathbb{R}^n} K(\cdot,y) f(y) \text{d} y, \quad K(x,y) := \sum_{k=1}^N e_k(x) \overline{e_k(y)},
\end{equation*}
for all $f \in L^2(\mathbb{R}^n)$.
We remark that the integral kernel $K$ is a $C^\infty$-function by elliptic regularity \cite[Theorem~4.20]{M00}. By the spectral theorem we have that $A_0 - \lambda$ is boundedly invertible in $P_\lambda(L^2(\mathbb{R}^n))$. Therefore, the map 
\begin{equation} \label{def_R_lambda}
  \mathcal{G}(\lambda) := P_\lambda (A_0 - \lambda)^{-1} P_\lambda
\end{equation}
is bounded in $L^2(\mathbb{R}^n)$, and it
is a paramatrix for $\mathcal{P} - \lambda$, as 
\begin{equation} \label{equation_paramatrix}
  (\mathcal{P} - \lambda) P_\lambda (A_0 - \lambda)^{-1} P_\lambda f = P_\lambda (A_0 - \lambda)^{-1} P_\lambda (\mathcal{P} - \lambda) f =  P_\lambda f = f - \widehat{P}_\lambda f
\end{equation}
holds for all $f \in C_0^\infty(\mathbb{R}^n)$.
Therefore, by \cite[Theorem~6.3 and Corollary~6.5]{M00} there exists an integral kernel $G(\lambda; x,y)$ such that for almost every $x \in \mathbb{R}^n$
\begin{equation} \label{resolvent_integral}
  \mathcal{G}(\lambda) f(x) = \int_{\mathbb{R}^n} G(\lambda; x,y) f(y) \text{d} y, \quad f \in L^2(\mathbb{R}^n).
\end{equation}
In the following proposition we show some additional mapping properties of $\mathcal{G}(\lambda)$ for $\lambda \in \rho(A_0) \cup \sigma_\text{disc}(A_0)$; they are standard and well-known, but for completeness we give the proof of this proposition.

\begin{prop} \label{proposition_resolvent}
  Let $A_0$ be defined by~\eqref{def_free_Op}, let $\lambda \in \rho(A_0) \cup \sigma_\textup{disc}(A_0)$, and let $\mathcal{G}(\lambda)$ be given by~\eqref{def_R_lambda}. Then, for any $s \in [-2, 0]$ the mapping $\mathcal{G}(\lambda)$ can be extended to a bounded operator
  \begin{equation} \label{mapping_property_s}
    \mathcal{G}(\lambda): H^s(\mathbb{R}^n) \rightarrow H^{s+2}(\mathbb{R}^n).
  \end{equation}
  Moreover, the map 
  \begin{equation*}
    \rho(A_0) \ni \lambda \mapsto (A_0-\lambda)^{-1}
  \end{equation*}
  is holomorphic in $\mathcal{B}(H^s(\mathbb{R}^n), H^{s+2}(\mathbb{R}^n))$.
\end{prop}
\begin{proof}
  Assume that $\lambda \in \rho(A_0) \cup \sigma_\textup{disc}(A_0)$ is fixed. First, we show that 
  \begin{equation} \label{mapping_properties_resolvent}
    \mathcal{G}(\lambda): L^2(\mathbb{R}^n) \rightarrow H^2(\mathbb{R}^n)
  \end{equation}
  is bounded. The operator in~\eqref{mapping_properties_resolvent} is well-defined, as $\ran \mathcal{G}(\lambda) = \ran P_\lambda (A_0-\lambda)^{-1} P_\lambda = P_\lambda \dom (A_0-\lambda)  \subset H^2(\mathbb{R}^n)$. Moreover, we claim that the operator in~\eqref{mapping_properties_resolvent} is closed, then it is also bounded by the closed graph theorem. Let $(f_n) \subset L^2(\mathbb{R}^n)$ be a sequence and let $f \in L^2(\mathbb{R}^n)$ and $g \in H^2(\mathbb{R}^n)$ be such that
  \begin{equation*}
    f_n \rightarrow f \quad \text{in } L^2(\mathbb{R}^n) \quad \text{and} \quad \mathcal{G}(\lambda) f_n \rightarrow g \quad \text{in } H^2(\mathbb{R}^n).
  \end{equation*}
  Since $\mathcal{G}(\lambda)$ is bounded in $L^2(\mathbb{R}^n)$, we get $\mathcal{G}(\lambda) f_n \rightarrow \mathcal{G}(\lambda) f$ in $L^2(\mathbb{R}^n)$. Moreover, as $H^2(\mathbb{R}^n)$ is continuously embedded in $L^2(\mathbb{R}^n)$, we also have
  \begin{equation*}
    \mathcal{G}(\lambda) f_n \rightarrow g \quad \text{in } L^2(\mathbb{R}^n).
  \end{equation*}
  Hence, we conclude $\mathcal{G}(\lambda) f=g$, which shows that the operator in~\eqref{mapping_properties_resolvent} is closed and thus, bounded.
  
  Since the operator in~\eqref{mapping_properties_resolvent} is bounded for any $\lambda \in \rho(A_0) \cup \sigma_\text{disc}(A_0)$, we conclude by duality that also
  \begin{equation*}
    \mathcal{G}(\lambda): H^{-2}(\mathbb{R}^n) \rightarrow L^2(\mathbb{R}^n)
  \end{equation*}
  is bounded. Therefore, interpolation yields that the mapping property~\eqref{mapping_property_s} holds also for all $s \in (-2, 0)$.
  
  In order to show that $\lambda \mapsto (A_0 - \lambda)^{-1}$ is holomorphic in $\mathcal{B}(H^s(\mathbb{R}^n), H^{s+2}(\mathbb{R}^n))$ for any $s\in[-2,0]$ in a fixed point $\lambda_0 \in \rho(A_0)$, we note that the resolvent identity implies
  \begin{equation*}
    \big[1 - (\lambda-\lambda_0) (A_0-\lambda_0)^{-1}\big] (A_0-\lambda)^{-1} = (A_0-\lambda_0)^{-1}.
  \end{equation*}
  If $\lambda$ is close to $\lambda_0$, we deduce from the Neumann formula that $1 - (\lambda-\lambda_0) (A_0-\lambda_0)^{-1}$ is boundedly invertible in $H^{s+2}(\mathbb{R}^n)$ and hence,
  \begin{equation*}
     (A_0-\lambda)^{-1} = \big[1 - (\lambda-\lambda_0) (A_0-\lambda_0)^{-1}\big]^{-1} (A_0-\lambda_0)^{-1}.
  \end{equation*}
  In particular, $(A_0-\lambda)^{-1}$ is uniformly bounded in $\mathcal{B}(H^s(\mathbb{R}^n), H^{s+2}(\mathbb{R}^n))$ for $\lambda$ belonging to a small neighborhood of $\lambda_0$ and continuous in $\lambda$. Employing this and once more the resolvent identity
  \begin{equation*}
    (A_0-\lambda)^{-1} - (A_0-\lambda_0)^{-1} = (\lambda - \lambda_0) (A_0-\lambda)^{-1} (A_0 - \lambda_0)^{-1},
  \end{equation*}
  we find that $\rho(A_0)\ni\lambda \mapsto (A_0-\lambda)^{-1}$ is holomorphic in $\mathcal{B}(H^s(\mathbb{R}^n), H^{s+2}(\mathbb{R}^n))$.
\end{proof}

\subsection{Surface potentials associated to $\mathcal{P}$} \label{section_surface_potentials}

In this section we introduce several families of integral operators associated to the paramatrix $\mathcal{G}(\lambda)$ which will be of importance in the study of $A_\alpha$ and $B_\beta$ and for the numerical calculation of their eigenvalues. Remark that many of the properties shown below are well known for special realizations of $\mathcal{P}$, for instance $\mathcal{P} = - \Delta$, but for completeness we also provide the proofs for general $\mathcal{P}$.

Throughout this section assume that $\Sigma$ is the boundary of a bounded Lipschitz domain $\Omega_\text{i}$, set $\Omega_\text{e} := \mathbb{R}^n \setminus \overline{\Omega_\text{i}}$, and let $\nu$ be the unit normal to $\Omega_\text{i}$. 
If $f$ is a function defined on $\mathbb{R}^n$, then in the following we will often use the notations $f_\text{i} := f \upharpoonright \Omega_\text{i}$ and $f_\text{e} := f \upharpoonright \Omega_\text{e}$.

Recall that the Dirichlet trace operator $\gamma: H^1(\mathbb{R}^n) \rightarrow H^{1/2}(\Sigma)$ is bounded by~\eqref{Dirichlet_trace}. Hence, it has a bounded dual $\gamma^*: H^{-1/2}(\Sigma) \rightarrow H^{-1}(\mathbb{R}^n)$. This allows us to define for $\lambda \in \rho(A_0) \cup \sigma_\text{disc}(A_0)$ the {\it single layer potential}
\begin{equation} \label{def_single_layer_potential}
  \text{SL}(\lambda) := \mathcal{G}(\lambda) \gamma^*: H^{-1/2}(\Sigma) \rightarrow H^1(\mathbb{R}^n).
\end{equation}
By the mapping properties of $\gamma^*$ and Proposition~\ref{proposition_resolvent} the map $\text{SL}(\lambda)$ is well-defined and bounded. Moreover, we have $\ran \text{SL}(\lambda) \subset \ran P_\lambda = L^2(\mathbb{R}^n) \ominus \ker (A_0-\lambda)$. With the help of~\eqref{resolvent_integral} and duality, it is not difficult to show that $\text{SL}(\lambda)$ acts on functions $\varphi \in L^2(\Sigma)$ and almost every $x \in \mathbb{R}^n \setminus \Sigma$ as
\begin{equation*}
  \text{SL}(\lambda) \, \varphi(x) = \int_\Sigma G(\lambda; x,y) \varphi(y) \text{d} \sigma(y).
\end{equation*}
Some further properties of $\text{SL}(\lambda)$ are collected in the following lemma. In particular, the map $\text{SL}(\lambda)$ plays an important role to construct eigenfunctions of the operator $A_\alpha$ defined in~\eqref{def_A_alpha_intro}. For that, we prove in the lemma below the correspondence of the range of $\text{SL}(\lambda)$ with all solutions $f \in H^1_\mathcal{P}(\mathbb{R}^n \setminus \Sigma)$ of the equation
\begin{equation*}
  (\mathcal{P} - \lambda) f = 0 \quad \text{in } \mathbb{R}^n \setminus \Sigma \quad \text{and} \quad \gamma f_\text{i} = \gamma f_\text{e}. 
\end{equation*}
For this purpose we define for $\lambda \in \rho(A_0) \cup \sigma_\text{disc}(A_0)$ the set
\begin{equation} \label{def_M_lambda}
  \mathcal{M}_\lambda := \{ \varphi \in H^{-1/2}(\Sigma): ( \varphi, \gamma f ) = 0 ~\forall f \in \ker(A_0 - \lambda) \}.
\end{equation}
We remark that $\mathcal{M}_\lambda = H^{-1/2}(\Sigma)$ for $\lambda \in \rho(A_0)$.

\begin{lem} \label{lemma_single_layer_potential}
  Let $\textup{SL}(\lambda)$, $\lambda \in \rho(A_0) \cup \sigma_\textup{disc}(A_0)$, be defined by~\eqref{def_single_layer_potential}. Then the following is true:
  \begin{itemize}
    \item[(i)] We have $\ran \textup{SL}(\lambda) \subset H^1_\mathcal{P}(\mathbb{R}^n \setminus \Sigma)$ and
    \begin{equation} \label{range_single_layer_potential}
      \textup{SL}(\lambda)(\mathcal{M}_\lambda) \oplus \ker(A_0-\lambda) = \big\{ f \in H^1(\mathbb{R}^n): (\mathcal{P} - \lambda) f = 0 \text{ in } \mathbb{R}^n \setminus \Sigma \big\}.
    \end{equation}
    \item[(ii)] Let $\mathcal{B}_\nu$ be the conormal derivative defined by~\eqref{Green_extended}. Then for any $\varphi \in H^{-1/2}(\Sigma)$ the jump relations
    \begin{equation*}
      \gamma (\textup{SL}(\lambda)\,\varphi)_\textup{i} - \gamma (\textup{SL}(\lambda)\,\varphi)_\textup{e} = 0 \quad \text{and} \quad
      \mathcal{B}_\nu (\textup{SL}(\lambda)\,\varphi)_\textup{i} - \mathcal{B}_\nu (\textup{SL}(\lambda)\,\varphi)_\textup{e} = \varphi
    \end{equation*}
    hold.
    \item[(iii)] The map
    \begin{equation*}
      \rho(A_0) \ni \lambda \mapsto \textup{SL}(\lambda)
    \end{equation*}
    is holomorphic in $\mathcal{B}(H^{-1/2}(\Sigma), H^1_\mathcal{P}(\mathbb{R}^n \setminus \Sigma))$.
  \end{itemize}
\end{lem}
\begin{proof}
  (i)--(ii) Let $\{ e_1, \dots, e_N \}$ be a basis of $\ker(A_0 - \lambda)$ (we use the convention that this set is empty for $\lambda \in \rho(A_0)$).
  Since $\mathcal{G}(\lambda)$ is a paramatrix for $\mathcal{P}-\lambda$, the considerations in \cite[equation~(6.19)]{M00} and~\eqref{equation_paramatrix} imply for $\varphi \in H^{-1/2}(\Sigma)$ that 
  \begin{equation} \label{range_SL}
    (\mathcal{P}-\lambda) \text{SL}(\lambda) \varphi = -\widehat{P}_\lambda \gamma^* \varphi = -\sum_{j=1}^N ( \varphi, \gamma e_j ) e_j \quad \text{on} \quad \mathbb{R}^n \setminus \Sigma.
  \end {equation}
  This implies, in particular, that $\ran \text{SL}(\lambda) \subset H^1_\mathcal{P}(\mathbb{R}^n \setminus \Sigma)$ and hence, $\mathcal{B}_\nu (\text{SL}(\lambda) \varphi)_{\text{i}/\text{e}}$ is well-defined for $\varphi \in H^{-1/2}(\Sigma)$ by~\eqref{conormal_derivative_extension}. The jump relations in item~(ii) are shown in~\cite[Theorem~6.11]{M00}. 
  Furthermore,~\eqref{range_SL} implies $(\mathcal{P}-\lambda) \text{SL}(\lambda) \varphi  = 0$ in $\mathbb{R}^n \setminus \Sigma$ for $\varphi \in \mathcal{M}_\lambda$ and thus,
  \begin{equation} \label{range_single_layer_potential1}
    \textup{SL}(\lambda) (\mathcal{M}_\lambda) \oplus \ker(A_0-\lambda) \subset \big\{ f \in H^1(\mathbb{R}^n): (\mathcal{P} - \lambda) f = 0 \text{ in } \mathbb{R}^n \setminus \Sigma \big\}.
  \end{equation}
  
  Next, we verify the second inclusion in~\eqref{range_single_layer_potential}. Let $f \in H^1(\mathbb{R}^n) \cap H^1_\mathcal{P}(\mathbb{R}^n \setminus \Sigma)$ such that $(\mathcal{P} - \lambda) f = 0$ in $\mathbb{R}^n \setminus \Sigma$. Set $\varphi := \mathcal{B}_\nu f_\text{i} - \mathcal{B}_\nu f_\text{e} \in H^{-1/2}(\Sigma)$. We claim that $\varphi \in \mathcal{M}_\lambda$. For $\lambda \in \rho(A_0)$ this is clear by the definition of $\mathcal{M}_\lambda$ in \eqref{def_M_lambda}. For $\lambda \in \sigma_\textup{disc}(A_0) \subset \mathbb{R}$ we get with~\eqref{Green_extended} applied  in $\Omega_\text{i}$ and $\Omega_\text{e}$ (note that $\nu$ is pointing outside $\Omega_\text{i}$ and inside $\Omega_\text{e}$) for any $g \in \ker(A_0-\lambda) \subset H^2(\mathbb{R}^n)$
  \begin{equation*}
    \begin{split}
      ( \varphi, \gamma g ) &= (\mathcal{B}_\nu f_\text{i} - \mathcal{B}_\nu f_\text{e}, \gamma g) - (\gamma f, \mathcal{B}_\nu g_\text{i} - \mathcal{B}_\nu g_\text{e}) \\
      &= (f, \mathcal{P} g)_{L^2(\mathbb{R}^n)} - (\mathcal{P} f, g)_{L^2(\mathbb{R}^n)}  
      =(f, \lambda g)_{L^2(\mathbb{R}^n)} - (\lambda f, g)_{L^2(\mathbb{R}^n)} =0,
    \end{split}
  \end{equation*}
  which implies $\varphi \in \mathcal{M}_\lambda$.
  Next, consider the function $h := f - \text{SL}(\lambda) \varphi$. Then $h \in H^1(\mathbb{R}^n)$ and by~(ii) we have
  \begin{equation*}
    \mathcal{B}_\nu h_\text{i} - \mathcal{B}_\nu h_\text{e} = \mathcal{B}_\nu f_\text{i} - \mathcal{B}_\nu f_\text{e} - \big(\mathcal{B}_\nu (\text{SL}(\lambda) \varphi)_\text{i} - \mathcal{B}_\nu (\text{SL}(\lambda) \varphi)_\text{e}\big) = \varphi - \varphi = 0.
  \end{equation*}
  Hence, \eqref{condition_H2} yields $h \in \dom A_0$. Eventually, due to the properties of $f$ and $\text{SL}(\lambda) \varphi$ for $\varphi \in \mathcal{M}_\lambda$ we conclude
  \begin{equation*}
    \begin{split}
      (A_0 - \lambda) h &= (\mathcal{P} - \lambda) h_\text{i} \oplus (\mathcal{P} - \lambda) h_\text{e} \\
      &= (\mathcal{P} - \lambda) (f_\text{i} - (\text{SL}(\lambda) \varphi)_\text{i}) \oplus (\mathcal{P} - \lambda) (f_\text{e} - \text{SL}(\lambda) \varphi)_\text{e}) = 0.
    \end{split}
  \end{equation*}
  This gives $h = f - \text{SL}(\lambda) \varphi \in \ker(A_0 - \lambda)$. Therefore, we have also verified 
  \begin{equation} \label{range_single_layer_potential2}
    \big\{ f \in H^1(\mathbb{R}^n): (\mathcal{P} - \lambda) f = 0 \text{ in } \mathbb{R}^n \setminus \Sigma \big\} \subset \ran \textup{SL}(\lambda) \oplus \ker(A_0-\lambda).
  \end{equation}
  The inclusions in~\eqref{range_single_layer_potential1} and~\eqref{range_single_layer_potential2} imply finally~\eqref{range_single_layer_potential}.
  
  (iii) By the definition of $\text{SL}(\lambda)$ and Proposition~\ref{proposition_resolvent} we have that $\text{SL}(\lambda)$ is holomorphic in $\mathcal{B}(H^{-1/2}(\Sigma), H^1(\mathbb{R}^n))$. Since $\mathcal{P} \text{SL}(\lambda) \varphi = \lambda \text{SL}(\lambda) \varphi$ in $\mathbb{R}^n \setminus \Sigma$ for $\lambda \in \rho(A_0)$ by~(i), we find that the $H^1$-norm is equivalent to the norm in $H^1_\mathcal{P}(\mathbb{R}^n \setminus \Sigma)$ on $\ran \text{SL}(\lambda)$. Therefore, $\text{SL}(\lambda)$ is also holomorphic in $\mathcal{B}(H^{-1/2}(\Sigma), H^1_\mathcal{P}(\mathbb{R}^n \setminus \Sigma))$.
\end{proof}

Two important objects associated to $\text{SL}(\lambda)$ are the {\it single layer boundary integral operator} $\mathcal{S}(\lambda)$, which is defined by
\begin{equation} \label{def_single_layer_bdd_int_op}
  \mathcal{S}(\lambda) : H^{-1/2}(\Sigma) \rightarrow H^{1/2}(\Sigma), \quad \mathcal{S}(\lambda) \varphi = \gamma \text{SL}(\lambda) \varphi = \gamma \mathcal{G}(\lambda) \gamma^* \varphi,
\end{equation}
and the mapping $\mathcal{T}(\lambda)'$, which is given by
\begin{equation} \label{def_T_prime}
  \mathcal{T}(\lambda)' : H^{-1/2}(\Sigma) \rightarrow H^{-1/2}(\Sigma), \quad \mathcal{T}(\lambda)' \varphi = \mathcal{B}_\nu (\text{SL}(\lambda) \varphi)_\text{i} + \mathcal{B}_\nu (\text{SL}(\lambda) \varphi)_\text{e}.
\end{equation}
The operators $\mathcal{S}(\lambda)$ and $\mathcal{T}(\lambda)'$ have for a  density $\varphi \in L^2(\Sigma)$ and almost all $x \in \Sigma$ the integral representations
\begin{equation*}
  \mathcal{S}(\lambda) \, \varphi(x) = \int_\Sigma G(\lambda; x,y) \varphi(y) \text{d} \sigma(y)
\end{equation*}
and
\begin{equation*}
  \mathcal{T}(\lambda)' \, \varphi(x) = 2 \lim_{\varepsilon \searrow 0} \int_{\Sigma \setminus B(x,\varepsilon)} \mathcal{B}_{\nu, x} G(\lambda; x,y) \varphi(y) \text{d} \sigma(y).
\end{equation*}
Some further properties of $\mathcal{S}(\lambda)$ and $\mathcal{T}(\lambda)'$ are stated in the following lemma:

\begin{lem} \label{lemma_single_layer_bdd_int_op}
  Let $\mathcal{S}(\lambda)$ and $\mathcal{T}(\lambda)'$, $\lambda \in \rho(A_0) \cup \sigma_\textup{disc}(A_0)$, be defined by~\eqref{def_single_layer_bdd_int_op} and~\eqref{def_T_prime}, respectively. Then, the following is true:
  \begin{itemize}
    \item[(i)] The restriction $\mathcal{S}_0(\lambda) := \mathcal{S}(\lambda) \upharpoonright L^2(\Sigma)$ has the mapping property $\mathcal{S}_0(\lambda): L^2(\Sigma) \rightarrow H^1(\Sigma)$. In particular, $\mathcal{S}_0(\lambda)$ is compact in $L^2(\Sigma)$.
    \item[(ii)] $\mathcal{S}(\lambda)$ is a Fredholm operator with index zero and there exist a compact operator $\mathcal{C}(\lambda):H^{-1/2}(\Sigma) \rightarrow H^{1/2}(\Sigma)$ and a constant $c(\lambda)>0$ such that
    \begin{equation*}
      \textup{Re}\, (\varphi, (\mathcal{S}(\lambda) + \mathcal{C}(\lambda)) \varphi)  \geq c(\lambda) \| \varphi \|_{H^{-1/2}(\Sigma)}^2
    \end{equation*}
    holds for all $\varphi \in H^{-1/2}(\Sigma)$.
    \item[(iii)] The maps
    \begin{equation*}
      \rho(A_0) \ni \lambda \mapsto \mathcal{S}(\lambda) \quad \text{and} \quad \rho(A_0) \ni \lambda \mapsto \mathcal{T}(\lambda)'
    \end{equation*}
    are holomorphic in $\mathcal{B}(H^{-1/2}(\Sigma), H^{1/2}(\Sigma))$ and $\mathcal{B}(H^{-1/2}(\Sigma))$, respectively.
    \item[(iv)] For any $\varphi \in H^{-1/2}(\Sigma)$ 
    \begin{equation*}
      \mathcal{B}_\nu (\textup{SL}(\lambda) \varphi)_\textup{i} = \frac{1}{2} (\varphi + \mathcal{T}(\lambda)' \varphi) \quad \text{and} \quad 
      \mathcal{B}_\nu (\textup{SL}(\lambda) \varphi)_\textup{e} = \frac{1}{2} (-\varphi + \mathcal{T}(\lambda)' \varphi)
    \end{equation*}
    hold.
  \end{itemize}
\end{lem}
\begin{proof}
  For the proof of the mapping property of $\mathcal{S}_0(\lambda)$ in~(i) we refer to the discussion after \cite[Theorem~6.12]{M00}, the compactness of $\mathcal{S}_0(\lambda)$ follows then from the fact that $H^{1}(\Sigma)$ is compactly embedded in $L^2(\Sigma)$. Statement~(ii) is shown in \cite[Theorem~7.6]{M00}. Item~(iii) is a consequence of Lemma~\ref{lemma_single_layer_potential}~(iii) and the mapping properties of $\gamma$ and $\mathcal{B}_\nu$, respectively. Finally, statement~(iv) follows immediately from Lemma~\ref{lemma_single_layer_potential}~(ii) and the definition of $\mathcal{T}(\lambda)'$ in~\eqref{def_T_prime}.
\end{proof}

Next, we define the double layer potential associated to $\mathcal{P} - \lambda$. For that we recall the definition of the conormal derivative $\mathcal{B}_\nu$ from~\eqref{def_B_nu} and note that $\mathcal{B}_\nu: H^2(\mathbb{R}^2) \rightarrow L^2(\Sigma)$ is bounded. Hence, it admits a dual $\mathcal{B}_\nu^*\in\mathcal{B}( L^2(\Sigma), H^{-2}(\mathbb{R}^n))$ and with the help of Proposition~\ref{proposition_resolvent} (applied for $s=-2$) we can define the {\it double layer potential} as the bounded operator
\begin{equation} \label{def_double_layer_potential}
  \text{DL}(\lambda) := \mathcal{G}(\lambda) \mathcal{B}_\nu^*: L^2(\Sigma) \rightarrow L^2(\mathbb{R}^n).
\end{equation}
Since $\ran \mathcal{G}(\lambda) \subset L^2(\mathbb{R}^n) \ominus \ker(A_0-\lambda)$, we have $\ran \text{DL}(\lambda) \subset L^2(\mathbb{R}^n) \ominus \ker(A_0-\lambda)$.
Using~\eqref{resolvent_integral} and duality it is not difficult to show that $\text{DL}(\lambda)$ acts on functions $\varphi \in L^2(\Sigma)$ and almost all $x \in \mathbb{R}^n \setminus \Sigma$ as
\begin{equation*}
  \text{DL}(\lambda) \, \varphi(x) = \int_\Sigma (\mathcal{B}_{\nu, y} G(\lambda; x,y)) \varphi(y) \text{d} \sigma(y).
\end{equation*}
Some further properties of $\text{DL}(\lambda)$ are collected in the following lemma. In particular, the map $\text{DL}(\lambda)$ plays an important role to construct eigenfunctions of the operator $B_\beta$ defined in~\eqref{def_B_beta_intro}. For that, we investigate the correspondence of the range of $\text{DL}(\lambda)$ with all solutions $f \in H^1_\mathcal{P}(\mathbb{R}^n \setminus \Sigma)$ of the equation
\begin{equation*}
  (\mathcal{P} - \lambda) f = 0 \quad \text{in } \mathbb{R}^n \setminus \Sigma \quad \text{and} \quad \mathcal{B}_\nu f_\text{i} = \mathcal{B}_\nu f_\text{e}. 
\end{equation*}
For this purpose we define for $\lambda \in \rho(A_0) \cup \sigma_\text{disc}(A_0)$ the set
\begin{equation} \label{def_N_lambda}
  \mathcal{N}_\lambda := \{ \varphi \in H^{1/2}(\Sigma): (\varphi, \mathcal{B}_\nu f) = 0 ~\forall f \in \ker(A_0 - \lambda) \}.
\end{equation}
We remark that $\mathcal{N}_\lambda = H^{1/2}(\Sigma)$ for $\lambda \in \rho(A_0)$. In analogy to Lemma~\ref{lemma_single_layer_potential} we have the following properties of $\text{DL}(\lambda)$.

\begin{lem} \label{lemma_double_layer_potential}
  Let $\textup{DL}(\lambda)$, $\lambda \in \rho(A_0) \cup \sigma_\textup{disc}(A_0)$, be defined by~\eqref{def_double_layer_potential}. Then the following is true:
  \begin{itemize}
    \item[(i)] The restriction of $\textup{DL}(\lambda)$ onto $H^{1/2}(\Sigma)$ gives rise to a bounded operator
    \begin{equation*}
      \textup{DL}(\lambda): H^{1/2}(\Sigma) \rightarrow H^1_\mathcal{P}(\mathbb{R}^n \setminus \Sigma) 
    \end{equation*}
    and 
    \begin{equation} \label{range_double_layer_potential}
      \begin{split}
      \textup{DL}(\lambda) (\mathcal{N}_\lambda) &\oplus \ker(A_0-\lambda) \\
      &= \big\{ f \in H^1_\mathcal{P}(\mathbb{R}^n \setminus \Sigma): \mathcal{B}_\nu f_\textup{i} = \mathcal{B}_\nu f_\textup{e},\, (\mathcal{P} - \lambda) f = 0 \text{ in } \mathbb{R}^n \setminus \Sigma \big\}.
      \end{split}
    \end{equation}
    \item[(ii)] Let $\mathcal{B}_\nu$ be the conormal derivative defined by~\eqref{Green_extended}. Then for any $\varphi \in H^{1/2}(\Sigma)$ the jump relations
    \begin{equation*}
      \gamma (\textup{DL}(\lambda)\,\varphi)_\textup{e} - \gamma (\textup{DL}(\lambda)\,\varphi)_\textup{i} = \varphi \quad \text{and} \quad
      \mathcal{B}_\nu (\textup{DL}(\lambda)\,\varphi)_\textup{i} - \mathcal{B}_\nu (\textup{DL}(\lambda)\,\varphi)_\textup{e} = 0
    \end{equation*}
    hold.
    \item[(iii)] The map
    \begin{equation*}
      \rho(A_0) \ni \lambda \mapsto \textup{DL}(\lambda)
    \end{equation*}
    is holomorphic in $\mathcal{B}(H^{1/2}(\Sigma), H^1_\mathcal{P}(\mathbb{R}^n \setminus \Sigma))$.
  \end{itemize}
\end{lem}
\begin{proof}
  The proofs of many statements of this lemma are analogous to the ones in Lemma~\ref{lemma_single_layer_potential}, so we point out only the main differences.
  Since $\mathcal{G}(\lambda)$ is a paramatrix for $\mathcal{P}-\lambda$, the considerations in \cite[equation~(6.19)]{M00} and~\eqref{equation_paramatrix} imply for $\varphi \in H^{1/2}(\Sigma)$ that 
  \begin{equation} \label{range_DL}
    (\mathcal{P}-\lambda) \text{DL}(\lambda) \varphi = -\widehat{P}_\lambda \mathcal{B}_\nu^* \varphi  \quad \text{on} \quad \mathbb{R}^n \setminus \Sigma.
  \end {equation}
  In particular, $\mathcal{P} (\text{DL} \varphi)_{\text{i}/\text{e}} \in L^2(\Omega_{\text{i}/\text{e}})$.
  Next, we show that
  $\textup{DL}(\lambda): H^{1/2}(\Sigma) \rightarrow H^1_\mathcal{P}(\mathbb{R}^n \setminus \Sigma)$ is bounded. Using the last observation and the closed graph theorem it is enough to verify 
  \begin{equation} \label{range_double_layer}
    \text{DL}(\lambda) \varphi \in H^1(\mathbb{R}^n\setminus \Sigma) \quad \text{for} \quad \varphi \in H^{1/2}(\Sigma);
  \end{equation}
  cf. the proof of~\eqref{mapping_properties_resolvent} for a similar argument. To prove~\eqref{range_double_layer} choose $R>0$ such that $\overline{\Omega_\text{i}}$ is contained in the open ball $B(0,R)$ of radius $R$ centered at the origin and a cutoff function $\chi \in C^\infty(\mathbb{R}^n)$ which is supported in $B(0, R+1)$ and satisfies $\chi \upharpoonright B(0,R) \equiv 1$. Moreover, let $\varphi \in H^{1/2}(\Sigma)$ be fixed. Then $\chi \text{DL}(\lambda) \varphi \in H^1(\mathbb{R}^n \setminus \Sigma)$ by \cite[Theorem~6.11]{M00}. Furthermore, $(1-\chi) \text{DL}(\lambda) \varphi$ belongs to $L^2(\mathbb{R}^n)$ and by the product rule we have
  \begin{equation*}
    \begin{split}
      \mathcal{P} &(1-\chi) \text{DL}(\lambda) \varphi = (1-\chi) \mathcal{P} \text{DL}(\lambda) \varphi \\
      &-\sum_{j,k=1}^n \big[ a_{j k} (\partial_k (1-\chi)) (\partial_j \text{DL}(\lambda) \varphi) + \text{DL}(\lambda) \varphi \partial_k(a_{j k} \partial_j (1-\chi) ) \\
      &\qquad \qquad + a_{jk} (\partial_j (1-\chi)) (\partial_k \text{DL}(\lambda) \varphi)\big] \\
      &\qquad \qquad+ \textup{DL}(\lambda) \varphi \sum_{j=1}^n [a_j \partial_j (1-\chi) - \overline{a_j} \partial_j (1-\chi)].
    \end{split}
  \end{equation*}
  Since $\text{supp}\, \nabla (1-\chi) = \text{supp}\, \nabla \chi \subset B(0,R+1)$, we have again with the help of \cite[Theorem~6.11]{M00} that $ (\partial_k (1-\chi)) (\partial_j \text{DL}(\lambda) \varphi) \in L^2(\mathbb{R}^n)$ and thus with $\mathcal{P} \text{DL}(\lambda) \varphi \in L^2(\mathbb{R}^n)$ and $a_j, a_{j k} \in C_b^\infty(\mathbb{R}^n)$ we obtain $\mathcal{P} (1-\chi) \text{DL}(\lambda) \varphi \in L^2(\mathbb{R}^2)$. Therefore, we conclude from elliptic regularity that $(1-\chi) \text{DL}(\lambda) \varphi \in H^2(\mathbb{R}^n)$. This implies eventually that
  \begin{equation*}
    \text{DL}(\lambda) \varphi = \chi \text{DL}(\lambda) \varphi + (1-\chi) \text{DL}(\lambda) \varphi \in H^1(\mathbb{R}^n \setminus \Sigma)
  \end{equation*}
  and thus~\eqref{range_double_layer}. 
  
  Next, item~(ii) is shown in~\cite[Theorem~6.11]{M00}. Furthermore, the relation \eqref{range_double_layer_potential} can be shown in the same way as~\eqref{range_single_layer_potential} using~\eqref{range_DL} instead of~\eqref{range_SL}.

  In order to prove statement~(iii), let $\lambda_0, \lambda \in \rho(A_0)$. Using the resolvent identity we have
  \begin{equation} \label{double_layer_resolvent_identity}
    \begin{split}
      \text{DL}(\lambda) - \text{DL}(\lambda_0) &= \big( (A_0-\lambda_0)^{-1} - (A_0-\lambda)^{-1} \big) \mathcal{B}_\nu^*\\
      &= (\lambda_0 - \lambda) (A_0-\lambda_0)^{-1} (A_0 - \lambda)^{-1} \mathcal{B}_\nu^*.
    \end{split}
  \end{equation}
  Since $(A_0-\lambda_0)^{-1} (A_0 - \lambda)^{-1} \in \mathcal{B}(H^{-2}(\mathbb{R}^n), H^2(\mathbb{R}^n))$ is continuous in $\lambda$ in this topology, see Proposition~\ref{proposition_resolvent}, we conclude that $\text{DL}(\lambda):H^{1/2}(\Sigma) \rightarrow H^1_\mathcal{P}(\mathbb{R}^n \setminus \Sigma)$ is holomorphic.
\end{proof}

Two important objects associated to $\text{DL}(\lambda)$ are the {\it hypersingular boundary integral operator} $\mathcal{R}(\lambda)$, which is defined by
\begin{equation} \label{def_hypersing_bdd_int_op}
  \mathcal{R}(\lambda): H^{1/2}(\Sigma) \rightarrow H^{-1/2}(\Sigma), \quad \mathcal{R}(\lambda) \varphi = -\mathcal{B}_\nu \text{DL}(\lambda) \varphi = -\mathcal{B}_\nu \mathcal{G}(\lambda) \mathcal{B}_\nu^* \varphi,
\end{equation}
and the operator
\begin{equation} \label{def_T}
  \mathcal{T}(\lambda): H^{1/2}(\Sigma) \rightarrow H^{1/2}(\Sigma), \quad \mathcal{T}(\lambda) \varphi = \gamma (\text{DL}(\lambda) \varphi)_\text{i} + \gamma (\text{DL}(\lambda) \varphi)_\text{e}.
\end{equation}
It follows from Lemma~\ref{lemma_double_layer_potential}~(i) and~\eqref{conormal_derivative_extension} that $\mathcal{R}(\lambda)$ and $\mathcal{T}(\lambda)$ are well-defined and bounded.
While $\mathcal{T}(\lambda)$ has for a continuous density $\varphi \in C(\Sigma)$ and almost all $x \in \Sigma$ a representation as a strongly singular integral operator,
\begin{equation*}
  \mathcal{T}(\lambda) \, \varphi(x) = 2 \lim_{\varepsilon \searrow 0} \int_{\Sigma \setminus B(x,\varepsilon)} (\mathcal{B}_{\nu, y} G(\lambda; x,y)) \varphi(y) \text{d} \sigma(y),
\end{equation*}
the hypersingular operator $\mathcal{R}(\lambda)$ can be only written as finite part integral
\begin{equation*}
  \mathcal{R}(\lambda) \, \varphi(x) = -\textup{f.p.}_{\varepsilon \searrow 0} \int_{\Sigma \setminus B(x,\varepsilon)} \mathcal{B}_{\nu, x} (\mathcal{B}_{\nu, y} G(\lambda; x,y)) \varphi(y) \text{d} \sigma(y),
\end{equation*}
see \cite[Section~7]{M00} for details. However, for special realizations of $\mathcal{P}$ the duality product $(\mathcal{R}(\lambda) \varphi, \psi)$ can be computed in a more convenient way, cf. e.g. \cite[Theorem~8.21]{M00}.
Some further properties of $\mathcal{R}(\lambda)$ and $\mathcal{T}(\lambda)$ are stated in the following lemma:

\begin{lem} \label{lemma_hypersing_bdd_int_op}
  Let $\mathcal{R}(\lambda)$ and $\mathcal{T}(\lambda)$, $\lambda \in \rho(A_0) \cup \sigma_\textup{disc}(A_0)$, be defined by~\eqref{def_hypersing_bdd_int_op} and~\eqref{def_T}, respectively. Then, the following is true:
  \begin{itemize}
    \item[(i)] $\mathcal{R}(\lambda)$ is a Fredholm operator with index zero and there exist a compact operator $\mathcal{C}(\lambda):H^{1/2}(\Sigma) \rightarrow H^{-1/2}(\Sigma)$ and a constant $c(\lambda)>0$ such that
    \begin{equation*}
      \textup{Re}\, (\varphi, (\mathcal{R}(\lambda) + \mathcal{C}(\lambda)) \varphi)  \geq c(\lambda) \| \varphi \|_{H^{1/2}(\Sigma)}^2
    \end{equation*}
    holds for all $\varphi \in H^{1/2}(\Sigma)$.
    \item[(ii)] The maps
    \begin{equation*}
      \rho(A_0) \ni \lambda \mapsto \mathcal{R}(\lambda) \quad \text{and} \quad \rho(A_0) \ni \lambda \mapsto \mathcal{T}(\lambda)
    \end{equation*}
    are holomorphic in $\mathcal{B}(H^{1/2}(\Sigma), H^{-1/2}(\Sigma))$ and $\mathcal{B}(H^{1/2}(\Sigma))$, respectively.
    \item[(iii)] For any $\varphi \in H^{1/2}(\Sigma)$ 
    \begin{equation*}
      \gamma (\textup{DL}(\lambda) \varphi)_\textup{i} = \frac{1}{2} (-\varphi + \mathcal{T}(\lambda) \varphi) \quad \text{and} \quad 
      \gamma (\textup{DL}(\lambda) \varphi)_\textup{e} = \frac{1}{2} (\varphi + \mathcal{T}(\lambda) \varphi)
    \end{equation*}
    hold.
    \item[(iv)] For all $\lambda, \nu \in \rho(A_0)$ the difference $\mathcal{T}(\lambda) - \mathcal{T}(\nu)$ is compact. 
    \item[(v)] The relation
    \begin{equation*}
      (\varphi, \mathcal{T}(\lambda) \psi)
          = (\mathcal{T}(\overline{\lambda})' \varphi, \psi)
    \end{equation*} 
    holds for all $\varphi \in H^{-1/2}(\Sigma)$ and $\psi \in H^{1/2}(\Sigma)$.
  \end{itemize}
\end{lem}
\begin{proof}
  Item~(i) follows immediately from \cite[Theorem~7.8]{M00}.   
  Assertion~(ii) is a consequence of Lemma~\ref{lemma_double_layer_potential}~(iii) and the mapping properties of $\gamma$ and $\mathcal{B}_\nu$ in~\eqref{Dirichlet_trace} and~\eqref{conormal_derivative_extension}. Next, the claim of item~(iii) follows directly from Lemma~\ref{lemma_double_layer_potential}~(ii) and the definition of $\mathcal{T}(\lambda)$.
  
  To show statement~(iv) assume that $\lambda \neq \nu \in \rho(A_0)$. As in~\eqref{double_layer_resolvent_identity} we see that $\text{DL}(\lambda) - \text{DL}(\nu): L^2(\Sigma) \rightarrow H^2(\mathbb{R}^n)$ is bounded. Since $H^{2}(\mathbb{R}^n)$ is boundedly embedded in $H^1(\mathbb{R}^n)$, we deduce with the mapping properties of $\gamma$ from~\eqref{Dirichlet_trace} that
  \begin{equation*}
    \mathcal{T}(\lambda) - \mathcal{T}(\nu) = (\nu - \lambda) \gamma (A_0-\nu)^{-1} (A_0 - \lambda)^{-1} \mathcal{B}_\nu^*
  \end{equation*}
  is bounded from $L^2(\Sigma)$ to $H^{1/2}(\Sigma)$. Since $H^{1/2}(\Sigma)$ is compactly embedded in $L^2(\Sigma)$, we conclude eventually that $\mathcal{T}(\lambda) - \mathcal{T}(\nu)$ is compact in $H^{1/2}(\Sigma)$.
  
  Finally, statement~(v) is shown in \cite[Chapter~7]{M00}, since the operator $T^*$ in \cite[Chapter~7]{M00} coincides with $\mathcal{T}(\overline{\lambda})'$.
\end{proof}

\subsection{Characterization of discrete eigenvalues of $A_0$} \label{section_ev_A_0}

In this section we show how the discrete eigenvalues of $A_0$ can be characterized with the help of the boundary integral operators $\mathcal{S}(\lambda), \mathcal{T}(\lambda), \mathcal{T}(\lambda)'$, and $\mathcal{R}(\lambda)$. For that purpose we follow closely considerations from \cite{BR15}, but we adapt the arguments to obtain a formulation on more general hypersurfaces $\Sigma$ which is also more convenient for numerical considerations.

We define for $\lambda \in \rho(A_0)$ the operator 
\begin{equation} \label{def_A_lambda}
  \begin{split}
    \mathcal{A}(\lambda)&: H^{-1/2}(\Sigma) \times H^{1/2}(\Sigma) \rightarrow H^{1/2}(\Sigma) \times H^{-1/2}(\Sigma), \\
    \mathcal{A}(\lambda) \begin{pmatrix} \varphi \\ \psi \end{pmatrix} &= \begin{pmatrix} \gamma \big( \text{SL}(\lambda) \varphi + \text{DL}(\lambda) \psi \big)_\text{i} \\ -\mathcal{B}_\nu \big( \text{SL}(\lambda) \varphi + \text{DL}(\lambda) \psi \big)_\text{e} \end{pmatrix}.
  \end{split}
\end{equation}
Due to the mapping properties of $\gamma$ from~\eqref{Dirichlet_trace} and $\mathcal{B}_\nu$ from~\eqref{conormal_derivative_extension} we get with Lemma~\ref{lemma_single_layer_potential}~(i) and Lemma~\ref{lemma_double_layer_potential}~(i) that $\mathcal{A}(\lambda)$ is well-defined and bounded.
With Lemma~\ref{lemma_single_layer_bdd_int_op}~(iv) and Lemma~\ref{lemma_hypersing_bdd_int_op}~(iii) we see that $\mathcal{A}(\lambda)$ can be written as the block operator matrix
\begin{equation} \label{A_lambda_block}
  \mathcal{A}(\lambda) = \begin{pmatrix} \mathcal{S}(\lambda) & \tfrac{1}{2} (-I + \mathcal{T}(\lambda)) \\ \tfrac{1}{2} (I - \mathcal{T}(\lambda)') & \mathcal{R}(\lambda) \end{pmatrix}.
\end{equation}
Some basic properties of $\mathcal{A}(\lambda)$ are collected in the following lemma:

\begin{lem} \label{lemma_A_lambda}
  Let $\mathcal{A}(\lambda)$, $\lambda \in \rho(A_0)$, be defined by~\eqref{def_A_lambda}. Then the following is true:
  \begin{itemize}
    \item[(i)] The map $\rho(A_0) \ni \lambda \mapsto \mathcal{A}(\lambda)$ is holomorphic.
    \item[(ii)] There exists a compact operator $\mathcal{K}(\lambda)$ and a constant $c(\lambda) > 0$ such that
    \begin{equation*}
      \left|\left( (\mathcal{A}(\lambda) + \mathcal{K}(\lambda)) \begin{pmatrix} \varphi \\ \psi \end{pmatrix}, \begin{pmatrix} \varphi \\ \psi \end{pmatrix} \right) \right| \geq c(\lambda) \big( \| \varphi \|_{H^{-1/2}(\Sigma)}^2 + \| \psi \|_{H^{1/2}(\Sigma)}^2 \big)
    \end{equation*}
    holds for all $\varphi \in H^{-1/2}(\Sigma)$ and $\psi \in H^{1/2}(\Sigma)$, where the duality product is the one for the pairing $H^{1/2}(\Sigma) \times H^{-1/2}(\Sigma)$ and $H^{-1/2}(\Sigma) \times H^{1/2}(\Sigma)$.
  \end{itemize}
\end{lem}
\begin{proof}
  Assertion~(i) follows from Lemma~\ref{lemma_single_layer_bdd_int_op}~(iii) and Lemma~\ref{lemma_hypersing_bdd_int_op}~(ii), as $\mathcal{S}(\lambda)$, $\mathcal{T}(\lambda)$, $\mathcal{T}(\lambda)',$ and $\mathcal{R}(\lambda)$ are holomorphic. To prove item~(ii) we compute
  \begin{equation*}
    \begin{split}
      \bigg( \mathcal{A}(\lambda)& \begin{pmatrix} \varphi \\ \psi \end{pmatrix}, \begin{pmatrix} \varphi \\ \psi \end{pmatrix} \bigg) = \left( \begin{pmatrix} \mathcal{S}(\lambda) & \tfrac{1}{2} (-I + \mathcal{T}(\lambda)) \\ \tfrac{1}{2} (I - \mathcal{T}(\lambda)') & \mathcal{R}(\lambda) \end{pmatrix} \begin{pmatrix} \varphi \\ \psi \end{pmatrix}, \begin{pmatrix} \varphi \\ \psi \end{pmatrix} \right) \\
      &= (\mathcal{S}(\lambda) \varphi, \varphi) + (\mathcal{R}(\lambda) \psi, \psi) + \frac{1}{2} \big( (\varphi, \psi) - (\psi, \varphi) \big) \\
      &\qquad + \frac{1}{2} \big( (\mathcal{T}(\lambda) \psi, \varphi) - (\varphi, \mathcal{T}(\lambda) \psi) \big) + \frac{1}{2} \big(  (\varphi, \mathcal{T}(\lambda) \psi) - (\mathcal{T}(\lambda)' \varphi, \psi)  \big).
    \end{split}
  \end{equation*}
  With Lemma~\ref{lemma_hypersing_bdd_int_op}~(v) we have
  \begin{equation*}
    \big(\varphi, \mathcal{T}(\lambda) \psi\big) - \big(\mathcal{T}(\lambda)' \varphi, \psi\big) = \big(\varphi, (\mathcal{T}(\lambda) - \mathcal{T}(\overline{\lambda}) )\psi\big) 
  \end{equation*}
  and the operator $\mathcal{T}(\lambda) - \mathcal{T}(\overline{\lambda})$ is compact by Lemma~\ref{lemma_hypersing_bdd_int_op}~(iv). Therefore, we get with a compact operator $\mathcal{K}(\lambda)$
  \begin{equation*}
    \begin{split}
      \text{Re}\, \bigg( \mathcal{A}(\lambda) \begin{pmatrix} \varphi \\ \psi \end{pmatrix}, \begin{pmatrix} \varphi \\ \psi \end{pmatrix} \bigg) 
      = \text{Re}\, \big( (\mathcal{S}(\lambda) \varphi, \varphi) + (\mathcal{R}(\lambda) \psi, \psi) + \big(\varphi, (\mathcal{T}(\lambda) - \mathcal{T}(\overline{\lambda}) )\psi \big)  \big)& \\
      \geq c(\lambda) \big( \| \varphi \|_{H^{-1/2}(\Sigma)}^2 + \| \psi \|_{H^{1/2}(\Sigma)}^2 \big) + \text{Re}\, \left( \mathcal{K}(\lambda) \begin{pmatrix} \varphi \\ \psi \end{pmatrix}, \begin{pmatrix} \varphi \\ \psi \end{pmatrix} \right)&,
    \end{split}
  \end{equation*}
  which implies because of $|z| \geq \text{Re}\, z$ for $z \in \mathbb{C}$ the claimed result.
\end{proof}

In the following theorem we characterize the discrete eigenvalues of $A_0$ with the help of the operator-valued function $\mathcal{A}$. For that we define for a number $\lambda_0 \in \sigma_\text{disc}(A_0) \cup \rho(A_0) = \mathbb{C} \setminus \sigma_\text{ess}(A_0)$, for which there exists an $\varepsilon > 0$ with $B(\lambda_0, \varepsilon) \setminus \{ \lambda_0 \} \subset \rho(A_0)$, the map
\begin{equation} \label{def_residual}
  R_{\mathcal{A}(\lambda_0)} := \lim_{\lambda \rightarrow \lambda_0} (\lambda - \lambda_0) \mathcal{A}(\lambda).
\end{equation}
The proof of the following theorem follows closely ideas from \cite[Theorem~3.2]{BR15}, but the operator $\mathcal{A}(\lambda)$ appearing in our formulation is easier accessible for numerical applications as the map $M(\lambda)$ in \cite{BR15} since it consists of explicitly computable integral operators.

\begin{thm} \label{theorem_spectrum_A_0}
  A number $\lambda_0$ belongs to the discrete spectrum of $A_0$ if and only if $\lambda_0$ is a pole of $\mathcal{A}(\lambda)$. Moreover,
  \begin{equation} \label{equation_eigenspace}
    \ran R_{\mathcal{A}(\lambda_0)} = \big\{ (\gamma f, \mathcal{B}_\nu f)^\top: f \in \ker (A_0  -\lambda_0) \big\}
  \end{equation}
  holds.
\end{thm}
\begin{proof}
  Let $\lambda_0 \notin \sigma_\text{ess}(A_0)$. It suffices to show that~\eqref{equation_eigenspace} is true. Let $\mu \in \mathbb{C} \setminus \mathbb{R}$ be fixed and let $\widehat{P}_{\lambda_0}$ be the orthogonal projection in $L^2(\mathbb{R}^n)$ onto $\ker(A_0 - \lambda_0)$. We claim first that
  \begin{equation} \label{equation_kernel_A_0}
    \ker (A_0 - \lambda_0) = \big\{ \widehat{P}_{\lambda_0} [\text{SL}(\mu) \varphi + \text{DL}(\mu) \psi]: \varphi \in H^{-1/2}(\Sigma), \psi \in H^{1/2}(\Sigma) \}.
  \end{equation}
  To show this assume that $f \in \ker(A_0 - \lambda_0)$ is such that
  \begin{equation*}
    0 = \big( f, \widehat{P}_{\lambda_0} [\text{SL}(\mu) \varphi + \text{DL}(\mu) \psi] \big)_{L^2(\mathbb{R}^n)} = \big( f, \text{SL}(\mu) \varphi + \text{DL}(\mu) \psi \big)_{L^2(\mathbb{R}^n)}
  \end{equation*}
  holds for all $\varphi \in H^{1/2}(\Sigma)$ and $\psi \in H^{-1/2}(\Sigma)$. Since $f \in \ker(A_0 - \lambda_0)$, we have $(A_0-\overline{\mu})^{-1} f = (\lambda_0 - \overline{\mu})^{-1} f$ and thus, the definitions of $\text{SL}(\mu)$ and $\text{DL}(\mu)$ lead to
  \begin{equation*}
    \begin{split}
      0 &= \big( f, (A_0 - \mu)^{-1} \gamma^* \varphi + (A_0 - \mu)^{-1} \mathcal{B}_\nu^* \psi \big)_{L^2(\mathbb{R}^n)} \\
      &= \big( \gamma(A_0 - \overline{\mu})^{-1} f, \varphi \big) + \big( \mathcal{B}_\nu (A_0 - \overline{\mu})^{-1} f, \psi \big) \\
      &= \frac{1}{\lambda_0 - \overline{\mu}} \big[ ( \gamma f, \varphi ) + ( \mathcal{B}_\nu f, \psi ) \big].
    \end{split}
  \end{equation*}
  Since this is true for all $\varphi \in H^{1/2}(\Sigma)$ and $\psi \in H^{-1/2}(\Sigma)$, we conclude that $\gamma f = \mathcal{B}_\nu f = 0$. It follows from \cite[Proposition~2.5]{BR12} (this result and its proof are also true for unbounded domains) that $f = 0$. Since for $\lambda_0 \notin \sigma_\text{ess}(A_0)$ the set $\ker(A_0 - \lambda_0)$ is finite-dimensional, \eqref{equation_kernel_A_0} is shown.
  
  We are now prepared to prove~\eqref{equation_eigenspace}. By the spectral theorem the resolvent of $A_0$ can be written in a small neighborhood of $\lambda_0$ as 
  \begin{equation*}
    (A_0 - \mu)^{-1} = \frac{1}{\lambda_0 - \mu} \widehat{P}_{\lambda_0} + \mathcal{F}(\mu),
  \end{equation*}
  where $\mathcal{F}(\mu)$ is a locally bounded and continuous operator in $\mu$. Hence, we conclude that $R_{\mathcal{A}(\lambda_0)}$ can be a nontrivial operator, only if $\widehat{P}_{\lambda_0}$ is nontrivial, and that 
  \begin{equation*}
    \ran R_{\mathcal{A}(\lambda_0)} \subset \big\{ (\gamma f, \mathcal{B}_\nu f)^\top: f \in \ker (A_0  -\lambda_0) \big\}.
  \end{equation*}
  To show the other inclusion in~\eqref{equation_eigenspace}, let $f \in \ker(A_0 - \lambda_0)$, fix $\mu \in \mathbb{C} \setminus \mathbb{R}$, and choose $\varphi \in H^{-1/2}(\Sigma)$ and $\psi \in H^{1/2}(\Sigma)$ such that $f = \widehat{P}_{\lambda_0} [\text{SL}(\mu) \varphi + \text{DL}(\mu)\psi]$; such a choice is always possible by~\eqref{equation_kernel_A_0}. Note that according to the spectral theorem we have $\widehat{P}_{\lambda_0} g = \lim_{\lambda \rightarrow \lambda_0} (\lambda_0-\lambda) (A_0-\lambda)^{-1} g$, where the limit is the one in $L^2(\mathbb{R}^n)$. Hence, we find
  \begin{equation*}
    \begin{split}
      \begin{pmatrix} \gamma f \\ \mathcal{B}_\nu f \end{pmatrix} 
      &= \begin{pmatrix} \gamma \\ \mathcal{B}_\nu \end{pmatrix} (A_0 - \mu)^{-1} (A_0-\mu) \widehat{P}_{\lambda_0} [\text{SL}(\mu) \varphi + \text{DL}(\mu)\psi] \\
      &= (\lambda_0 - \mu) \begin{pmatrix} \gamma \\ \mathcal{B}_\nu \end{pmatrix} (A_0 - \mu)^{-1}  \widehat{P}_{\lambda_0} [\text{SL}(\mu) \varphi + \text{DL}(\mu)\psi] \\
      &= (\lambda_0 - \mu) \begin{pmatrix} \gamma \\ \mathcal{B}_\nu \end{pmatrix} (A_0 - \mu)^{-1}  \lim_{\lambda \rightarrow \lambda_0} (\lambda_0-\lambda) (A_0-\lambda)^{-1} [\text{SL}(\mu) \varphi + \text{DL}(\mu)\psi].
    \end{split}
  \end{equation*}
  Note that the mapping
  \begin{equation*}
    \begin{pmatrix} \gamma \\ \mathcal{B}_\nu \end{pmatrix} (A_0 - \mu)^{-1}: L^2(\mathbb{R}^n) \rightarrow H^{1/2}(\Sigma) \times H^{-1/2}(\Sigma) 
  \end{equation*}
  is continuous. Hence, we conclude
  \begin{equation*}
    \begin{split}
      &\begin{pmatrix} \gamma f \\ \mathcal{B}_\nu f \end{pmatrix} 
      = \lim_{\lambda \rightarrow \lambda_0} (\lambda_0-\lambda) (\lambda_0 - \mu) \begin{pmatrix} \gamma \\ \mathcal{B}_\nu \end{pmatrix} (A_0 - \mu)^{-1} (A_0-\lambda)^{-1} [\text{SL}(\mu) \varphi + \text{DL}(\mu)\psi] \\
      &\qquad= \lim_{\lambda \rightarrow \lambda_0} (\lambda_0-\lambda) (\lambda_0 - \mu) \begin{pmatrix} \gamma \\ \mathcal{B}_\nu \end{pmatrix} (A_0 - \mu)^{-1} (A_0-\lambda)^{-1} (A_0 - \mu)^{-1} [\gamma^* \varphi + \mathcal{B}_\nu^* \psi].
    \end{split}
  \end{equation*}
  Applying two times the resolvent identity, we find first for $g \in L^2(\mathbb{R}^n)$ that
  \begin{equation*}
    \begin{split}
      (A_0 - \mu)^{-1} &(A_0-\lambda)^{-1} (A_0 - \mu)^{-1} g = \frac{1}{\mu - \lambda} [(A_0 - \mu)^{-1} - (A_0 - \lambda)^{-1}] (A_0 - \mu)^{-1} g \\
      &= \frac{1}{\mu - \lambda} (A_0 - \mu)^{-2} g - \frac{1}{(\mu - \lambda)^2} [(A_0 - \mu)^{-1} - (A_0 - \lambda)^{-1}] g.
    \end{split}
  \end{equation*}
  With a continuity argument this extends to all $g \in H^{-2}(\mathbb{R}^n)$. Using this, we find finally
  \begin{equation*}
    \begin{split}
      \begin{pmatrix} \gamma f \\ \mathcal{B}_\nu f \end{pmatrix} \!
      &=\!\! \lim_{\lambda \rightarrow \lambda_0} (\lambda_0-\lambda) (\lambda_0 - \mu) \! \begin{pmatrix} \gamma \\ \mathcal{B}_\nu \end{pmatrix} \!(A_0 - \mu)^{-1} (A_0-\lambda)^{-1} (A_0 - \mu)^{-1} [\gamma^* \varphi \! + \! \mathcal{B}_\nu^* \psi] \\
      &= \lim_{\lambda \rightarrow \lambda_0} \frac{(\lambda_0 - \lambda)(\lambda_0-\mu)}{(\lambda - \mu)^2} \begin{pmatrix} \gamma \\ \mathcal{B}_\nu \end{pmatrix} [(A_0-\lambda)^{-1}  \gamma^* \varphi + (A_0-\lambda)^{-1} \mathcal{B}_\nu^* \psi] \\
      &= \lim_{\lambda \rightarrow \lambda_0} \frac{(\lambda_0 - \lambda)(\lambda_0-\mu)}{(\lambda - \mu)^2} \mathcal{A}(\lambda) \begin{pmatrix} \varphi \\ \psi \end{pmatrix} 
          = \frac{1}{\lambda_0 - \mu} R_{\mathcal{A}(\lambda_0)} \begin{pmatrix} \varphi \\ \psi \end{pmatrix},
    \end{split}
  \end{equation*}
  which shows that also the second inclusion in~\eqref{equation_eigenspace} is true. This finishes the proof of this theorem.
\end{proof}

\section{Elliptic differential operators with $\delta$-potentials supported on compact Lipschitz smooth surfaces}
\label{section_delta}

This section is devoted to the study of the spectral properties of the differential operator which is formally given by $A_\alpha := \mathcal{P} + \alpha \delta_\Sigma$. First, we introduce $A_\alpha$ in Section~\ref{section_delta_op_analysis} as an operator in $L^2(\mathbb{R}^n)$ and show its self-adjointness; in this procedure we also obtain in Proposition~\ref{proposition_Birman_Schwinger_delta} the Birman-Schwinger principle to characterize the discrete eigenvalues of $A_\alpha$ via boundary integral equations. Then, in Section~\ref{SubSection:ApproxDelta} we discuss how these boundary integral equations can be solved numerically by boundary element methods. Finally, in Section~\ref{SubSection:NumExDelta} we show some numerical examples.

\subsection{Definition and self-adjointness of $A_\alpha$} \label{section_delta_op_analysis}

As usual, $\Omega_\text{i} \subset \mathbb{R}^n$ is a bounded Lipschitz domain with boundary $\Sigma := \partial \Omega_\text{i}$, $\Omega_\text{e} := \mathbb{R}^n \setminus \overline{\Omega_\text{i}}$, and $\nu$ denotes the unit normal to $\Omega_\text{i}$. Recall the definition of the elliptic partial differential expression $\mathcal{P}$ from~\eqref{def_P}, the Sobolev space $H^1_\mathcal{P}(\Omega_{\text{i}/\text{e}})$ from~\eqref{def_H_P}, and the weak conormal derivative $\mathcal{B}_\nu$ from~\eqref{def_B_nu} and~\eqref{conormal_derivative_extension}. For a real valued function $\alpha \in L^\infty(\Sigma)$ we define in $L^2(\mathbb{R}^n)$ the partial differential operator $A_\alpha$ by
\begin{equation} \label{def_A_alpha}
  \begin{split}
    A_\alpha f &:= \mathcal{P} f_\text{i} \oplus \mathcal{P} f_\text{e}, \\
    \dom A_\alpha &:= \big\{ f = f_\text{i} \oplus f_\text{e} \in H^1_\mathcal{P}(\Omega_\text{i}) \oplus H^1_\mathcal{P}(\Omega_\text{e}): \gamma f_\text{i} = \gamma f_\text{e}, \, \mathcal{B}_\nu f_\text{e} - \mathcal{B}_\nu f_\text{i} = \alpha \gamma f \big\}.
  \end{split}
\end{equation}
With the help of~\eqref{Green_extended} it is not difficult to show that $A_\alpha$ is symmetric in $L^2(\mathbb{R}^n)$:

\begin{lem} \label{lemma_symmetric_delta}
  Let $\alpha \in L^\infty(\Sigma)$ be real valued. Then the operator $A_\alpha$ defined by~\eqref{def_A_alpha} is symmetric in $L^2(\mathbb{R}^n)$.
\end{lem}
\begin{proof}
  We show that $(A_\alpha f, f)_{L^2(\mathbb{R}^n)} \in \mathbb{R}$ for all $f \in \dom A_\alpha$. Let $f \in \dom A_\alpha$ be fixed. 
  Using~\eqref{Green_extended} in $\Omega_\text{i}$ and $\Omega_\text{e}$ and that the normal $\nu$ is pointing outside of $\Omega_\text{i}$ and inside of $\Omega_\text{e}$ we get
  \begin{equation*}
    \begin{split}
      (A_\alpha f, f)_{L^2(\mathbb{R}^n)} &= (\mathcal{P} f_\text{i}, f_\text{i})_{L^2(\Omega_\text{i})} + (\mathcal{P} f_\text{e}, f_\text{e})_{L^2(\Omega_\text{e})} \\
      &= \Phi_{\Omega_\text{i}}[f_\text{i}, f_\text{i}] - (\mathcal{B}_\nu f_\text{i}, \gamma f_\text{i})  + \Phi_{\Omega_\text{e}}[f_\text{e}, f_\text{e}] + ( \mathcal{B}_\nu f_\text{e}, \gamma f_\text{e}).
    \end{split}
  \end{equation*}
  Since $f \in \dom A_\alpha$ we have $\gamma f_\text{i} = \gamma f_\text{e}$. This implies, in particular, $f \in H^1(\mathbb{R}^n)$ and hence $\Phi_{\Omega_\text{i}}[f_\text{i}, f_\text{i}] + \Phi_{\Omega_\text{e}}[f_\text{i}, f_\text{e}] = \Phi_{\mathbb{R}^n}[f, f]$. With the help of the transmission condition for $f \in \dom A_\alpha$ along $\Sigma$ we conclude
  \begin{equation*}
    \begin{split}
      (A_\alpha f, f)_{L^2(\mathbb{R}^n)} &= \Phi_{\mathbb{R}^n}[f, f] + (\mathcal{B}_\nu f_\text{e} - \mathcal{B}_\nu f_\text{i}, \gamma f ) 
      =\Phi_{\mathbb{R}^n}[f, f] + (\alpha \gamma f, \gamma f).
    \end{split}
  \end{equation*}
  Since the sesquilinear form $\Phi_{\mathbb{R}^n}$ is symmetric and $\alpha$ is real valued, the latter number is real and therefore, the claim is shown.
\end{proof}

In the following proposition we show how the discrete eigenvalues of $A_\alpha$ can be characterized with the help of boundary integral operators. First, we determine the eigenfunctions in $\ker(A_\alpha - \lambda) \ominus \ker(A_0-\lambda)$ with the Birman-Schwinger principle for $A_\alpha$, where the linear eigenvalue problem for the unbounded partial differential operator $A_\alpha$ is translated to the nonlinear eigenvalue problem for a family of boundary integral operators which are related to the single layer boundary integral operator $\mathcal{S}(\lambda)$. The eigenfunctions of $A_\alpha$ in $\ker(A_\alpha - \lambda) \cap \ker(A_0-\lambda)$ are characterized with the help of Theorem~\ref{theorem_spectrum_A_0}. To formulate the result below recall for $\lambda \in \rho(A_0) \cup \sigma_\text{disc}(A_0)$ the definition of the single layer potential $\text{SL}(\lambda)$ from~\eqref{def_single_layer_potential},  the set $\mathcal{M}_\lambda$ from~\eqref{def_M_lambda},  the single layer boundary integral operator $\mathcal{S}(\lambda)$ from~\eqref{def_single_layer_bdd_int_op}, $\mathcal{S}_0(\lambda) := \mathcal{S}(\lambda) \upharpoonright L^2(\Sigma)$, and $R_{\mathcal{A}(\lambda_0)}$ from~\eqref{def_residual}. The following result allows us later in Section~\ref{SubSection:ApproxDelta} to apply boundary element methods to compute all discrete eigenvalues of $A_\alpha$ numerically.

\begin{prop} \label{proposition_Birman_Schwinger_delta}
  Let $\alpha \in L^\infty(\Sigma)$ be real valued and let $A_\alpha$ be defined by~\eqref{def_A_alpha}. Then the following is true for any $\lambda \in \rho(A_0) \cup \sigma_\textup{disc}(A_0)$:
  \begin{itemize}
    \item[(i)] $\ker(A_\alpha - \lambda) \ominus \ker(A_0-\lambda) \neq \{ 0 \}$ if and only if there exists $0 \neq \varphi \in \mathcal{M}_\lambda \cap L^2(\Sigma)$ such that $(I + \alpha \mathcal{S}_0(\lambda)) \varphi = 0$. Moreover,
    \begin{equation} \label{kernel_Birman_Schwinger_delta}
      \ker (A_\alpha - \lambda) \ominus \ker(A_0-\lambda) = \big\{ \textup{SL}(\lambda) \varphi: \varphi \in \mathcal{M}_\lambda \cap L^2(\Sigma), (I + \alpha \mathcal{S}_0(\lambda)) \varphi  =0 \big\}.
    \end{equation}
    \item[(ii)] If $\lambda \in \rho(A_0)$, then $\lambda \in \sigma_\textup{p}(A_\alpha)$ if and only if $-1 \in \sigma_\textup{p}(\alpha \mathcal{S}_0(\lambda))$.
    \item[(iii)] $\ker(A_\alpha - \lambda) \cap \ker(A_0-\lambda) \neq \{ 0 \}$ if and only if there exists $(\varphi, \psi)^\top \in \ran R_{\mathcal{A}(\lambda_0)}$ such that $\alpha \varphi = 0$.
    \item[(iv)] If $\lambda \notin \sigma_\textup{p}(A_\alpha) \cup \sigma(A_0)$, then $I + \alpha \mathcal{S}_0(\lambda)$ admits a bounded and everywhere defined inverse in $L^2(\Sigma)$.
  \end{itemize}
\end{prop}
\begin{proof}
  (i) Assume first that $\ker(A_\alpha - \lambda) \ominus \ker(A_0-\lambda) \neq \{ 0 \}$ and let $f \in \ker(A_\alpha - \lambda) \ominus \ker(A_0-\lambda)$. Then by Lemma~\ref{lemma_single_layer_potential}~(i) there exists $\varphi \in \mathcal{M}_\lambda$ such that $f = \text{SL}(\lambda) \varphi$. Since $f \in \dom A_\alpha$ one has with Lemma~\ref{lemma_single_layer_potential}~(ii)
  \begin{equation*}
    \alpha \gamma f = \mathcal{B}_\nu f_\text{e} - \mathcal{B}_\nu f_\text{i} = \mathcal{B}_\nu (\text{SL}(\lambda) \varphi)_\text{e} - \mathcal{B}_\nu (\text{SL}(\lambda) \varphi)_\text{i} = -\varphi.
  \end{equation*}
  In particular, we deduce $\varphi \in L^2(\Sigma)$ and with $\gamma f = \mathcal{S}(\lambda) \varphi = \mathcal{S}_0(\lambda) \varphi$ this can be rewritten as $-\varphi = \alpha \mathcal{S}_0(\lambda) \varphi$. Moreover, the above considerations show
  \begin{equation} \label{kernel_Birman_Schwinger_delta1}
    \ker(A_\alpha - \lambda) \ominus \ker(A_0-\lambda) \subset \big\{ \textup{SL}(\lambda) \varphi: \varphi \in \mathcal{M}_\lambda, (I + \alpha \mathcal{S}_0(\lambda)) \varphi  =0 \big\}.
  \end{equation}
  
  Conversely, assume that there exists $0 \neq \varphi \in \mathcal{M}_\lambda \cap L^2(\Sigma)$ such that $(I + \alpha \mathcal{S}_0(\lambda)) \varphi = 0$. Then $f := \text{SL}(\lambda) \varphi \in H^1_\mathcal{P}(\mathbb{R}^n \setminus \Sigma) \cap H^1(\mathbb{R}^n)$ and it follows from Lemma~\ref{lemma_single_layer_potential}~(ii) that $f$ is nontrivial. Using the jump properties of $\text{SL}(\lambda) \varphi$ from Lemma~\ref{lemma_single_layer_potential}~(ii) we conclude further 
  \begin{equation*}
    \mathcal{B}_\nu f_\text{e} - \mathcal{B}_\nu f_\text{i} = - \varphi = \alpha \mathcal{S}_0(\lambda) \varphi = \alpha \gamma f,
  \end{equation*}
  where it was used that $\varphi$ belongs to the kernel of $I + \alpha \mathcal{S}_0(\lambda)$. Hence, $f \in \dom A_\alpha$. With Lemma~\ref{lemma_single_layer_potential}~(i) we conclude, as $\varphi \in \mathcal{M}_\lambda$, that
  \begin{equation*}
    (A_\alpha - \lambda) f = (\mathcal{P} - \lambda) (\text{SL}(\lambda) \varphi)_\text{i} \oplus (\mathcal{P} - \lambda) (\text{SL}(\lambda) \varphi)_\text{e}  =0,
  \end{equation*}
  which shows $\lambda \in \sigma_\textup{p}(A_\alpha)$ and
  \begin{equation} \label{kernel_Birman_Schwinger_delta2}
    \big\{ \textup{SL}(\lambda) \varphi: \varphi \in \mathcal{M}_\lambda, (I + \alpha \mathcal{S}_0(\lambda)) \varphi  =0 \big\} \subset \ker(A_\alpha - \lambda).
  \end{equation}
  Note that~\eqref{kernel_Birman_Schwinger_delta1} and~\eqref{kernel_Birman_Schwinger_delta2} give~\eqref{kernel_Birman_Schwinger_delta}. Hence, all claims in item~(i) are proved.
  
  Assertion~(ii) is a simple consequence of item~(i), as for $\lambda \notin \sigma(A_0)$ we have $\ker(A_0-\lambda)=\{ 0 \}$ and $\mathcal{M}_\lambda = H^{-1/2}(\Sigma)$.
  
  Statement~(iii) follows from Theorem~\ref{theorem_spectrum_A_0}. Note that $f \in \dom A_\alpha \cap \dom A_0$ if and only if $f \in H^2(\mathbb{R}^n)$ and $\alpha \gamma f = \mathcal{B}_\nu f_\text{e} - \mathcal{B}_\nu f_\text{i} = 0$. With Theorem~\ref{theorem_spectrum_A_0} it follows that $f \in \ker(A_\alpha - \lambda) \cap \ker(A_0 - \lambda)$ if and only if there exists $(\varphi, \psi)^\top = (\gamma f, \mathcal{B}_\nu f)^\top \in \ran R_{\mathcal{A}(\lambda)}$ such that $\alpha \varphi = 0$.
  
  (iv) Since $\mathcal{S}_0(\lambda)$ is compact in $L^2(\Sigma)$ by Lemma~\ref{lemma_single_layer_bdd_int_op}~(i), it follows  from Fredholm's alternative that $I + \alpha \mathcal{S}_0(\lambda)$ is bijective in $L^2(\Sigma)$ and admits a bounded inverse, if $0 \notin \sigma_\textup{p}(I + \alpha \mathcal{S}_0(\lambda))$. According to item~(ii) this is fulfilled, if $\lambda \notin \sigma_\textup{p}(A_\alpha) \cup \sigma(A_0)$.
\end{proof}

Now we are prepared to show the self-adjointness of the operator $A_\alpha$. In the proof of this result we show also a Krein type resolvent formula, which allows us to verify that the essential spectrum of $A_\alpha$ coincides with the essential spectrum of the unperturbed operator $A_0$. We remark that the resolvent formula in~\eqref{krein_delta} is well defined, as $I + \alpha \mathcal{S}_0(\lambda)$ is boundedly invertible in $L^2(\Sigma)$ for $\lambda \in \rho(A_0) \cap \rho(A_\alpha)$ by Proposition~\ref{proposition_Birman_Schwinger_delta}~(iv).

\begin{prop} \label{proposition_self_adjoint_delta}
  Let $\alpha \in L^\infty(\Sigma)$ be real valued, let the operators $A_0$, $\textup{SL}(\lambda)$, and $\mathcal{S}(\lambda)$, $\lambda \in \rho(A_0)$, be given by~\eqref{def_free_Op}, \eqref{def_single_layer_potential}, and~\eqref{def_single_layer_bdd_int_op}, respectively, and let $\mathcal{S}_0(\lambda) = \mathcal{S}(\lambda) \upharpoonright L^2(\Sigma)$. Then the operator $A_\alpha$ defined by~\eqref{def_A_alpha} is self-adjoint in $L^2(\mathbb{R}^n)$ and the following is true:
  \begin{itemize}
    \item[(i)] For $\lambda \in \rho(A_0) \cap \rho(A_\alpha)$ the resolvent of $A_\alpha$ is given by
    \begin{equation} \label{krein_delta}
      (A_\alpha - \lambda)^{-1} = (A_0 - \lambda)^{-1} - \textup{SL}(\lambda) \big( I + \alpha \mathcal{S}_0(\lambda) \big)^{-1} \alpha \gamma(A_0 - \lambda)^{-1}.
    \end{equation}
    \item[(ii)] $\sigma_\textup{ess}(A_\alpha) = \sigma_\textup{ess}(A_0)$.
  \end{itemize}
\end{prop}
\begin{proof}
  In order to prove that $A_\alpha$ is self-adjoint, we show that $\ran (A_\alpha - \lambda) = L^2(\mathbb{R}^n)$ for $\lambda \in \mathbb{C} \setminus (\sigma(A_0) \cup \sigma_\textup{p}(A_\alpha))$. Let $f \in L^2(\mathbb{R}^n)$ be fixed and define
  \begin{equation*} 
    g := (A_0 - \lambda)^{-1} f - \textup{SL}(\lambda) \big( I + \alpha \mathcal{S}_0(\lambda) \big)^{-1} \alpha \gamma(A_0 - \lambda)^{-1} f.
  \end{equation*}
  Note that $g$ is well defined, as $I + \alpha \mathcal{S}_0(\lambda)$ admits a bounded inverse in $L^2(\Sigma)$ for $\lambda \notin \sigma(A_0) \cup \sigma_\textup{p}(A_\alpha)$ by Proposition~\ref{proposition_Birman_Schwinger_delta}~(iv).
  We are going to show that $g \in \dom A_\alpha$ and $(A_\alpha - \lambda) g = f$. This shows then $\ran (A_\alpha - \lambda) = L^2(\mathbb{R}^n)$ and also~\eqref{krein_delta}.
  
  Since $(A_0 - \lambda)^{-1} f \in H^2(\mathbb{R}^n)$ by Proposition~\ref{proposition_resolvent}, we conclude $\gamma (A_0 - \lambda)^{-1} f \in L^2(\Sigma)$ and further from Proposition~\ref{proposition_Birman_Schwinger_delta}~(ii) and Lemma~\ref{lemma_single_layer_potential} that
  \begin{equation*}
    \textup{SL}(\lambda) \big( I + \alpha \mathcal{S}_0(\lambda) \big)^{-1} \alpha \gamma(A_0 - \lambda)^{-1} f \in H^1_\mathcal{P}(\mathbb{R}^n \setminus \Sigma) \cap H^1(\mathbb{R}^n).
  \end{equation*}
  Therefore, also $g \in H^1_\mathcal{P}(\mathbb{R}^n \setminus \Sigma) \cap H^1(\mathbb{R}^n)$. Moreover, we have by Lemma~\ref{lemma_single_layer_potential}~(ii)
  \begin{equation*}
    \begin{split}
      \mathcal{B}_\nu g_\text{e} - \mathcal{B}_\nu g_\text{i} - \alpha \gamma g &= \big( I + \alpha \mathcal{S}_0(\lambda) \big)^{-1} \alpha \gamma(A_0 - \lambda)^{-1} f - \alpha \gamma (A_0 - \lambda)^{-1} f \\
      &\qquad + \alpha \mathcal{S}_0(\lambda) \big( I + \alpha \mathcal{S}_0(\lambda) \big)^{-1} \alpha \gamma(A_0 - \lambda)^{-1} f = 0,
    \end{split}
  \end{equation*}
  which shows $g \in \dom A_\alpha$. Next, we have with $\varphi := ( I + \alpha \mathcal{S}_0(\lambda) )^{-1} \alpha \gamma(A_0 - \lambda)^{-1} f$
  \begin{equation*}
    \begin{split}
      (A_\alpha - \lambda) g &= (\mathcal{P} - \lambda) (A_0-\lambda)^{-1} f - (\mathcal{P}-\lambda) (\text{SL}(\lambda) \varphi)_\text{i} \oplus (\mathcal{P}-\lambda) (\text{SL}(\lambda) \varphi)_\text{e} = f,
    \end{split}
  \end{equation*}
  where~\eqref{range_single_layer_potential} for $\lambda \in \rho(A_0)$ was used in the last step.
  With the previous considerations we deduce now the self-adjointness of $A_\alpha$ and~\eqref{krein_delta}.
  
  It remains to show assertion~(ii). Let $\lambda \in \mathbb{C} \setminus \mathbb{R}$ be fixed. First, due to the mapping properties of the resolvent of $A_0$ from Proposition~\ref{proposition_resolvent} and the mapping properties of $\gamma$ from~\eqref{Dirichlet_trace} the operator
  \begin{equation*}
    \gamma (A_0-\lambda)^{-1}: L^2(\mathbb{R}^n) \rightarrow H^{3/2}(\Sigma) \hookrightarrow H^{1/2}(\Sigma)
  \end{equation*}
  is bounded. Since $H^{1/2}(\Sigma)$ is compactly embedded in $L^2(\Sigma)$ this and Proposition~\ref{proposition_Birman_Schwinger_delta}~(iv) yield that
  \begin{equation*}
    \big( I + \alpha \mathcal{S}_0(\lambda) \big)^{-1} \alpha \gamma(A_0 - \lambda)^{-1}: L^2(\mathbb{R}^n) \rightarrow L^2(\Sigma)
  \end{equation*}
  is compact. As $L^2(\Sigma)$ is boundedly embedded in $H^{-1/2}(\Sigma)$ and $\text{SL}(\lambda): H^{-1/2}(\Sigma) \rightarrow L^2(\mathbb{R}^n)$ is bounded, we conclude that
  \begin{equation*}
    (A_\alpha - \lambda)^{-1} - (A_0 - \lambda)^{-1} = - \textup{SL}(\lambda) \big( I + \alpha \mathcal{S}_0(\lambda) \big)^{-1} \alpha \gamma(A_0 - \lambda)^{-1}
  \end{equation*}
  is compact in $L^2(\mathbb{R}^n)$. Therefore, with the Weyl theorem we get $\sigma_\text{ess}(A_\alpha) = \sigma_\text{ess}(A_0)$.
\end{proof}

By combining the results from Proposition~\ref{proposition_Birman_Schwinger_delta} and Proposition~\ref{proposition_self_adjoint_delta} we can prove now the following proposition about the inverse of the Birman-Schwinger operator $I + \alpha \mathcal{S}_0(\lambda)$, which will be of great importance for the numerical calculation of the discrete eigenvalues of $A_\alpha$ via boundary element methods.

\begin{prop} \label{proposition_Birman_Schwinger_kernel_delta}
  Let $\alpha \in L^\infty(\Sigma)$ be real valued and let $A_\alpha$ be defined by~\eqref{def_A_alpha}. Then the map 
  \begin{equation*}
    \rho(A_\alpha) \cap \rho(A_0) \ni \lambda \mapsto \big(I + \alpha \mathcal{S}_0(\lambda) \big)^{-1}
  \end{equation*}
  can be extended to a holomorphic operator-valued function, which is holomorphic in $\rho(A_\alpha)$ with respect to the toplogy in $\mathcal{B}(L^2(\Sigma))$. Moreover, for $\lambda_0 \notin \sigma_\textup{ess}(A_\alpha) = \sigma_\textup{ess}(A_0)$ one has $\ker(A_\alpha - \lambda_0) \ominus \ker(A_0-\lambda_0) \neq \{ 0 \}$ if and only if $(I + \alpha \mathcal{S}_0(\lambda) )^{-1}$ has a pole at $\lambda_0$ and
  \begin{equation} \label{kernel_inverse_delta}
    \ker (A_\alpha - \lambda_0) \ominus \ker (A_0-\lambda_0) = \big\{ \textup{SL}(\lambda_0) \varphi: \lim_{\lambda \rightarrow \lambda_0} (\lambda - \lambda_0) (I + \alpha \mathcal{S}_0(\lambda))^{-1} \varphi \neq 0 \big\}.
  \end{equation}
\end{prop}
\begin{proof}
  The proof is split into 4 steps.
  
  {\it Step 1:} Define the map 
  \begin{equation*}
    [\mathcal{B}_\nu]_\Sigma: H^1_\mathcal{P}(\mathbb{R}^n \setminus \Sigma) \rightarrow H^{-1/2}(\Sigma), \quad
    [\mathcal{B}_\nu]_\Sigma f := \mathcal{B}_\nu f_\text{i} - \mathcal{B}_\nu f_\text{e},
  \end{equation*}
  and let $\lambda \in \rho(A_\alpha) \cap \rho(A_0)$ be fixed. We show that
  \begin{equation} \label{Birman_Schwinger_inverse_delta}
    \big(I + \alpha \mathcal{S}(\lambda) \big)^{-1} = [\mathcal{B}_\nu]_\Sigma (A_\alpha - \lambda)^{-1} \gamma^*.
  \end{equation}
Note that $(I+\alpha\mathcal{S}(\lambda))^{-1}$ is well defined by the same reasons as in Proposition~\ref{proposition_Birman_Schwinger_delta}~(iv), as $\alpha\mathcal{S}(\lambda)\in\mathcal{B}(H^{-1/2}(\Sigma),L^2(\Sigma))$ is compact in $H^{-1/2}(\Sigma)$.   In particular, this implies that $[\mathcal{B}_\nu]_\Sigma (A_\alpha - \lambda)^{-1} \gamma^* \in \mathcal{B}(H^{-1/2}(\Sigma))$. To show~\eqref{Birman_Schwinger_inverse_delta} we note first that~\eqref{krein_delta} implies
  \begin{equation*}
    \gamma(A_\alpha - \overline{\lambda})^{-1} = \gamma(A_0-\overline{\lambda})^{-1} - \mathcal{S}(\overline{\lambda}) \big(I + \alpha \mathcal{S}(\overline{\lambda}) \big)^{-1} \gamma (A_0-\overline{\lambda})^{-1},
  \end{equation*}
  which implies, after taking the dual,
  \begin{equation*}
    (A_\alpha - \lambda)^{-1} \gamma^* = \text{SL}(\lambda) - \text{SL}(\lambda) \alpha \big(I + \mathcal{S}(\lambda) \alpha  \big)^{-1} \mathcal{S}(\lambda).
  \end{equation*}
  Using
  \begin{equation*}
    \begin{split}
    \alpha &\big(I + \mathcal{S}(\lambda) \alpha  \big)^{-1} - \big(I + \alpha \mathcal{S}(\lambda) \big)^{-1} \alpha
        \\
    &= \big(I + \alpha \mathcal{S}(\lambda) \big)^{-1} \big[ \big(I + \alpha \mathcal{S}(\lambda) \big) \alpha - \alpha \big(I + \mathcal{S}(\lambda) \alpha  \big) \big] \big(I + \mathcal{S}(\lambda) \alpha  \big)^{-1} = 0,
    \end{split}
  \end{equation*}
  we can simplify the last expression to
  \begin{equation*}
    \begin{split}
      (A_\alpha - \lambda)^{-1} \gamma^* &= \text{SL}(\lambda) - \text{SL}(\lambda) \big(I + \alpha \mathcal{S}(\lambda) \big)^{-1} \alpha \mathcal{S}(\lambda) \\
      &= \text{SL}(\lambda) \big(I + \alpha \mathcal{S}(\lambda) \big)^{-1}  \big[ \big(I + \alpha \mathcal{S}(\lambda) \big) - \alpha \mathcal{S}(\lambda) \big] \\
      &= \text{SL}(\lambda) \big(I + \alpha \mathcal{S}(\lambda) \big)^{-1}.
    \end{split}
  \end{equation*}
  In particular, by Lemma~\ref{lemma_single_layer_potential} the right hand side belongs to $\mathcal{B}(H^{-1/2}(\Sigma), H^1_\mathcal{P}(\mathbb{R}^n \setminus \Sigma))$ and thus, the same must be true for $(A_\alpha - \lambda)^{-1} \gamma^*$. Therefore, we are allowed to apply $[\mathcal{B}_\nu]_\Sigma$ and the last formula shows, with the help of Lemma~\ref{lemma_single_layer_potential}~(ii), the relation~\eqref{Birman_Schwinger_inverse_delta}.

  {\it Step 2:} We show that $[\mathcal{B}_\nu]_\Sigma (A_\alpha - \lambda)^{-1} \gamma^* \in \mathcal{B}(H^{-1/2}(\Sigma))$ for any $\lambda \in \rho(A_\alpha)$ and that $\rho(A_\alpha) \ni \lambda \mapsto [\mathcal{B}_\nu]_\Sigma (A_\alpha - \lambda)^{-1} \gamma^*$ is holomorphic in $\mathcal{B}(H^{-1/2}(\Sigma))$.
  
  First, we note that $\dom A_\alpha \subset H^1(\mathbb{R}^n) \cap H^1_\mathcal{P}(\mathbb{R}^n \setminus \Sigma)$ implies that
  \begin{equation*}
    (A_\alpha - \lambda)^{-1} \in \mathcal{B}(L^2(\mathbb{R}^n), H^1(\mathbb{R}^n)) \quad \text{and} \quad 
    (A_\alpha - \lambda)^{-1} \in \mathcal{B}(L^2(\mathbb{R}^n), H^1_\mathcal{P}(\mathbb{R}^n \setminus \Sigma)),
  \end{equation*}
  see~\eqref{mapping_properties_resolvent} for a similar argument. Hence, by duality also 
  \begin{equation*}
    (A_\alpha - \lambda)^{-1} \in \mathcal{B}(H^{-1}(\mathbb{R}^n), L^2(\mathbb{R}^n)).
  \end{equation*}
  With the resolvent identity this implies for any $\lambda_0 \in \rho(A_\alpha)$ and $\lambda \in \rho(A_\alpha) \cap \rho(A_0)$, in a similar way as in the proof of Proposition~\ref{proposition_resolvent}, first that
  \begin{equation*}
    \begin{split}
      [\mathcal{B}_\nu]_\Sigma (A_\alpha - \lambda_0)^{-1} \gamma^* &- [\mathcal{B}_\nu]_\Sigma (A_\alpha - \lambda)^{-1} \gamma^* \\
      &= (\lambda_0 - \lambda) [\mathcal{B}_\nu]_\Sigma (A_\alpha - \lambda_0)^{-1} (A_\alpha - \lambda)^{-1} \gamma^*,
    \end{split}
  \end{equation*}
  which yields first with~\eqref{Birman_Schwinger_inverse_delta} that $[\mathcal{B}_\nu]_\Sigma (A_\alpha - \lambda_0)^{-1} \gamma^* \in \mathcal{B}(H^{-1/2}(\Sigma))$ and in a second step, that $[\mathcal{B}_\nu]_\Sigma (A_\alpha - \lambda_0)^{-1} \gamma^*$ is holomorphic in $\mathcal{B}(H^{-1/2}(\Sigma))$, which shows the claim of this step.

  {\it Step 3:} With the help of~\eqref{Birman_Schwinger_inverse_delta} and the result from {\it Step 2} we know that $(I + \alpha \mathcal{S}(\lambda))^{-1}$ can be extended to a holomorphic map in $\mathcal{B}(H^{-1/2}(\Sigma))$ for $\lambda \in \rho(A_\alpha)$. By duality we deduce that $(I + \alpha \mathcal{S}(\lambda))^{-1}$ is holomorphic in $\mathcal{B}(H^{1/2}(\Sigma))$ for $\lambda \in \rho(A_\alpha)$. Finally, by interpolation we conclude that $(I + \alpha \mathcal{S}(\lambda))^{-1}$ is also holomorphic in $\mathcal{B}(L^2(\Sigma))$ for $\lambda \in \rho(A_\alpha)$.
  
  {\it Step 4:} Finally, it follows from Proposition~\ref{proposition_Birman_Schwinger_delta}~(i) that $\ker(A_\alpha - \lambda_0) \ominus \ker(A_0-\lambda_0) \neq \{ 0 \}$ if and only if there exists $\varphi \in \mathcal{M}_{\lambda_0}$ such that $(I + \alpha \mathcal{S}_0(\lambda_0)) \varphi = 0$, i.e. if and only if $\lambda \mapsto (I + \alpha \mathcal{S}_0(\lambda))^{-1}$ has a pole at $\lambda_0$. This shows immediately \eqref{kernel_inverse_delta}.
\end{proof}

\newcommand{\triang}{\mathcal{T}}

\subsection{Numerical approximation of discrete eigenvalues of $A_\alpha$}\label{SubSection:ApproxDelta}
For the numerical approximation of the discrete eigenvalues of $A_\alpha$ and the corresponding eigenfunctions we consider boundary element methods. 
These require the knowledge of an explicit integral representation of the paramatrix ${\mathcal G}(\lambda)$ of ${\mathcal P}-\lambda$ or at least a good approximation of the boundary integral operator $\mathcal{S}_0(\lambda)$.
This is for example the case when ${\mathcal P}$ has constant coefficients.
We restrict ourselves to three-dimensional domains $\Omega_\text{i} \subset \mathbb{R}^3$ in order to keep the presentation simple.  The presented procedure  and the obtained convergence results can be straightforwardly transfered  to domains with general space dimensions. The discrete eigenvalues of $A_\alpha$ split into the eigenvalues of the nonlinear eigenvalue problem 
\begin{equation}\label{Eq:EigContinuous}
 (I + \alpha \mathcal{S}_0(\lambda)) \varphi = 0
\end{equation}
in $\rho(A_0)$ and into distinct discrete eigenvalues of $A_0$, which can be characterized on the one hand as poles of the operator-valued functions ${\mathcal A}(\cdot)$ having the property specified in Proposition~\ref{proposition_Birman_Schwinger_delta}~(iii) and  on the other hand as poles of $[I+\alpha{\mathcal S}_0(\cdot)]^{-1}$ lying in $\sigma_\text{disc}(A_0)$.

In the following we will first consider the case that there are no discrete eigenvalues of $A_0$, that means that all discrete eigenvalues can be characterized as eigenvalues of the nonlinear eigenvalue problem~\eqref{Eq:EigContinuous}. This is for example the case when ${\mathcal P}$ has constant or periodic coefficients.  Afterwards the general case will be treated. 
For both cases we will present  convergence results of the boundary element approximations of the discrete eigenvalues of $A_\alpha$. 
In the first situation a complete numerical analysis is provided, whereas in the general case for the approximation of the eigenvalues in $\sigma_\text{disc}(A_0)$ the convergence theory of Section~\ref{SectionGalerkinApprox} can not be applied.   
In addition  we will  address the numerical solution of the discretized problems which  results in the determination of the poles of  matrix-valued functions.  For that the so-called contour integral method is suggested~\cite{Beyn} which is a reliable method for finding all poles of a meromorphic matrix-valued function  inside a given contour in the complex plane.      

\subsubsection{Approximation of discrete eigenvalues of $A_\alpha$ for the case $\sigma_\textnormal{disc}(A_0)=\varnothing$}\label{SubsubsectionApproxDeltaEinfach} If $\sigma_\text{disc}(A_0)=\varnothing$, then, by  Proposition~\ref{proposition_Birman_Schwinger_delta}~(ii), $\lambda_0\in\C\setminus\sigma_\text{ess}(A_\alpha)$ is a discrete eigenvalue  of $A_\alpha$ if and only if it is an eigenvalue of the nonlinear eigenvalue problem~\eqref{Eq:EigContinuous}.  Any conforming Galerkin method for the approximation of the eigenvalue problem~\eqref{Eq:EigContinuous} is according to the abstract results in Section~\ref{SectionGalerkinApprox} a convergent method since $\rho(A_0)\ni\lambda\mapsto(I + \alpha \mathcal{S}_0(\lambda))$ is by Lemma~\ref{lemma_single_layer_bdd_int_op}~(iii) holomorphic in ${\mathcal B}(L^2(\Sigma))$ and $(I + \alpha \mathcal{S}_0(\lambda))$ satisfies for $\lambda\in\rho(A_0)$   G\r{a}rding's  inequality of the form~\eqref{Inequality:Garding} because $\alpha \mathcal{S}_0(\lambda):L^2(\Sigma)\rightarrow L^2(\Sigma)$ is  compact, see Lemma~\ref{lemma_single_layer_bdd_int_op}~(i).

For the presentation of the boundary element method for the approximation of the discrete eigenvalues of $A_\alpha$  we want to consider first the case that $\Omega_\text{i}\subset\mathbb{R}^3$ is a bounded polyhedral Lipschitz domain. The general case is commented in Remark~\ref{Remark:GeneralLipschitzDomain}.  Let $(\triang_{N})_{N\in\N}$ be a sequence of quasi-uniform triangulations of the boundary $\Sigma$ of $\Omega_\text{i}$, see e.~g.~\cite[Chapter~4]{SauterSchwab} or \cite[Chapter~10]{Steinbach}, such that 
\begin{equation}\label{Eq:PropertiesOfTriangulations}
 \triang_N =\{\tau^N_1,\ldots,\tau^N_{n(N)}\}\quad\text{and}\quad \Sigma=\bigcup_{j=1}^{n(N)}\tau_j^N,
\end{equation}
where we assume that for the mesh-sizes $h(N)$ of the triangulations $\triang_N$ the relation $h(N)\rightarrow 0$ holds as $N\rightarrow\infty$. 
We choose the spaces of piecewise constant functions $S_0(\triang_N)$ with respect to the triangulations $\triang_N$ as spaces for the approximations of eigenfunctions of the eigenvalue problem~\eqref{Eq:EigContinuous}. For a finite-dimensional subspace $V\subset H^{s}(\Sigma)$, $s\in[0,1]$, we have the following approximation property of $S_0(\triang_N)$ with respect to $\|\cdot\|_{L^2(\Sigma)}$ \cite[Thm.~10.1]{Steinbach}:
\begin{equation}\label{Eq:ApproxPiecewiseConstant}
 \delta_{L^{2}(\Sigma)}(V, S_0(\triang_N))
=\sup_{\substack{v\in V\\
\|v\|_{L^2(\Sigma)}= 1}} \inf_{\varphi_N\in S_0(\triang_N)}\|v-\varphi_N\|_{L^2(\Sigma)}
 =\mathcal{O}(h(N)^{s}).
\end{equation}

The Galerkin approximation of the eigenvalue problem~\eqref{Eq:EigContinuous} reads as: find eigenpairs $(\lambda_N,\varphi_N)\in\C\times S_0(\triang_N)\setminus\{0\}$ such that
\begin{equation}\label{Eq:GalerkinEVPDelta}
\left((I + \alpha \mathcal{S}_0(\lambda_N)) \varphi_N,\psi_N\right) = 0\qquad\forall \psi_N\in S_0(\triang_N).
\end{equation}
All abstract convergence results from Theorem~\ref{Theorm:AbstrctConvergenceResults} can be applied to the approximation of the eigenvalue problem~\eqref{Eq:EigContinuous} by the Galerkin eigenvalue problem~\eqref{Eq:GalerkinEVPDelta}. 
In the following theorem we only state the asymptotic convergence order of the approximations of the eigenvalues and the corresponding eigenfunctions. 

\begin{thm}\label{Thm:ConvergenceDeltaInRhoAzero}
Let $D\subset\rho(A_0)$ be a compact  and    connected set in $\C$  such that $\partial D$ is a simple 
rectifiable curve. Suppose that $\lambda\in
\mathring{D}$ is the only eigenvalue of $I+\alpha{\mathcal S}_0(\cdot)$ in $D$ and that 
$\ker(I+\alpha{\mathcal S}_0(\lambda))\subset H^{s}(\Sigma)$ for some $s\in(0,1]$.
Then
there exist an
$N_0\in\N$ and a constant  $c>0$ such that for all $N\geq N_0$  we
have:
\begin{itemize}
\item[(i)] For all eigenvalues $\lambda_N$ of the Galerkin eigenvalue problem~\eqref{Eq:GalerkinEVPDelta} in $D$ 
\begin{equation}\label{Eq:ErrorApproxBoundary}
 \vert \lambda-\lambda_N\vert \leq
c (h(N))^{1+s}
\end{equation}
holds.
\item[(ii)] If  $(\lambda_N,u_N)$ is an eigenpair of 
\eqref{Eq:GalerkinEVPDelta} with $\lambda_N\in D$ and  
$\|\varphi_N\|_{L^2(\Sigma)}=1$, then 
\begin{equation*}
 \inf_{\varphi\in \ker(I+\alpha{\mathcal S}_0(\lambda))}\| \varphi-\varphi_N\|_{L^2(\Sigma)}
 \leq 
 c 
 \left(\vert \lambda_N-\lambda \vert + (h(N))^s \right).
\end{equation*}
\end{itemize}
\end{thm}
\begin{proof}
The error estimates follow from the abstract convergence results in Theorem~\ref{Theorm:AbstrctConvergenceResults}, the approximation property \eqref{Eq:ApproxPiecewiseConstant} of $S_0(\triang_N)$, and the fact, that the eigenfunctions of the adjoint problem are more regular than those of $(I + \alpha \mathcal{S}_0(\cdot))$. To see the last claim, we note that a solution of the adjoint eigenproblem
\begin{equation*}
  (I + \alpha \mathcal{S}_0(\lambda))^*\varphi = (I + \mathcal{S}_0(\lambda)\alpha) \varphi = 0
\end{equation*}
belongs by Lemma~\ref{lemma_single_layer_bdd_int_op}~(i) to $H^1(\Sigma)$ and hence, by~\eqref{Eq:ApproxPiecewiseConstant}
\begin{equation*}
  \delta_{L^2(\Sigma)}\big(\ker((I + \alpha \mathcal{S}_0(\lambda))^*, S_0(\triang_N)\big) \leq c h(N)
\end{equation*}
holds. 
\end{proof}

\begin{remark}\label{Remark:GeneralLipschitzDomain}
If $\Omega$ is a bounded Lipschitz domain with a curved piecewise $C^2$-boundary the approximation of the boundary by a triangulation with flat triangles as described in~\cite[Chapter~8]{SauterSchwab} still guarantees convergence of the approximations of the eigenvalues and eigenfunctions with the same asymptotic convergence order as in Theorem~\ref{Thm:ConvergenceDeltaInRhoAzero}. 
This can be shown by using the results of the discretization of boundary integral operators for approximated boundaries \cite[Chapter~8]{SauterSchwab} and the abstract results of eigenvalue problem approximations \cite{Karma1,Karma2}.   
\end{remark}

The Galerkin eigenvalue problem~\eqref{Eq:GalerkinEVPDelta} results in a nonlinear matrix eigenvalue problem of size $n(N)\times n(N)$, which can be solved by the contour integral method~\cite{Beyn}. 
The contour integral method is 
a reliable method for the 
approximation of all eigenvalues of a holomorphic matrix-valued function $M(\cdot)$ which lie inside of a given contour in the 
complex plane, and for the approximation of  the corresponding  eigenvectors. 
The method is based on  the contour integration of the inverse function $M(\cdot)^{-1}$  and utilizes 
that the eigenvalues of the eigenvalue problem for $M(\cdot)$ are poles of  $M(\cdot)^{-1}$. 
By contour integration of the inverse $M(\cdot)^{-1}$ a reduction of the  holomorphic 
eigenvalue problem for  $M(\cdot)$ to an equivalent
linear eigenvalue problem is possible such that the eigenvalues of the linear 
eigenvalue problem
coincide with the eigenvalues of the nonlinear eigenvalue problem inside the
contour. For details of the implementation of the method we refer to~\cite{Beyn}.

\subsubsection{Approximation of discrete eigenvalues of $A_\alpha$ for the case $\sigma_\textnormal{disc}(A_0)\neq\varnothing$ }
\newcommand{\ms}{\mathcal{S}}
\newcommand{\ma}{\mathcal{A}}
If $\sigma_\text{disc}(A_0)\neq\varnothing$, then Proposition~\ref{proposition_Birman_Schwinger_delta} and Proposition~\ref{proposition_Birman_Schwinger_kernel_delta} show that the discrete eigenvalues of $A_\alpha$ are poles of $[I+\alpha \ms_0(\cdot)]^{-1}$ or the poles of $\ma(\cdot)$ satisfying the property specified in Proposition~\ref{proposition_Birman_Schwinger_delta}~(iii). The boundary element approximation of the discrete eigenvalues of $A_\alpha$ are based on these characterizations. 

First we want to consider the approximation of the poles of $(I+\alpha \ms_0(\cdot))^{-1}$. For  those poles of $(I+\alpha \ms_0(\cdot))^{-1}$ which lie in $\rho(A_0)$ the abstract convergence results of Section~2 can be applied with the same reasoning as in the case $\sigma_\text{disc}(A_0)=\varnothing$, since $(I+\alpha \ms_0(\cdot))$ is holomorphic in $\rho(A_0)$ and the poles of $(I+\alpha \ms_0(\cdot))^{-1}$ in $\rho(A_0)$ coincide with the eigenvalues of the eigenvalue problem for $(I+\alpha \ms_0(\cdot))$ in $\rho(A_0)$. If $\lambda_0$ is a pole of $(I+\alpha \ms_0(\cdot))^{-1}$ which lies in $\sigma_\text{disc}(A_0)$, then $(I+\alpha \ms_0(\cdot))$ is not holomorphic in $\lambda_0$ and therefore the convergence results of Section~2 are not  applicable for the boundary element approximation of $\lambda_0$. To the best of our knowledge a rigorous numerical analysis of the Galerkin approximation of such kind of poles of Fredholm operator-valued functions for which the inverse is not holomorphic at the poles   have not been considered so far in the literature. 
However, we expect similar convergence results also of such kind of poles. If this holds, then this kind of poles of $(I+\alpha \mathcal{S}_0(\lambda))^{-1}$, which is holomorphic in $\rho(A_\alpha)$ by Proposition~\ref{proposition_Birman_Schwinger_kernel_delta}, are appropriately represented as poles of the discretized problem and will be identified by the contour integral method. 

Finally, we want to discuss the approximation of the discrete eigenvalues of $A_\alpha$ which are not poles of $[I+\alpha \ms_0(\cdot)]^{-1}$. If $\lambda_0$ is such an eigenvalue, then, by Proposition~\ref{proposition_Birman_Schwinger_delta}~(iii), it is a pole of $\ma(\cdot)$ such that a pair $(\varphi,\psi)\in\ran R_{\mathcal{A}(\lambda_0)}$ defined by \eqref{def_residual} exists with $\alpha\varphi=0$ or equivalently that $(\lambda,\varphi,\psi)$, $(\varphi,\psi)\neq(0,0)$, satisfies 
\begin{equation}\label{Eq:EigContAinv}
 {\mathcal A}(\lambda)^{-1} 
\begin{pmatrix}
 \psi\\
\varphi 
\end{pmatrix}
  =\begin{pmatrix}
 0\\
0
\end{pmatrix}
\qquad\text{and}
\qquad\alpha\varphi=0.
\end{equation}
The characterization in \eqref{Eq:EigContAinv} can be used for the numerical approximation of the  discrete eigenvalues of $A_\alpha$ which are not poles of $[I+\alpha \ms_0(\cdot)]^{-1}$. 
For the boundary element approximation of the eigenvalue problem in~\eqref{Eq:EigContAinv} we need in addition to the space of piecewise constant functions $S_0(\triang_N)$  the space of piecewise linear functions $S_1(\triang_N)$. Formally, the Galerkin eigenvalue problem
\begin{equation}\label{Eq:EigContAinvDiscret}
\left( {\mathcal A}(\lambda)^{-1} 
\begin{pmatrix}
 \psi_N 
 \\
 \varphi_N
\end{pmatrix},
\begin{pmatrix}
 \widetilde \psi_N
 \\
 \widetilde\varphi_N
\end{pmatrix}
\right)
  =0
\qquad\text{for all }
\begin{pmatrix}
 \widetilde \psi_N
 \\
 \widetilde\varphi_N
\end{pmatrix}
\in S_1(\triang_N)\times S_0(\triang_N)
\end{equation}
is considered.
However, if the contour integral method is used for the computations of the eigenvalues of the Galerkin eigenvalue problem~\eqref{Eq:EigContAinvDiscret}, then  $\ma(\cdot)^{-1}$ has not to be computed, since the contour integral method operates on its inverse $\ma(\cdot)$. The abstract convergence results of Section~\ref{SectionGalerkinApprox} can be applied to the approximation of those eigenvalues $\lambda_0$ of the eigenvalue problem~\eqref{Eq:EigContAinv} for which $\ma(\cdot)^{-1}$ is holomorphic. In general it is possible  that $\lambda_0$ is a pole of  $\ma(\cdot)$ and of
$\ma(\cdot)^{-1}$. In this case, as mentioned before, a rigorous analysis of the Galerkin approximation has not been provided so far.

\subsection{Numerical examples}\label{SubSection:NumExDelta}
We present two numerical examples for  $\mathcal{P}=-\Delta$. In this case $A_0$ is the free Laplace operator and $\sigma(A_0) = \sigma_\text{ess}(A_0) = [0, \infty)$, and 
the fundamental solution for $\mathcal{P}-\lambda$ is given by $G(\lambda;x,y)=e^{i\sqrt{\lambda}\| x-y\|}(4\pi\|x-y\|)^{-1}$~\cite[Chapter~9]{M00}. In particular, the operator $A_0$ has no discrete eigenvalues and therefore the eigenvalues of $A_\alpha$ coincide  with the eigenvalues of the eigenvalue problem for $I+\alpha \ms_0(\cdot)$. The Galerkin eigenvalue problem~\eqref{Eq:GalerkinEVPDelta} is used for the computation of approximations of discrete eigenvalues of $A_\alpha$ and corresponding eigenfunctions. In all numerical experiments the open-source  library BEM++~\cite{BEMplusplus:2015} is employed for the computations of the boundary element 
matrices.

\subsubsection{Unit ball} As first numerical example we consider  as domain $\Omega_\text{i}$  the unit ball and a constant $\alpha$. The  eigenvalues of $A_{\alpha}$ for constant $\alpha$ have an analytical representation \cite[Theorem~3.2]{AGS87} which are used to show that in the numerical experiments the predicted convergence order~\eqref{Eq:ErrorApproxBoundary} is reflected. Let $l \in \mathbb{N}_0$ be such that $2 l + 1 < -\alpha$. Then $\lambda^{(l)}$
is an eigenvalue of $A_{\alpha}$ of multiplicity $2 l + 1$ if
  \begin{equation*}
    1 + \alpha I_{l+1/2}\big(\sqrt{-\lambda^{(l)}} \big) K_{l+1/2}\big(\sqrt{-\lambda^{(l)}} \big) =
0,
  \end{equation*}
  where $I_{l+1/2}$ and $K_{l+1/2}$ denote modified Bessel functions of order
$l+1/2$.
  Conversely, all  eigenvalues of $A_{\alpha}$ are of the above
form. 

For the numerical experiments we choose $\alpha=-6$.  
In Table~\ref{Table:Errors} the errors of the approximation of the eigenvalues of $A_\alpha$ with $\alpha=-6$
for three different mesh sizes $h$ are  given. For multiple eigenvalues $\lambda^{(l)}$, $l=1,2$,  we have used the mean value of the approximations, denoted by $\widehat{\lambda}^{(l)}_{h}$, for the computation of the error.
The experimental convergence order (eoc) reflects the predicted quadratic
convergence order~\eqref{Eq:ErrorApproxBoundary}. In
Figure~\ref{Fig:EigFunctionsPlane} plots of  computed eigenfunctions of
$A_{\alpha}$ in the $xy$-plane are given where for each exact eigenvalue
one approximated eigenfunction is selected.
\begin{table}[h]
\begin{center}
\begin{tabular}{l r r  r  r  r  r r r r r  r  r  r  r r r}
\toprule
$h$   &  & $ \frac{\left\vert\lambda^{(0)}_{h}-\lambda^{(0)}\right\vert}{\left\vert \lambda^{(0)}\right\vert} $ &eoc& &
$ \frac{\left\vert\widehat\lambda^{(1)}_{h}-\lambda^{(1)}\right\vert}{\left\vert \lambda^{(1)}\right\vert} $
 & eoc& & 
 $ \frac{\left\vert\widehat\lambda^{(2)}_{h}-\lambda^{(2)}\right\vert}{\left\vert \lambda^{(2)}\right\vert} $
 & eoc \\
\midrule 
0.2     & & 1.203e-2 & - &  & 2.837e-2 & -    & & 1.666e-1 & -\\
0.1  &  &2.473e-3 & 2.28& & 6.968e-3 & 2.02& &3.969e-2 & 2.07 \\
0.05 &  &4.344e-4& 2.48& & 1.781e-3 & 1.95 & &9.593e-3 & 2.06 \\
\bottomrule
\end{tabular} 
\end{center}
\caption{Error of the approximations of the eigenvalues of
$A_{\alpha}$, $\alpha=-6$, of the unit sphere for different mesh-sizes
$h$. }\label{Table:Errors}
\end{table}

 \begin{figure}[ht]
 \begin{center}
  \includegraphics[width=0.32\textwidth]{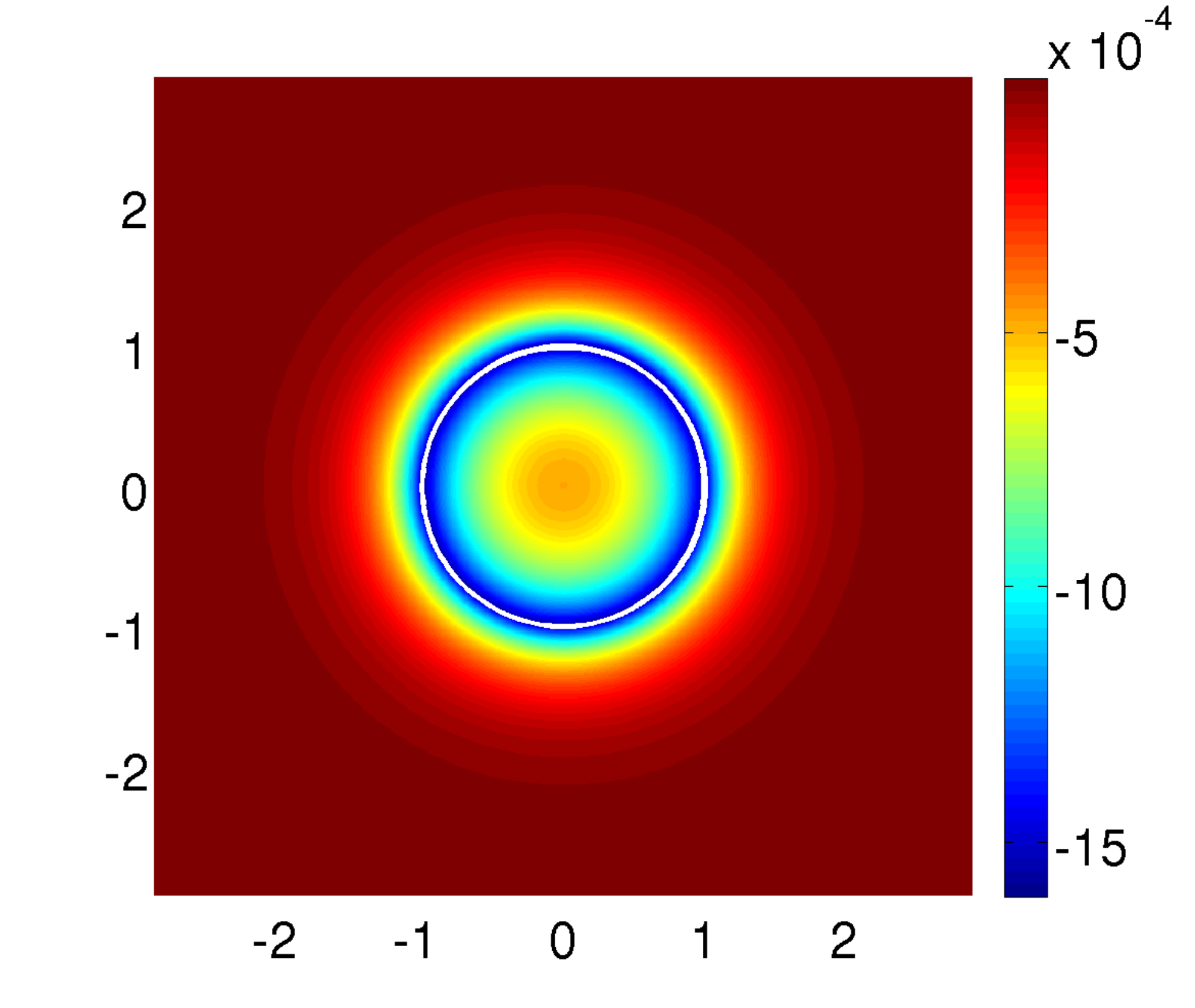} 
 \includegraphics[width=0.32\textwidth]{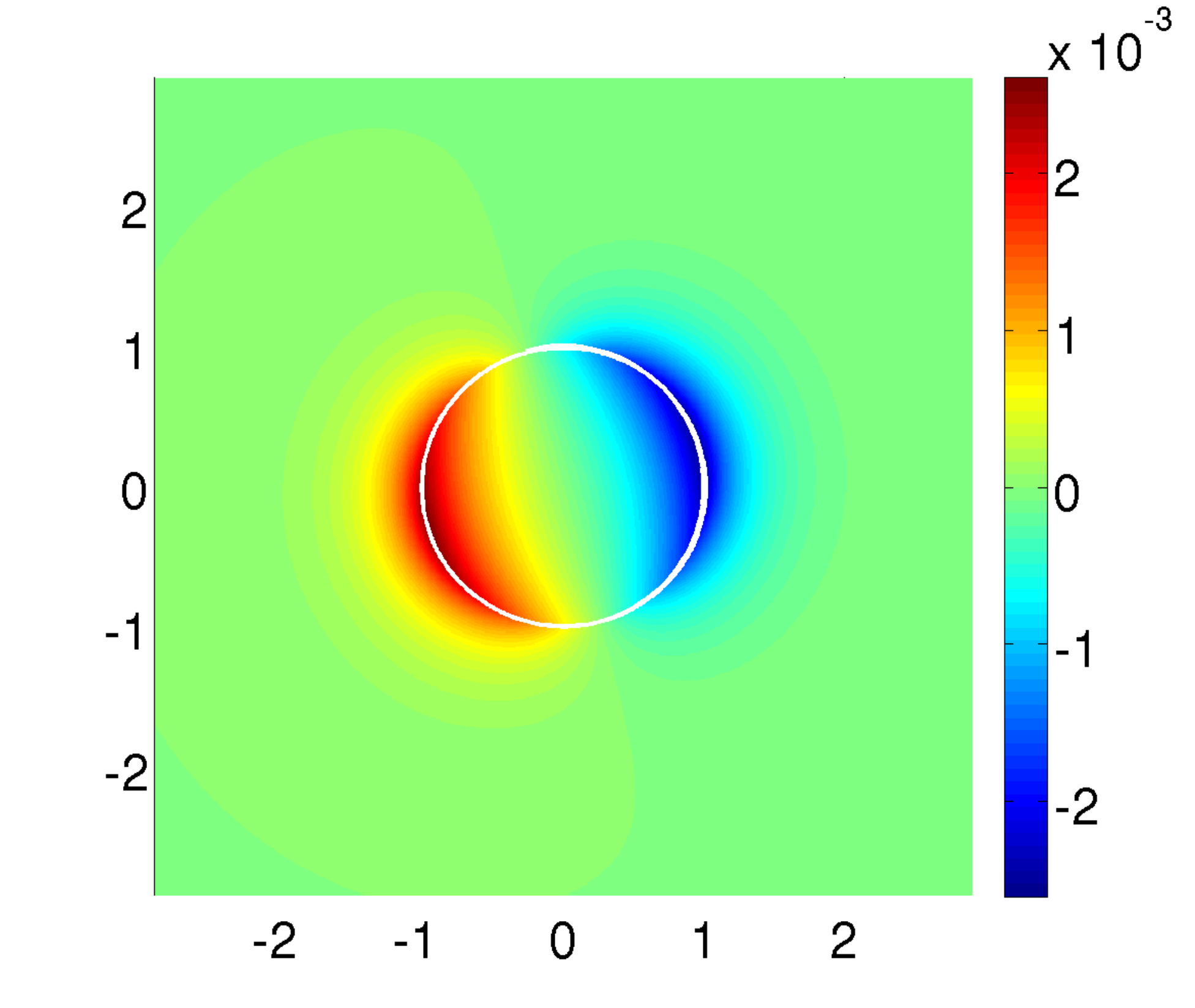}
  \includegraphics[width=0.32\textwidth]{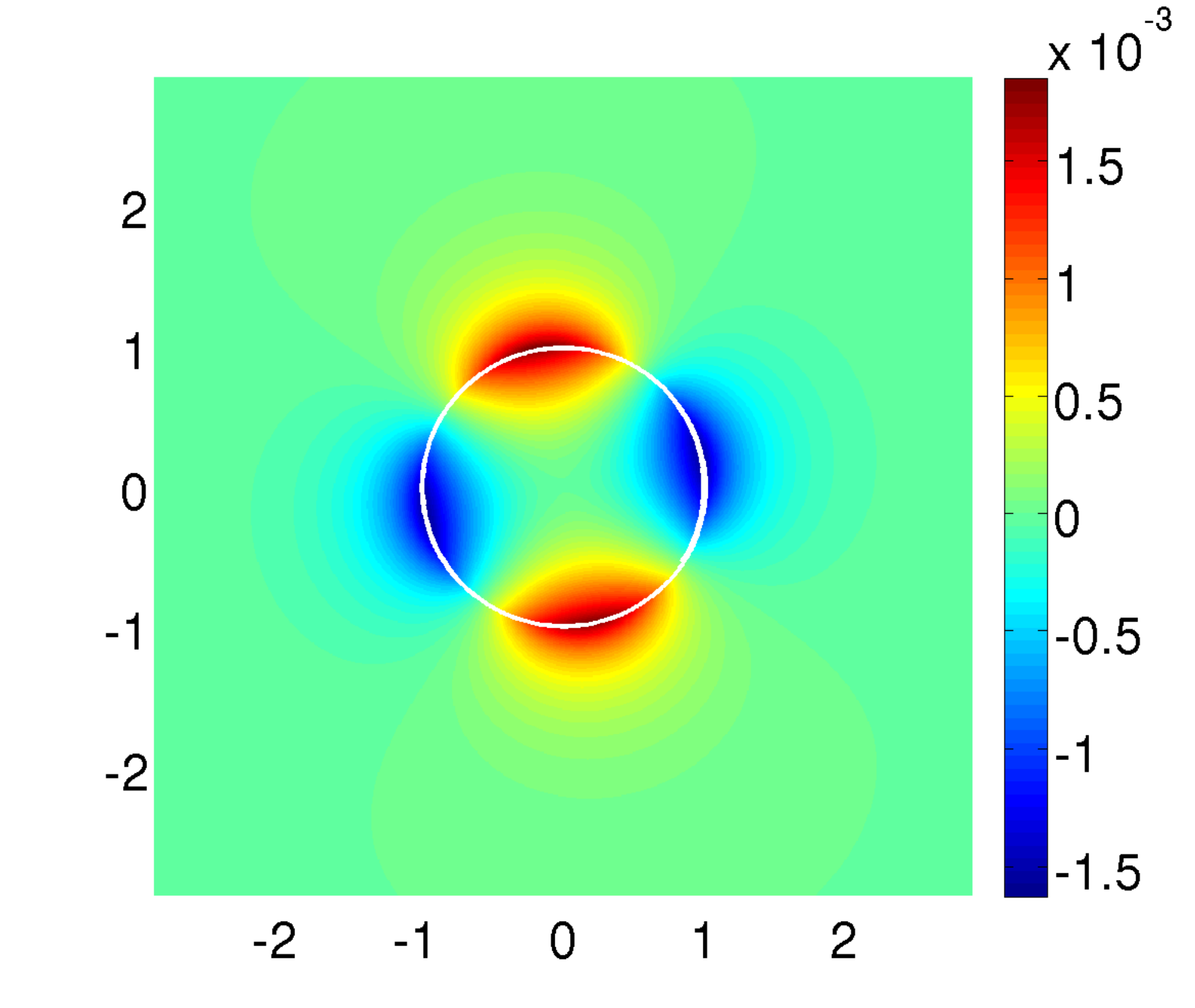}  
 \end{center}
 \caption{ Computed eigenfunctions of
 $A_{\alpha}$, $\alpha=-6$, in the
 $xy$-plane for the unit ball.}\label{Fig:EigFunctionsPlane}
 \end{figure}
 
\subsubsection{Screen} For the second numerical example we have chosen a $\delta$-potential supported on the non-closed surface $\Gamma:=[0,1]\times[0,1]\times\{0\}\subset\R^3$, which is referred to as screen. 
The interaction strength $\alpha$ is  defined by $\alpha=-15\chi_{\Gamma}$, where $\chi_{\Gamma}$ is the characteristic function on $\Gamma$ given as
\begin{equation*}
 \chi_{\Gamma}(x)
 :=
 \begin{cases}
 1,& \text{for }x\in  \Gamma,\\
 0, & \text{else }.
 \end{cases}
\end{equation*}
Such a problem fits in the described theory of this section. Take for example as domain $\Omega_\text{i}$ the unit cube, as we have done in our numerical experiments, then $\Gamma$ is identical with one of the faces of $\Sigma=\partial\Omega_\text{i}$.

In the numerical experiments we have chosen as  contour the ellipse 
$g(t)=c+a\cos(t)+i b \sin(t)$, $t\in [0,2\pi]$, with $c=-15.0$, $a=14.99$ and $b=0.01$. 
We have got four eigenvalues of the discretized eigenvalue problem inside the contour, namely $\lambda_h^{(1)}=  -43.02$, $\lambda_h^{(2)}=-23.93$, $\lambda_h^{(3)}=-23.88$, and $\lambda_h^{(3)}=-5.59$ for the mesh-size $h=0.0125$. Plots of the numerical approximations of the eigenfunctions in the $xy$-plane are given in Figure~\ref{Fig:EigFunctionsPlaneScreen}.
\begin{figure}[ht]
 \begin{center}
 \includegraphics[width=0.24\textwidth]{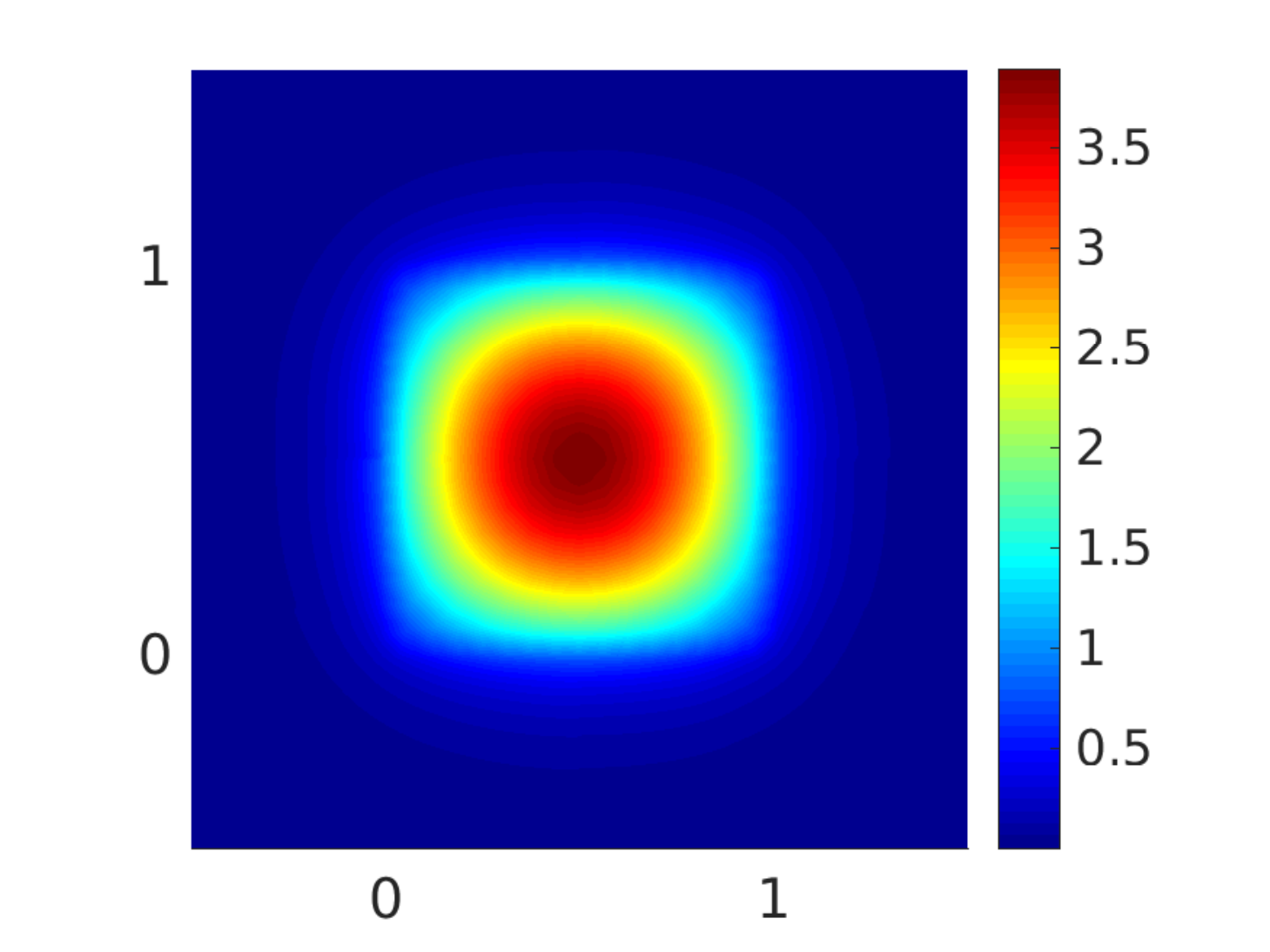}    
 \includegraphics[width=0.24\textwidth]{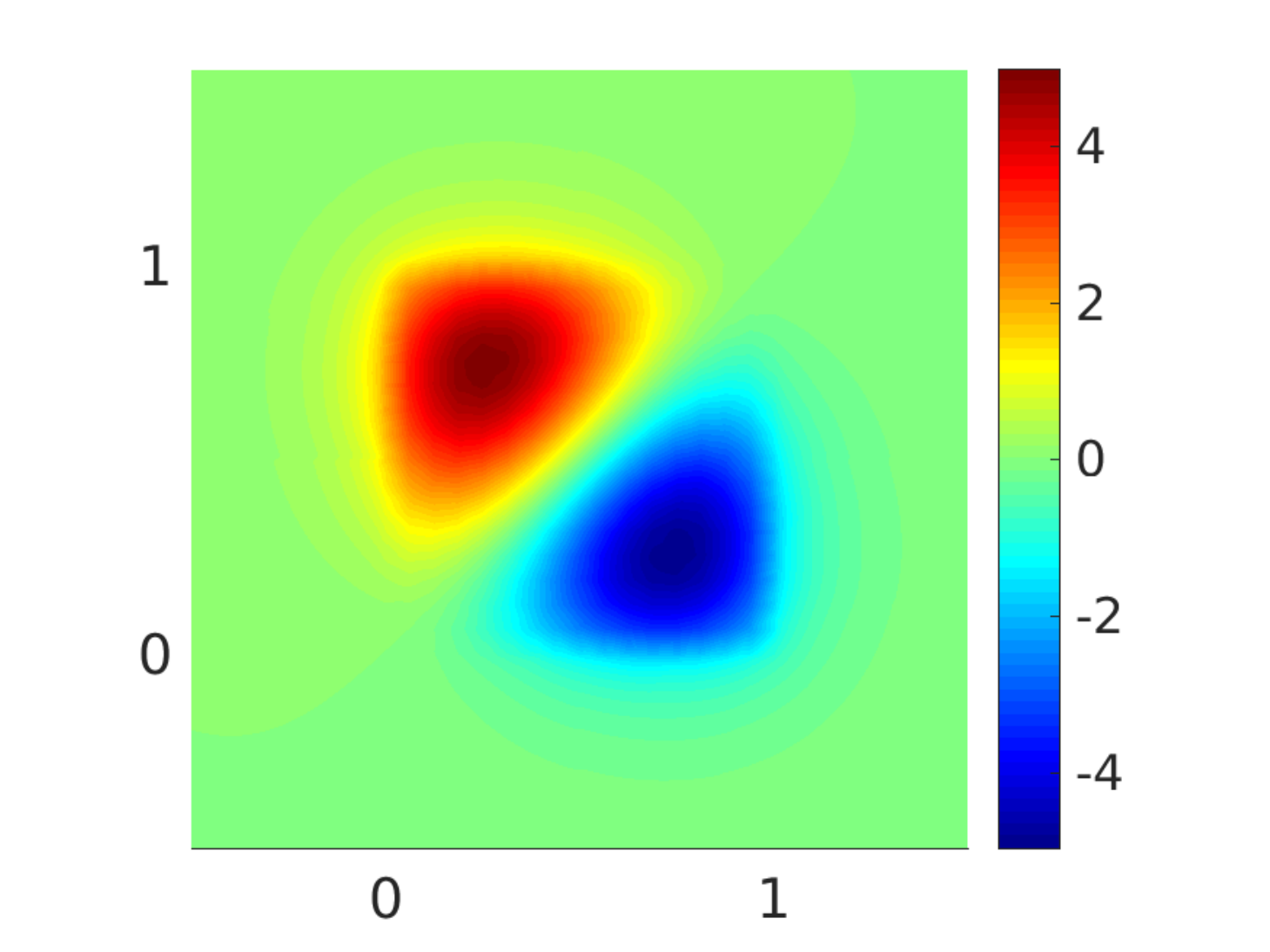}
  \includegraphics[width=0.24\textwidth]{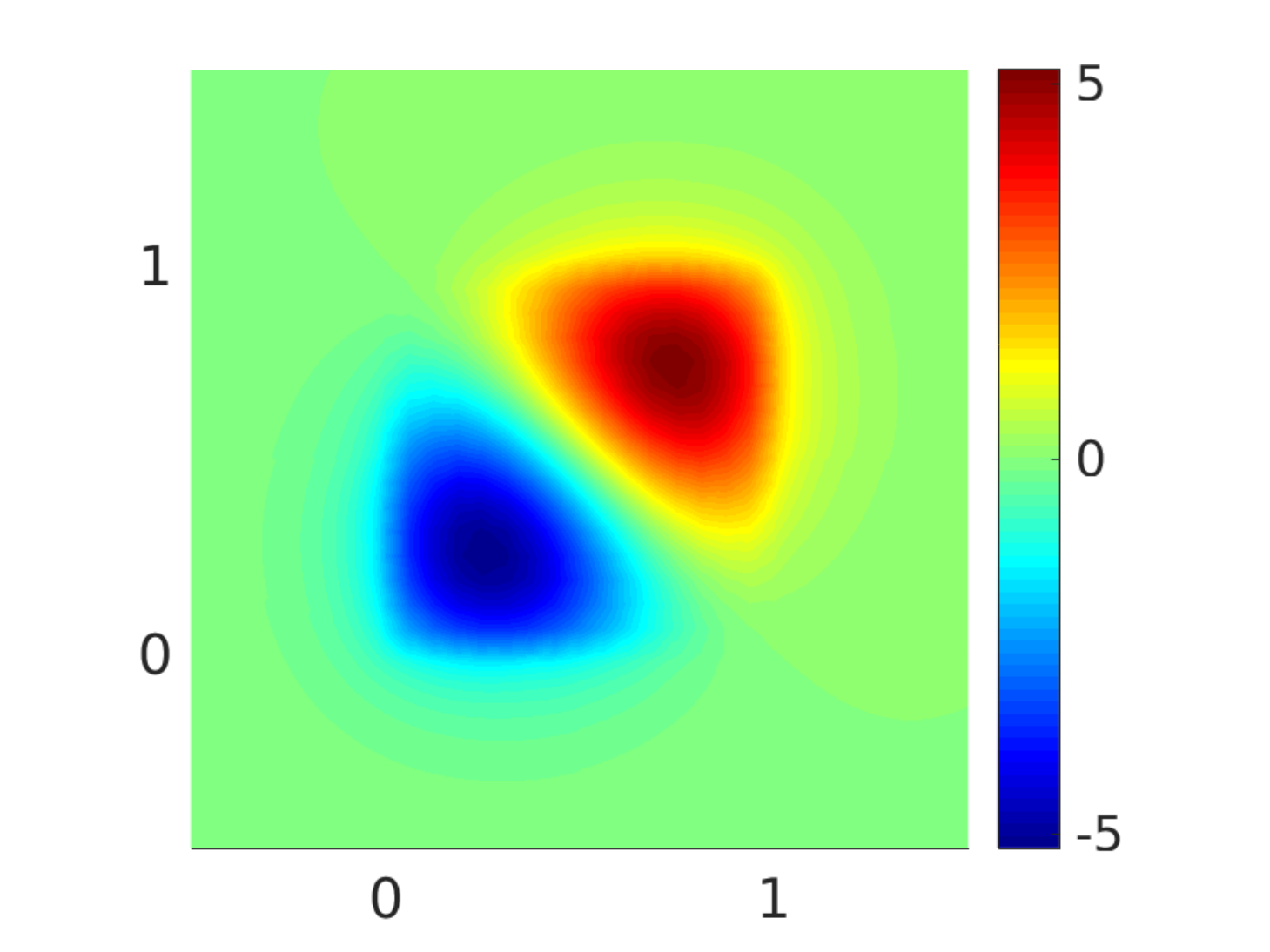}
 \includegraphics[width=0.24\textwidth]{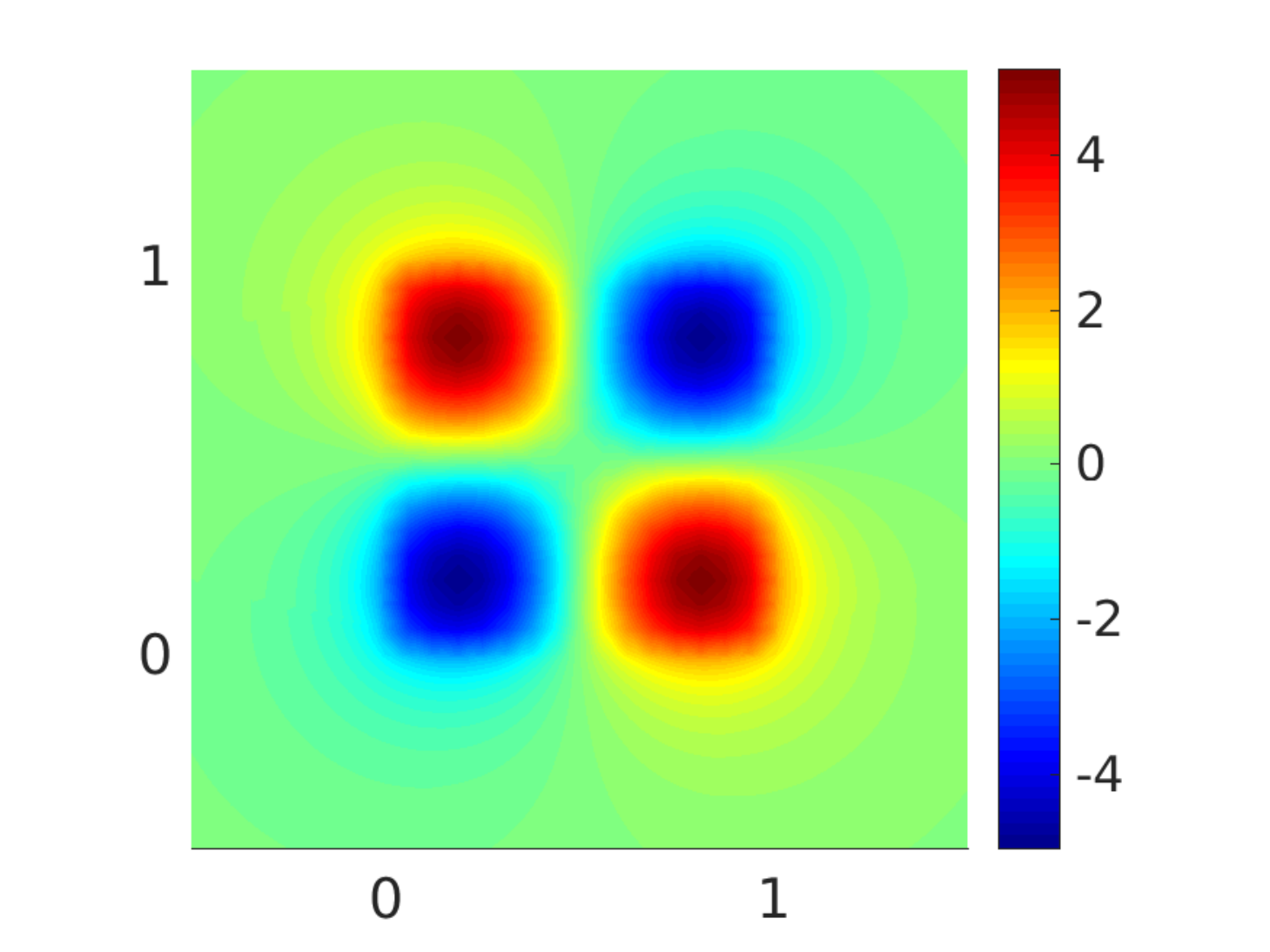}
 \end{center}
 \caption{Computed eigenfunctions of
 $A_{\alpha}$ in the
 $xy$-plane for $\alpha=-15\chi_{[0,1]\times[0,1]\times\{0\}}$.} \label{Fig:EigFunctionsPlaneScreen}
 \end{figure}

\section{Elliptic differential operators with $\delta'$-interactions supported on compact Lipschitz smooth surfaces}
\label{section_delta_prime}

In this section we study the spectral properties of the partial differential operator which corresponds to the formal expression $B_\beta := \mathcal{P} + \beta \langle \delta_\Sigma', \cdot \rangle \delta'$ in a mathematically rigorous way and study its spectral properties. The considerations are very similar as for $A_\alpha$ in Section~\ref{section_delta}. First, in Section~\ref{section_delta_prime_op_analysis} we show the self-adjointness of $B_\beta$ in $L^2(\mathbb{R}^n)$ and obtain the Birman-Schwinger principle to characterize the discrete eigenvalues of $B_\beta$ via boundary integral operators in Proposition~\ref{proposition_Birman_Schwinger_delta_prime}. Then, in Section~\ref{SubSectNumDeltaPrime} we discuss how these boundary integral equations can be solved numerically by boundary element methods. Finally, in Section~\ref{SubSecNumPrime} we show some numerical examples.

\subsection{Definition and self-adjointness of $B_\beta$} \label{section_delta_prime_op_analysis}

For a real valued function $\beta$ with $\beta^{-1} \in L^\infty(\Sigma)$ we define in $L^2(\mathbb{R}^n)$ the partial differential operator $B_\beta$ by
\begin{equation} \label{def_B_beta}
  \begin{split}
    B_\beta f &:= \mathcal{P} f_\text{i} \oplus \mathcal{P} f_\text{e}, \\
    \dom B_\beta &:= \big\{ f = f_\text{i} \oplus f_\text{e} \in H^1_\mathcal{P}(\Omega_\text{i}) \oplus H^1_\mathcal{P}(\Omega_\text{e}): \mathcal{B}_\nu f_\text{i} = \mathcal{B}_\nu f_\text{e}, \, \gamma f_\text{e} - \gamma f_\text{i} = \beta \mathcal{B}_\nu f \big\}.
  \end{split}
\end{equation}
With the help of~\eqref{Green_extended} it is not difficult to show that $B_\beta$ is symmetric in $L^2(\mathbb{R}^n)$:

\begin{lem} \label{lemma_symmetric_delta_prime}
  Let $\beta$ be a real valued function on $\Sigma$ with $\beta^{-1} \in L^\infty(\Sigma)$. Then the operator $B_\beta$ defined by~\eqref{def_B_beta} is symmetric in $L^2(\mathbb{R}^n)$.
\end{lem}
\begin{proof}
  We show that $(B_\beta f, f)_{L^2(\mathbb{R}^n)} \in \mathbb{R}$ for all $f \in \dom B_\beta$. Let $f \in \dom B_\beta$ be fixed. 
  Using~\eqref{Green_extended} in $\Omega_\text{i}$ and $\Omega_\text{e}$ and that the normal $\nu$ is pointing outside of $\Omega_\text{i}$ and inside of $\Omega_\text{e}$ we get
  \begin{equation*}
    \begin{split}
      (B_\beta f, f)_{L^2(\mathbb{R}^n)} &= (\mathcal{P} f_\text{i}, f_\text{i})_{L^2(\Omega_\text{i})} + (\mathcal{P} f_\text{e}, f_\text{e})_{L^2(\Omega_\text{e})} \\
      &= \Phi_{\Omega_\text{i}}[f_\text{i}, f_\text{i}] - (\mathcal{B}_\nu f_\text{i}, \gamma f_\text{i}) + \Phi_{\Omega_\text{e}}[f_\text{e}, f_\text{e}] + (\mathcal{B}_\nu f_\text{e}, \gamma f_\text{e}).
    \end{split}
  \end{equation*}
  Since $f \in \dom B_\beta$ we have $\mathcal{B}_\nu f_\text{i} = \mathcal{B}_\nu f_\text{e}$ and $\beta \mathcal{B}_\nu f = (\gamma f_\text{e} - \gamma f_\text{i})$. Therefore, we conclude
  \begin{equation*}
    \begin{split}
      (B_\beta f, f)_{L^2(\mathbb{R}^n)} &= \Phi_{\Omega_\text{i}}[f_\text{i}, f_\text{i}] + \Phi_{\Omega_\text{e}}[f_\text{e}, f_\text{e}] + (\mathcal{B}_\nu f, \gamma f_\text{e} - \gamma f_{\text{i}}) \\ 
      &= \Phi_{\Omega_\text{i}}[f_\text{i}, f_\text{i}] + \Phi_{\Omega_\text{e}}[f_\text{e}, f_\text{e}] + (\mathcal{B}_\nu f, \beta \mathcal{B}_\nu f).
    \end{split}
  \end{equation*}
  Since the sesquilinear forms $\Phi_{\Omega_{\text{i}/\text{e}}}$ are symmetric, the latter number is real and therefore, the claim is shown.
\end{proof}

The following proposition is the  counterpart of Proposition~\ref{proposition_Birman_Schwinger_delta} to characterize the discrete eigenvalues of $B_\beta$ via boundary integral operators. It is the theoretic basis to compute these eigenvalues with the help of boundary element methods in Section~\ref{SubSectNumDeltaPrime}.
To formulate the result below recall for $\lambda \in \rho(A_0) \cup \sigma_\text{disc}(A_0)$ the definition of the double layer potential $\text{DL}(\lambda)$ from~\eqref{def_double_layer_potential},  the set $\mathcal{N}_\lambda$ from~\eqref{def_N_lambda},  the hypersingular boundary integral operator $\mathcal{R}(\lambda)$ from~\eqref{def_hypersing_bdd_int_op}, and $R_{\mathcal{A}(\lambda_0)}$ from~\eqref{def_residual}. 

\begin{prop} \label{proposition_Birman_Schwinger_delta_prime}
  Let $\beta$ be a real valued function on $\Sigma$ with $\beta^{-1} \in L^\infty(\Sigma)$ and let $B_\beta$ be defined by~\eqref{def_B_beta}. Then the following is true for any $\lambda \in \rho(A_0) \cup \sigma_\textup{disc}(A_0)$:
  \begin{itemize}
    \item[(i)] $\ker(B_\beta - \lambda) \ominus \ker(A_0-\lambda) \neq \{ 0 \}$ if and only if there exists $0 \neq \varphi \in \mathcal{N}_\lambda$ such that $(\beta^{-1} + \mathcal{R}(\lambda)) \varphi = 0$. Moreover,
    \begin{equation} \label{kernel_Birman_Schwinger_delta_prime}
      \ker (B_\beta - \lambda) \ominus \ker(A_0-\lambda) = \big\{ \textup{DL}(\lambda) \varphi: \varphi \in \mathcal{N}_\lambda, (\beta^{-1} + \mathcal{R}(\lambda)) \varphi  =0 \big\}.
    \end{equation}
    \item[(ii)] If $\lambda \in \rho(A_0)$, then $\lambda \in \sigma_\textup{p}(B_\beta)$ if and only if $0 \in \sigma_\textup{p}(\beta^{-1} + \mathcal{R}(\lambda))$.
    \item[(iii)] $\ker(B_\beta - \lambda) \cap \ker(A_0-\lambda) \neq \{ 0 \}$ if and only if there exists $(\varphi, \psi)^\top \in \ran R_{\mathcal{A}(\lambda_0)}$ such that $\psi = 0$.
    \item[(iv)] If $\lambda \notin \sigma_\textup{p}(B_\beta) \cup \sigma(A_0)$, then $\beta^{-1} + \mathcal{R}(\lambda): H^{1/2}(\Sigma) \rightarrow H^{-1/2}(\Sigma)$ admits a bounded and everywhere defined inverse.
  \end{itemize}
\end{prop}
\begin{proof}
  (i) Assume first that $\ker(B_\beta - \lambda) \ominus \ker(A_0-\lambda) \neq \{ 0 \}$ and let $f \in \ker(B_\beta - \lambda) \ominus \ker(A_0-\lambda) $. Then by Lemma~\ref{lemma_double_layer_potential}~(i) there exists $\varphi \in \mathcal{N}_\lambda$ such that $f = \text{DL}(\lambda) \varphi$. Since $f \in \dom B_\beta$ one has with Lemma~\ref{lemma_double_layer_potential}~(ii)
  \begin{equation*}
    \beta \mathcal{B}_\nu f = \gamma f_\text{e} - \gamma f_\text{i} = \gamma(\text{DL}(\lambda) \varphi)_\text{e} - \gamma(\text{DL}(\lambda) \varphi)_\text{i} = \varphi.
  \end{equation*}
  With $\mathcal{B}_\nu f = -\mathcal{R}(\lambda) \varphi$ this can be rewritten as $\beta^{-1} \varphi = -\mathcal{R}(\lambda) \varphi$. Hence, the above considerations show
  \begin{equation} \label{kernel_Birman_Schwinger_delta_prime1}
    \ker (B_\beta - \lambda) \ominus \ker(A_0-\lambda) \subset \big\{ \textup{DL}(\lambda) \varphi: \varphi \in \mathcal{N}_\lambda, (\beta^{-1} + \mathcal{R}(\lambda)) \varphi  =0 \big\}.
  \end{equation}
  
  Conversely, assume that there exists $\varphi \in \mathcal{N}_\lambda$ such that $(\beta^{-1} + \mathcal{R}(\lambda))\varphi=0$. Then $f := \text{DL}(\lambda) \varphi \in H^1_\mathcal{P}(\mathbb{R}^n \setminus \Sigma)$ is nontrivial by Lemma~\ref{lemma_double_layer_potential}~(ii). Using the jump properties of $\text{DL}(\lambda) \varphi$ from Lemma~\ref{lemma_double_layer_potential}~(ii) we conclude further $\mathcal{B}_\nu f_\text{i} = \mathcal{B}_\nu f_\text{e}$ and
  \begin{equation*}
    \gamma f_\text{e} - \gamma f_\text{i} = \gamma(\text{DL}(\lambda) \varphi)_\text{e} - \gamma(\text{DL}(\lambda) \varphi)_\text{i} = \varphi = -\beta \mathcal{R}(\lambda) \varphi = \beta \mathcal{B}_\nu f,
  \end{equation*}
  where $\beta^{-1}(I  + \beta \mathcal{R}(\lambda)) \varphi=0$ was used. Hence, $f \in \dom B_\beta$. With Lemma~\ref{lemma_double_layer_potential}~(i) we conclude with $\varphi \in \mathcal{N}_\lambda$ eventually
  \begin{equation*}
    (B_\beta - \lambda) f = (\mathcal{P} - \lambda) (\text{DL}(\lambda) \varphi)_\text{i} \oplus (\mathcal{P} - \lambda) (\text{DL}(\lambda) \varphi)_\text{e}  =0,
  \end{equation*}
  which shows $\lambda \in \sigma_\textup{p}(B_\beta)$ and
  \begin{equation} \label{kernel_Birman_Schwinger_delta_prime2}
     \big\{ \textup{DL}(\lambda) \varphi: \varphi \in \mathcal{N}_\lambda, (\beta^{-1} + \mathcal{R}(\lambda)) \varphi  =0 \big\} \subset \ker (B_\beta - \lambda) \ominus \ker(A_0-\lambda).
  \end{equation}
  Note that~\eqref{kernel_Birman_Schwinger_delta_prime1} and~\eqref{kernel_Birman_Schwinger_delta_prime2} give~\eqref{kernel_Birman_Schwinger_delta_prime}. Hence, all claims in item~(i) are proved.
  
  Assertion~(ii) is a simple consequence of item~(i), as for $\lambda \notin \sigma(A_0)$ we have $\ker(A_0-\lambda)=\{ 0 \}$ and $\mathcal{N}_\lambda = H^{1/2}(\Sigma)$.
  
  Statement~(iii) follows from Theorem~\ref{theorem_spectrum_A_0}. Note that $f \in \dom B_\beta \cap \dom A_0$ if and only if $f \in H^2(\mathbb{R}^n)$ and $\mathcal{B}_\nu f = \beta^{-1} (\gamma f_\text{e} - \gamma f_\text{i}) = 0$. With Theorem~\ref{theorem_spectrum_A_0} it follows that $f \in \ker(B_\beta - \lambda) \cap \ker(A_0 - \lambda)$ if and only if there exists $(\varphi, \psi)^\top = (\gamma f, \mathcal{B}_\nu f)^\top \in \ran R_{\mathcal{A}(\lambda_0)}$ such that $\psi = 0$.
  
  (iv) First, we note that the multiplication with the function $\beta^{-1} \in L^\infty(\Sigma)$ gives rise to a bounded operator from $H^{1/2}(\Sigma)$ to $L^2(\Sigma)$ and as $L^2(\Sigma)$ is compactly embedded in $H^{-1/2}(\Sigma)$, the operator $\beta^{-1}: H^{1/2}(\Sigma) \rightarrow H^{-1/2}(\Sigma)$ is compact. Therefore, we deduce from \cite[Theorem~2.26]{M00} that $\beta^{-1} + \mathcal{R}(\lambda): H^{1/2}(\Sigma) \rightarrow H^{-1/2}(\Sigma)$ is a Fredholm operator with index zero, as $\mathcal{R}(\lambda)$ is a Fredholm operator with index zero by Lemma~\ref{lemma_hypersing_bdd_int_op}~(ii). Since $\lambda$ is not an eigenvalue of $B_\beta$ by assumption, we deduce from~(ii) that $\beta^{-1} + \mathcal{R}(\lambda)$ is injective and hence, this operator is also surjective. Therefore, it follows from the open mapping theorem that $\beta^{-1} + \mathcal{R}(\lambda)$  has a bounded inverse from $H^{-1/2}(\Sigma)$ to $H^{1/2}(\Sigma)$.
\end{proof}

In the following proposition we show the self-adjointness of $B_\beta$ and a Krein type resolvent formula for this operator. We remark that the resolvent formula in~\eqref{krein_delta_prime} is well defined, as $\beta^{-1} + \mathcal{R}(\lambda): H^{1/2}(\Sigma) \rightarrow H^{-1/2}(\Sigma)$ is boundedly invertible for $\lambda \in \rho(A_0) \cap \rho(B_\beta)$ by Proposition~\ref{proposition_Birman_Schwinger_delta_prime}~(iv).

\begin{prop} \label{proposition_self-adjoint_delta_prime}
  Let $\beta$ be a real valued function on $\Sigma$ with $\beta^{-1} \in L^\infty(\Sigma)$ and let the operators $A_0$, $\textup{DL}(\lambda)$, and $\mathcal{R}(\lambda)$, $\lambda \in \rho(A_0)$, be given by~\eqref{def_free_Op}, \eqref{def_double_layer_potential}, and~\eqref{def_hypersing_bdd_int_op}, respectively. Then the operator $B_\beta$ defined by~\eqref{def_B_beta} is self-adjoint in $L^2(\mathbb{R}^n)$ and the following is true:
  \begin{itemize}
    \item[(i)] For $\lambda \in \rho(A_0) \cap \rho(B_\beta)$ the resolvent of $B_\beta$ is given by
    \begin{equation} \label{krein_delta_prime}
      (B_\beta - \lambda)^{-1} = (A_0 - \lambda)^{-1} + \textup{DL}(\lambda) \big( \beta^{-1} + \mathcal{R}(\lambda) \big)^{-1} \mathcal{B}_\nu (A_0 - \lambda)^{-1}.
    \end{equation}
    \item[(ii)] $\sigma_\textup{ess}(B_\beta) = \sigma_\textup{ess}(A_0)$.
  \end{itemize}
\end{prop}
\begin{proof}
  In order to show that $B_\beta$ is self-adjoint, we show that $\ran (B_\beta - \lambda) = L^2(\mathbb{R}^n)$ for $\lambda \in \mathbb{C} \setminus (\sigma(A_0) \cup \sigma_\textup{p}(B_\beta))$. Let $f \in L^2(\mathbb{R}^n)$ be fixed and define
  \begin{equation*} 
    g := (A_0 - \lambda)^{-1} f + \textup{DL}(\lambda) \big( \beta^{-1} + \mathcal{R}(\lambda) \big)^{-1} \mathcal{B}_\nu (A_0 - \lambda)^{-1} f.
  \end{equation*}
  Note that $g$ is well defined, as $\beta^{-1} + \mathcal{R}(\lambda): H^{1/2}(\Sigma) \rightarrow H^{-1/2}(\Sigma)$ admits a bounded inverse for $\lambda \notin \sigma(A_0) \cup \sigma_\textup{p}(B_\beta)$ by Proposition~\ref{proposition_Birman_Schwinger_delta_prime}~(iv).
  We are going to show that $g \in \dom B_\beta$ and $(B_\beta - \lambda) g = f$. This shows then $\ran (B_\beta - \lambda) = L^2(\mathbb{R}^n)$ and also~\eqref{krein_delta_prime}.
  
  Since $(A_0 - \lambda)^{-1} f \in H^2(\mathbb{R}^n)$ by Proposition~\ref{proposition_resolvent} implies $\mathcal{B}_\nu (A_0 - \lambda)^{-1} f \in L^2(\Sigma) \subset H^{-1/2}(\Sigma)$, we conclude from Proposition~\ref{proposition_Birman_Schwinger_delta_prime}~(iv) and Lemma~\ref{lemma_double_layer_potential} that
  \begin{equation*}
    \textup{DL}(\lambda) \big( \beta^{-1} + \mathcal{R}(\lambda) \big)^{-1} \mathcal{B}_\nu (A_0 - \lambda)^{-1} f \in H^1_\mathcal{P}(\mathbb{R}^n \setminus \Sigma).
  \end{equation*}
  Therefore, also $g \in H^1_\mathcal{P}(\mathbb{R}^n \setminus \Sigma)$. Moreover, we get with the help of Lemma~\ref{lemma_double_layer_potential}~(ii) that $\mathcal{B}_\nu g_\text{e} = \mathcal{B}_\nu g_\text{i}$.
  Applying once more Lemma~\ref{lemma_double_layer_potential}~(ii) we conclude
  \begin{equation*}
    \begin{split}
      \beta^{-1}(\gamma g_\text{e} - \gamma g_\text{i}) - \mathcal{B}_\nu g &= \beta^{-1} \big( \beta^{-1} + \mathcal{R}(\lambda) \big)^{-1} \mathcal{B}_\nu (A_0 - \lambda)^{-1} f - \mathcal{B}_\nu (A_0 - \lambda)^{-1} f \\
      &\qquad +  \mathcal{R}(\lambda) \big( \beta^{-1} + \mathcal{R}(\lambda) \big)^{-1} \mathcal{B}_\nu (A_0 - \lambda)^{-1} f = 0,
    \end{split}
  \end{equation*}
  which shows $g \in \dom B_\beta$. Next, we have with $\varphi := ( \beta^{-1} + \mathcal{R}(\lambda) )^{-1} \mathcal{B}_\nu (A_0 - \lambda)^{-1} f$
  \begin{equation*}
    \begin{split}
      (B_\beta - \lambda) g &= (\mathcal{P} - \lambda) (A_0-\lambda)^{-1} f + (\mathcal{P}-\lambda) (\text{DL}(\lambda) \varphi)_\text{i} \oplus (\mathcal{P}-\lambda) (\text{DL}(\lambda) \varphi)_\text{e} = f,
    \end{split}
  \end{equation*}
  where~\eqref{range_double_layer_potential} was used in the last step.
  With the previous considerations we deduce now the self-adjointness of $B_\beta$ and~\eqref{krein_delta_prime}.
  
  It remains to show assertion~(ii). Let $\lambda \in \mathbb{C} \setminus \mathbb{R}$ be fixed. Due to the mapping properties of the resolvent of $A_0$ from Proposition~\ref{proposition_resolvent} and the mapping properties of $\mathcal{B}_\nu$ from~\eqref{conormal_derivative_extension} the operator
  \begin{equation*}
    \mathcal{B}_\nu (A_0-\lambda)^{-1}: L^2(\mathbb{R}^n) \rightarrow L^2(\Sigma) \hookrightarrow H^{-1/2}(\Sigma)
  \end{equation*}
  is bounded. Hence, Proposition~\ref{proposition_Birman_Schwinger_delta_prime}~(iv) yields that
  \begin{equation*}
    \big( \beta^{-1} + \mathcal{R}(\lambda) \big)^{-1} \mathcal{B}_\nu (A_0 - \lambda)^{-1}: L^2(\mathbb{R}^n) \rightarrow H^{1/2}(\Sigma)
  \end{equation*}
  is bounded. As $H^{1/2}(\Sigma)$ is compactly embedded in $L^2(\Sigma)$ we conclude that the latter operator is compact from $L^2(\mathbb{R}^n)$ to $L^2(\Sigma)$. Since $\text{DL}(\lambda): L^2(\Sigma) \rightarrow L^2(\mathbb{R}^n)$ is bounded by~\eqref{def_double_layer_potential}, we find eventually that
  \begin{equation*}
    (B_\beta - \lambda)^{-1} - (A_0 - \lambda)^{-1} = \textup{DL}(\lambda) \big( \beta^{-1} + \mathcal{R}(\lambda) \big)^{-1} \mathcal{B}_\nu (A_0 - \lambda)^{-1}
  \end{equation*}
  is compact in $L^2(\mathbb{R}^n)$. Therefore, we get with the Weyl theorem $\sigma_\text{ess}(B_\beta) = \sigma_\text{ess}(A_0)$.
\end{proof}

The following result is the counterpart of Proposition~\ref{proposition_Birman_Schwinger_kernel_delta} on the inverse of the Birman-Schwinger operator $\beta^{-1} + \mathcal{R}(\lambda)$, which will be of great importance for the numerical calculation of the discrete eigenvalues of $B_\beta$ via boundary element methods.

\begin{prop} \label{proposition_Birman_Schwinger_kernel_delta_prime}
  Let $\beta$ be a real valued function with $\beta^{-1} \in L^\infty(\Sigma)$ and let $B_\beta$ be defined by~\eqref{def_B_beta}. Then the map 
  \begin{equation*}
    \rho(B_\beta) \cap \rho(A_0) \ni \lambda \mapsto \big(\beta^{-1} + \mathcal{R}(\lambda) \big)^{-1}
  \end{equation*}
  can be extended to a holomorphic operator-valued function, which is holomorphic in $\rho(B_\beta)$ with respect to the toplogy in $\mathcal{B}(H^{-1/2}(\Sigma), H^{1/2}(\Sigma))$. Moreover, for $\lambda_0 \notin \sigma_\textup{ess}(B_\beta) = \sigma_\textup{ess}(A_0)$ one has $\ker(B_\beta - \lambda_0) \ominus \ker(A_0-\lambda_0) \neq \{ 0 \}$ if and only if $(\beta^{-1} + \mathcal{R}(\lambda) )^{-1}$ has a pole at $\lambda_0$ and
  \begin{equation} \label{kernel_inverse_delta_prime}
    \ker (B_\beta - \lambda_0) \ominus \ker (A_0-\lambda_0) = \big\{ \textup{DL}(\lambda_0) \varphi: \lim_{\lambda \rightarrow \lambda_0} (\lambda - \lambda_0) (\beta^{-1} + \mathcal{R}(\lambda))^{-1} \varphi \neq 0 \big\}.
  \end{equation}
\end{prop}
\begin{proof}
  The proof is similar as the proof of Proposition~\ref{proposition_Birman_Schwinger_kernel_delta} and split into 3 steps.
  
  {\it Step 1:} Define the map 
  \begin{equation*}
    [\gamma]_\Sigma: H^1_\mathcal{P}(\mathbb{R}^n \setminus \Sigma) \rightarrow H^{1/2}(\Sigma), \quad
    [\gamma]_\Sigma f := \gamma f_\text{e} - \gamma f_\text{i}.
  \end{equation*}
  Let $\lambda \in \rho(B_\beta) \cap \rho(A_0)$ be fixed. We show that
  \begin{equation} \label{Birman_Schwinger_inverse_delta_prime}
    \big(\beta^{-1} + \mathcal{R}(\lambda) \big)^{-1} = [\gamma]_\Sigma (B_\beta - \lambda)^{-1} [\gamma]_\Sigma^*.
  \end{equation}
  In particular, with Proposition~\ref{proposition_Birman_Schwinger_delta_prime}~(iv) this implies that $[\gamma]_\Sigma (B_\beta - \lambda)^{-1} [\gamma]_\Sigma$ belongs to $\mathcal{B}(H^{-1/2}(\Sigma), H^{1/2}(\Sigma))$. To show~\eqref{Birman_Schwinger_inverse_delta_prime} we note first that $[\gamma]_\Sigma f = \beta \mathcal{B}_\nu f$ holds for $f \in \dom B_\beta$ and hence~\eqref{krein_delta_prime} implies
  \begin{equation*}
    \begin{split}
      [\gamma]_\Sigma (B_\beta - \overline{\lambda})^{-1} &= \beta \mathcal{B}_\nu (B_\beta - \overline{\lambda})^{-1} \\
      &= \beta \mathcal{B}_\nu (A_0-\overline{\lambda})^{-1} - \beta \mathcal{R}(\overline{\lambda}) \big(\beta^{-1} + \mathcal{R}(\overline{\lambda}) \big)^{-1} \mathcal{B}_\nu (A_0-\overline{\lambda})^{-1} \\
      &= \beta \big[ \big(\beta^{-1} + \mathcal{R}(\overline{\lambda}) \big) - \mathcal{R}(\overline{\lambda}) \big] \big(\beta^{-1} + \mathcal{R}(\overline{\lambda}) \big)^{-1} \mathcal{B}_\nu (A_0-\overline{\lambda})^{-1} \\
      &= \big(\beta^{-1} + \mathcal{R}(\overline{\lambda}) \big)^{-1} \mathcal{B}_\nu (A_0-\overline{\lambda})^{-1},
    \end{split}
  \end{equation*}
  which implies, after taking the dual,
  \begin{equation*}
    \begin{split}
      (B_\beta - \lambda)^{-1} [\gamma]_\Sigma^* = \text{DL}(\lambda) \big(\beta^{-1} + \mathcal{R}(\lambda) \big)^{-1}.
    \end{split}
  \end{equation*}
  In particular, by Lemma~\ref{lemma_double_layer_potential} and Proposition~\ref{proposition_Birman_Schwinger_delta_prime}~(iv) the right hand side belongs to $\mathcal{B}(H^{-1/2}(\Sigma), H^1_\mathcal{P}(\mathbb{R}^n \setminus \Sigma))$ and thus the same must be true for $(B_\beta - \lambda)^{-1} [\gamma]_\Sigma^*$. Therefore, we are allowed to apply $[\gamma]_\Sigma$ and the last formula shows, with the help of Lemma~\ref{lemma_double_layer_potential}~(ii), the relation~\eqref{Birman_Schwinger_inverse_delta_prime}.

  {\it Step 2:} We show that $[\gamma]_\Sigma (B_\beta - \lambda)^{-1} [\gamma]_\Sigma^* \in \mathcal{B}(H^{-1/2}(\Sigma), H^{1/2}(\Sigma))$ for any $\lambda \in \rho(B_\beta)$ and that the mapping $\rho(B_\beta) \ni \lambda \mapsto [\gamma]_\Sigma (B_\beta - \lambda)^{-1} [\gamma]_\Sigma^*$ is holomorphic in $\mathcal{B}(H^{-1/2}(\Sigma), H^{1/2}(\Sigma))$.
  
  First, we note that $\dom B_\beta \subset H^1_\mathcal{P}(\mathbb{R}^n \setminus \Sigma)$ implies that
  \begin{equation*}
    (B_\beta - \lambda)^{-1} \in \mathcal{B}(L^2(\mathbb{R}^n), H^1_\mathcal{P}(\mathbb{R}^n \setminus \Sigma)),
  \end{equation*}
  see~\eqref{mapping_properties_resolvent} for a similar argument. Hence, by duality also 
  \begin{equation*}
    (B_\beta - \lambda)^{-1} \in \mathcal{B}((H^1_\mathcal{P}(\mathbb{R}^n\setminus \Sigma))^*, L^2(\mathbb{R}^n)).
  \end{equation*}
  With the resolvent identity this implies for any $\lambda_0 \in \rho(B_\beta)$ and $\lambda \in \rho(B_\beta) \cap \rho(A_0)$, in a similar way as in the proof of Proposition~\ref{proposition_resolvent}, that
  \begin{equation*}
    \begin{split}
      [\gamma]_\Sigma (B_\beta - \lambda_0)^{-1} [\gamma]_\Sigma^* &- [\gamma]_\Sigma (B_\beta - \lambda)^{-1} [\gamma]_\Sigma^* \\
      &= (\lambda_0 - \lambda) [\gamma]_\Sigma (B_\beta - \lambda_0)^{-1} (B_\beta - \lambda)^{-1} [\gamma]_\Sigma^*,
    \end{split}
  \end{equation*}
  which shows first with~\eqref{Birman_Schwinger_inverse_delta_prime} that $[\gamma]_\Sigma (B_\beta - \lambda_0)^{-1} [\gamma]_\Sigma^* \in \mathcal{B}(H^{-1/2}(\Sigma), H^{1/2}(\Sigma))$ and in a second step, that $[\gamma]_\Sigma (B_\beta - \lambda)^{-1} [\gamma]_\Sigma^*$ is holomorphic in $\mathcal{B}(H^{-1/2}(\Sigma), H^{1/2}(\Sigma))$, which shows the claim of this step.

  {\it Step 3:} Finally, it follows from Proposition~\ref{proposition_Birman_Schwinger_delta_prime}~(i) that $\ker(B_\beta - \lambda_0) \ominus \ker(A_0-\lambda_0) \neq \{ 0 \}$ if and only if there exists $\varphi \in \mathcal{N}_{\lambda_0}$ such that $(\beta^{-1} + \mathcal{R}(\lambda_0)) \varphi = 0$, i.e. if and only if $\lambda \mapsto (\beta^{-1} + \mathcal{R}(\lambda))^{-1}$ has a pole at $\lambda_0$. This shows immediately \eqref{kernel_inverse_delta_prime}.
\end{proof}

\subsection{Numerical approximation of discrete eigenvalues of $B_\beta$}\label{SubSectNumDeltaPrime}
\newcommand{\mr}{\mathcal{R}}
\newcommand{\mrf}{(\beta^{-1}+\mr(\cdot))}
\newcommand{\mrflambda}{\beta^{-1}+\mr(\lambda)}

The approximation of the discrete eigenvalues of $B_\beta$ by boundary element methods is based on the same principles as those for the discrete eigenvalues of $A_\alpha$  described in Section~\ref{SubSection:ApproxDelta}.
In order to apply boundary element methods for the approximation of the discrete eigenvalues of $B_\beta$ it is necessary to have an integral representation of the paramatrix  ${\mathcal G}(\lambda)$ of ${\mathcal P}-\lambda$ or at least a good approximation of the boundary integral operator $\mathcal{R}(\lambda)$.
We use the characterization of the discrete   eigenvalues of $B_\beta$ in terms of boundary integral operators  given in Proposition~\ref{proposition_Birman_Schwinger_delta_prime} and   Proposition~\ref{proposition_Birman_Schwinger_kernel_delta_prime}. The discrete eigenvalues split into  the eigenvalues of the nonlinear eigenvalue problem
\begin{equation}\label{Eq:EigContinuousDeltaPrimeBIE}
(\beta^{-1} +\mr(\lambda))\psi=0
\end{equation}
in $\rho(A_0)$ and into distinct discrete eigenvalues of $A_0$ which are either the poles of $\ma(\cdot)$ satisfying the properties specified in Proposition~\ref{proposition_Birman_Schwinger_delta_prime}~(iii) or poles of $\mrf^{-1} $ in $\sigma_\text{disc}(A_0)$.

In the following presentation of the boundary element method we want to consider first the case that $\sigma_\text{disc}(A_0)=\varnothing$ and then the general case.
If $\sigma_\text{disc}(A_0)=\varnothing$,  then the discrete eigenvalues of $B_\beta$ coincide 
with the eigenvalues of the nonlinear eigenvalue problem~\eqref{Eq:EigContinuousDeltaPrimeBIE} in $\rho(A_0)$ as shown in Proposition~\ref{proposition_Birman_Schwinger_delta_prime}~(ii). In this situation a complete  convergence analysis   is provided by the theory of Section~\ref{SectionGalerkinApprox}. For the general case the convergence of the approximations of the discrete eigenvalues of $B_\beta$ which lie in $\sigma_\text{disc}(A_0)$ is an open issue.

The discretized problems for the approximation of the discrete eigenvalues of $B_\beta$ which result  from the approximations of the boundary integral operators by boundary element methods are problems for the determination of poles of meromorphic matrix-valued functions. For this kind of problems we suggest the contour integral method~\cite{Beyn}, which was  discussed in Section~\ref{SubsubsectionApproxDeltaEinfach}.

\subsubsection{Approximation of discrete eigenvalues of $B_\alpha$ for the case $\sigma_\textnormal{disc}(A_0)=\varnothing$} 
For $\sigma_\text{disc}(A_0)=\varnothing$ the discrete eigenvalues of $B_\beta$ coincide, according to Proposition~\ref{proposition_Birman_Schwinger_delta_prime}~(ii), with the eigenvalues of the eigenvalue problem for $\mrf$. Lemma~\ref{lemma_hypersing_bdd_int_op}~(iii) shows that the map $\rho(A_0)\ni\lambda\mapsto(\mrflambda)$ is holomorphic in $\mathcal{B}(H^{1/2}(\Sigma),H^{-1/2}(\Sigma))$. Moreover, by Lemma~\ref{lemma_hypersing_bdd_int_op}~(i) the operators $\mrflambda$ satisfy for $\lambda\in\rho(A_0)$ G\r{a}rding's  inequality of the form~\eqref{Inequality:Garding}. Hence, any conforming  Galerkin method for the approximation of the eigenvalue problem~\eqref{Eq:EigContinuousDeltaPrimeBIE} is a convergent method, which follows from the theory in Section~\ref{SectionGalerkinApprox}.  

For the boundary element approximation  of the eigenvalue problem~\eqref{Eq:EigContinuousDeltaPrimeBIE} we first assume that $\Omega_\text{i}$ is a polyhedral Lipschitz domain. The case of general Lipschitz domains is addressed in Remark~\ref{Remark:GeneralLipschitzDomainDeltaPrime}. Let $(\triang_N)_{N\in\N}$ be a sequence of quasi-uniform triangulations of $\Omega_\text{i}$ with the properties specified in~\eqref{Eq:PropertiesOfTriangulations}. As approximation space for the approximation of the eigenfunctions of the eigenvalue problem~\eqref{Eq:EigContinuousDeltaPrimeBIE} we choose  the space $S^1(\triang_N)$ of piecewise linear functions with respect to the triangulation $\triang_N$. 
The approximation property of $S_1(\triang_N)$ depends on the regularity of the functions which are approximated. In order to measure the regularity of functions defined on a piecewise smooth boundary $\Sigma$,  partitioned by open sets $\Sigma_1,\ldots,\Sigma_J$ such that  
\begin{equation*}
 \Sigma=\bigcup_{j=1}^J\overline{\Sigma}_j,\; \Sigma_j\cap\Sigma_i =\varnothing\;\text{for }i\neq j,
\end{equation*}
so-called piecewise Sobolev spaces of order $s> 1$ defined by   
\begin{equation*}
\begin{split}
 H_{\text{pw}}^{s}(\Sigma):=\{v\in H^1(\Sigma)\;:\;v\upharpoonright\Sigma_j\in H^s(\Sigma_j)\text{ for }j=1,\ldots, J\}
\end{split}
\end{equation*}
are used,  see  \cite[Definition~4.1.48]{SauterSchwab}. For $s\in[0,1]$ the space $H^s_{\text{pw}}(\Sigma)$ is defined by  $H^s_{\text{pw}}(\Sigma):=H^s(\Sigma)$. If $W$ is  a finite dimensional subspace of $H^{1/2+s}_{\text{pw}}(\Sigma)$ for $s\in(0,\tfrac{3}{2}]$, then  
\begin{equation}\label{Eq:ApproxPiecewiseLinear}
 \delta_{H^{1/2}(\Sigma)}(W, S_1(\triang_N))
=\sup_{\substack{w\in W\\
\|w\|_{H^{1/2}(\Sigma)}= 1}} \inf_{\psi_N\in S_1(\triang_N)}\|w-\psi_N\|_{H^{1/2}(\Sigma)}
 =\mathcal{O}(h(N)^{s})
\end{equation}
holds~\cite[Proposition~4.1.50]{SauterSchwab}. 

The Galerkin approximation of the eigenvalue problem~\eqref{Eq:EigContinuousDeltaPrimeBIE} reads as follows: find eigenvalues $\lambda_N\in\C$ and corresponding eigenfunctions $\psi_N\in S_1(\triang_N)\setminus\{0\}$ such that
\begin{equation}\label{Eq:GalerkinEVPDeltaPrime}
\left((\beta^{-1} +  \mathcal{R}(\lambda_N)) \psi_N,\chi_N\right) = 0\qquad\forall \chi_N\in S_1(\triang_N).
\end{equation}
We can apply all convergence results from Theorem~\ref{Theorm:AbstrctConvergenceResults}   to the Galerkin eigenvalue problem~\eqref{Eq:GalerkinEVPDeltaPrime}. 
In the following theorem  the asymptotic convergence order of the approximations of the eigenvalues and the corresponding eigenfunctions are specified. 
\begin{thm}\label{Thm:ConvergenceDeltaPrime}
Let $D\subset\rho(A_0)$ be a compact  and    connected set in $\C$ such that $\partial D$ is a simple 
rectifiable curve. Suppose that $\lambda\in
\mathring{D}$ is the only eigenvalue of $\mrf$ in $D$ and that 
$\ker(\mrflambda)\subset H^{1/2+s}_{\text{pw}}(\Sigma)$ for some $s\in(0,\tfrac{3}{2}]$.
Then
there exist an
$N_0\in\N$ and a constant  $c>0$ such that for all $N\geq N_0$  we
have:
\begin{itemize}
\item[(i)] For all eigenvalues $\lambda_N$ of the Galerkin eigenvalue problem~\eqref{Eq:GalerkinEVPDeltaPrime} in $D$
\begin{equation}\label{Eq:ErrorApproxBoundaryDeltaPrime}
 \vert \lambda-\lambda_N\vert \leq
c (h(N))^{2s}
\end{equation}
holds.
\item[(ii)] If  $(\lambda_N,u_N)$ is an eigenpair of 
\eqref{Eq:GalerkinEVPDeltaPrime} with $\lambda_N\in D$ and  
$\|\psi_N\|_{H^{1/2}(\Sigma)}=1$, then 
\begin{equation}\label{Eq:ErrorApproxBoundaryDeltaPrimeFunc}
 \inf_{\psi\in \ker(\mrflambda)}\| \psi-\psi_N\|_{H^{1/2}(\Sigma)}
 \leq 
 c 
 \left(\vert \lambda_N-\lambda \vert + (h(N))^s \right).
\end{equation}
\end{itemize}
\end{thm}
\begin{proof}
We have shown in Lemma~\ref{lemma_hypersing_bdd_int_op}~(iii) that the map $\rho(A_0)\ni\lambda\mapsto(\mrflambda)$ is holomorphic in $\mathcal{B}(H^{1/2}(\Sigma),H^{-1/2}(\Sigma))$. Moreover, by Lemma~\ref{lemma_hypersing_bdd_int_op}~(i),  the operators $\mrflambda$ satisfy for $\lambda\in\rho(A_0)$ G\r{a}rding's  inequality of the form~\eqref{Inequality:Garding}.
The Galerkin approximation~\eqref{Eq:GalerkinEVPDeltaPrime} of the eigenvalue problem for $\mrf$ is a conforming approximation since $S^1(\triang_N)$ is a subspace of $H^{1/2}(\Sigma)$. Hence, we can use Theorem~\ref{Theorm:AbstrctConvergenceResults}. The error estimates follows from the approximation property~\eqref{Eq:ApproxPiecewiseLinear} of $S^1(\triang_N)$ and the fact, that the eigenfunction of the adjoint problem are as regular as for  $\mrf$.
\end{proof}

\begin{remark}\label{Remark:GeneralLipschitzDomainDeltaPrime}
If $\Omega$ is a bounded Lipschitz domain with a curved piecewise $C^2$-boundary the approximation of the boundary by a triangulation with flat triangles as described in~\cite[Chapter~8]{SauterSchwab} reduces the maximal possible convergence order $s$ for the error of the eigenvalues in~\eqref{Eq:ErrorApproxBoundaryDeltaPrime} and  for the error of the eigenfunctions in~\eqref{Eq:ErrorApproxBoundaryDeltaPrimeFunc} from $s=\tfrac{3}{2}$ to $s=1$. 
This follows from  the results of the discretization of boundary integral operators for approximated boundaries \cite[Chapter~8]{SauterSchwab} and from the abstract results of eigenvalue problem approximations \cite{Karma1,Karma2}.   
\end{remark}

\subsubsection{Approximation of discrete eigenvalues of $B_\alpha$ for the case $\sigma_\textnormal{disc}(A_0)\neq\varnothing$ }
If $\sigma_\text{disc}(A_0)\neq\varnothing$, then  Proposition~\ref{proposition_Birman_Schwinger_delta_prime} and   Proposition~\ref{proposition_Birman_Schwinger_kernel_delta_prime} imply that the discrete eigenvalues of $B_\beta$ are  poles of $\mrf^{-1}$ or poles of $\ma(\cdot)$ with the property given in Proposition~\ref{proposition_Birman_Schwinger_delta_prime}~(iii). These characterizations are used for the approximation of the discrete eigenvalues of $B_\beta$.  We will separately discuss both cases.

Let $\lambda_0$ be a discrete eigenvalue of $B_\beta$ and in addition be a pole of $\mrf^{-1}$. Then $\mrf$ is  either holomorphic in $\lambda_0$, which is the case for $\lambda_0\in\rho(A_0)$, or $\lambda_0$ is a pole of $\mrf$. 
A pole $\lambda_0\in\rho(A_0)$ of $\mrf^{-1}$ can be considered as an eigenvalue of the eigenvalue problem for the homomorphic   Fredholm operator-valued function $\mrf$ in $\rho(A_0)$ and the convergence results of Section~\ref{SectionGalerkinApprox} can be applied with the same reasoning as in the case of $\sigma_\text{disc}(A_0)=\varnothing$.  If    $\lambda_0\in\sigma_\text{disc}(A_0)$ is a pole of $\mrf$, then the convergence theory of Section~\ref{SectionGalerkinApprox} is not applicable for $\lambda_0$. We expect convergence of the approximations for this kind of poles of $\mrf^{-1}$, but a rigorous numerical analysis has not established so far. 
   
The approximation of a discrete eigenvalue $\lambda_0$  of $B_\beta$ which is not a pole of $\mrf^{-1}$ is based on the following characterization  from Proposition~\ref{proposition_Birman_Schwinger_delta_prime}~(iii): $\lambda_0$ is a pole of $\ma(\cdot)$ and there exists a  pair $(0,0)\neq(\psi,\varphi)\in H^{1/2}(\Sigma)\times H^{-1/2}(\Sigma)$ such that
\begin{equation}\label{Eq:EigContAinvDeltaPrime}
 {\mathcal A}(\lambda)^{-1} 
\begin{pmatrix}
 \psi\\
\varphi 
\end{pmatrix}
  =\begin{pmatrix}
 0\\
0
\end{pmatrix}
\qquad\text{and}
\qquad\psi=0.
\end{equation}
For the approximation of the eigenvalues of the nonlinear eigenvalue problem in~\eqref{Eq:EigContAinvDeltaPrime} formally the Galerkin problem in $S^1(\triang_N)\times S^0(\triang_N)$ as given  in~\eqref{Eq:EigContAinvDiscret} is considered.  If the contour integral method is used for the computations of the approximations of the eigenvalues for $\ma(\cdot)^{-1}$, then $\ma(\cdot)^{-1}$ does not have not be computed, but instead its inverse 
$\ma(\cdot)$. The convergence theory of Section~\ref{SectionGalerkinApprox} can be applied to the approximation of those eigenvalues of $\ma(\cdot)^{-1}$ for which $\ma(\cdot)^{-1}$ is holomorphic. If $\lambda_0$ is a pole $
\ma(\cdot)$ and of $\ma(\cdot)^{-1}$, we again expect  convergence, but a numerical analysis for such kind of poles has not been provided so far.  
\subsection{Numerical examples}\label{SubSecNumPrime}
For the numerical examples of the approximation of discrete eigenvalues of $B_\beta$ we choose $\mathcal{P}=-\Delta$. In this case $\sigma_\text{ess}(A_0)=[0,\infty)$,  $\sigma_\text{disc}(A_0)=\varnothing$, the fundamental solution is given by  $G(\lambda;x,y)=e^{i\sqrt{\lambda}\|x-y\|}(4\pi\|x-y\|)^{-1}$,  and the discrete eigenvalues of $B_\beta$ coincide with the eigenvalues of the nonlinear eigenvalue problem for $\mrf$. The Galerkin eigenvalue problem~\eqref{Eq:GalerkinEVPDeltaPrime} is used for the computation of approximations of discrete eigenvalues of $B_\beta$ and corresponding eigenfunctions.

\subsubsection{Unit ball} We consider  for the first numerical example as domain $\Omega_\text{i}$ again the unit ball.  
Analytical representations for  the discrete eigenvalues of $B_\beta$ are known in this case~\cite[Section~6]{AGS87}  and are used  to compute the errors of the approximations and to check the predicted asymptotic error estimate~\eqref{Eq:ErrorApproxBoundaryDeltaPrime}. 
The errors of the approximations of the eigenvalues of $B_\beta$ with $\beta^{-1}=-1.5$ which lie inside the contour  $g(t)=c+a\cos(t)+i b \sin(t)$, $t\in [0,2\pi]$, with $c=-6.0$,  $a=5.99$  
and  $b=0.01$
 are  given in Table~\ref{Table:ErrorsPrime} for three different mesh sizes $h$. We denote by  $\widehat{\lambda}^{(l)}_{h}$, $l=1,2$,  the mean value of the approximations of the multiple eigenvalues $\lambda^{(l)}$.
A quadratic experimental convergence order (eoc) can be observed  which is according  to Remark~\ref{Remark:GeneralLipschitzDomainDeltaPrime} the best possible convergence order if flat triangles are used for the triangulation of a curved boundary as it has been  done in our experiments.  In
Figure~\ref{Fig:EigFunctionsPlanePrime} plots of  computed eigenfunctions of
$B_{\beta}$ in the $xy$-plane are given where for each exact eigenvalue
one approximated eigenfunction is selected.
\begin{table}[h]
\begin{center}
\begin{tabular}{l r r  r  r  r  r r r r r  r  r  r  r r r}
\toprule
$h$   &  & $ \frac{\left\vert\lambda^{(0)}_{h}-\lambda^{(0)}\right\vert}{\left\vert \lambda^{(0)}\right\vert} $ &eoc& &
$ \frac{\left\vert\widehat\lambda^{(1)}_{h}-\lambda^{(1)}\right\vert}{\left\vert \lambda^{(1)}\right\vert} $
 & eoc& & 
 $ \frac{\left\vert\widehat\lambda^{(2)}_{h}-\lambda^{(2)}\right\vert}{\left\vert \lambda^{(2)}\right\vert} $
 & eoc \\
\midrule  
0.2   & &3.232e-3 &-&&1.885e-3 & -&& 6.745e-3&-\\
0.1 & &  7.099e-4&2.19 &&3.926e-4 & 2.26& &1.406e-3&2.26\\
0.05  & &  1.635e-4& 2.11&&8.958e-5 &2.13& &3.054e-4&2.20\\
\bottomrule
\end{tabular} 
\end{center}
\caption{Error of the approximations of the eigenvalues of
$B_{\beta}$, $\beta^{-1}=-1.5$, of the unit ball for different mesh-sizes
$h$. }\label{Table:ErrorsPrime}
\end{table}

\begin{figure}[ht]
 \begin{center}
 \includegraphics[width=0.32\textwidth]{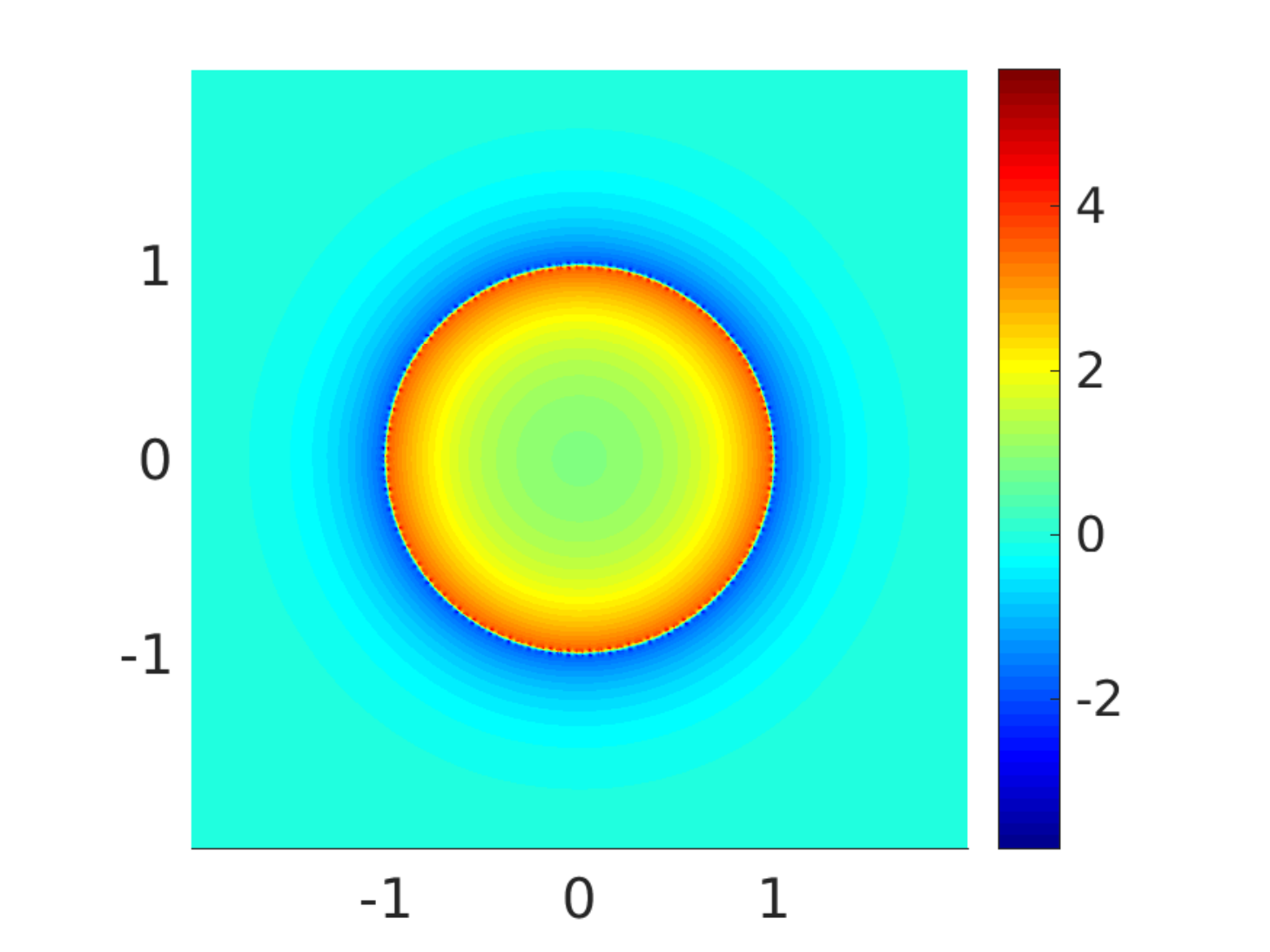}    
  \includegraphics[width=0.32\textwidth]{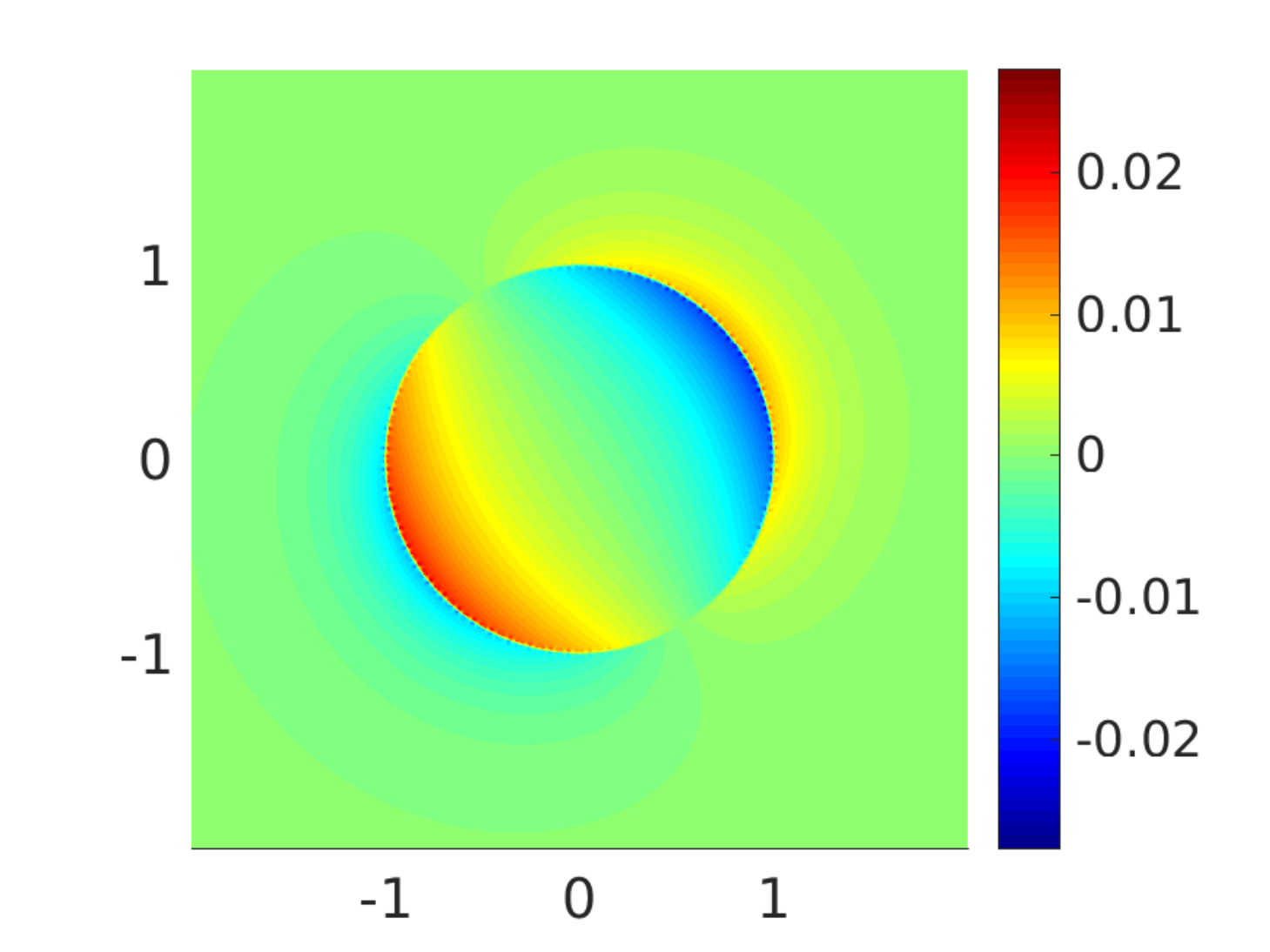} 
   \includegraphics[width=0.32\textwidth]{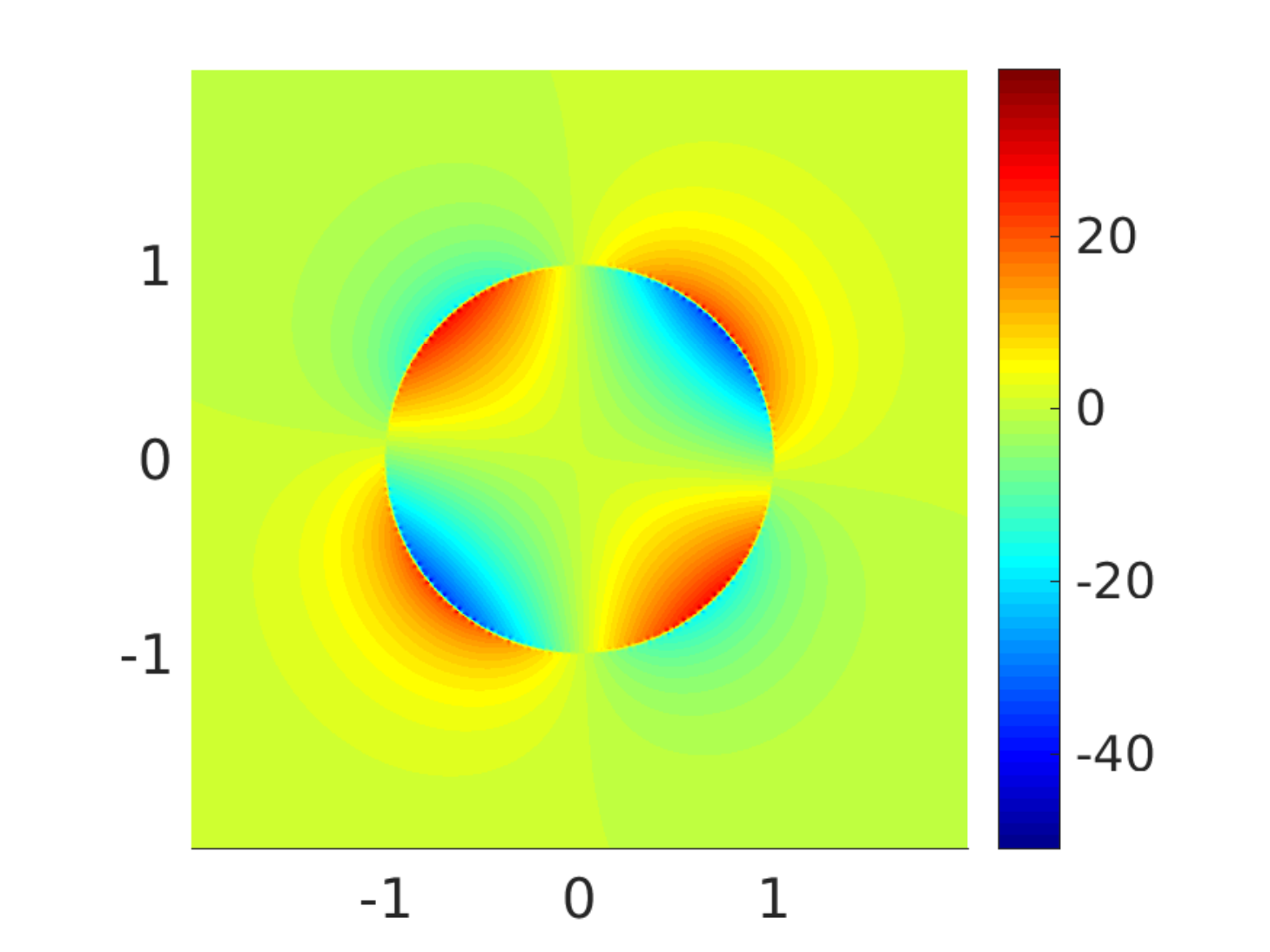}     
 \end{center}
\caption{ Computed eigenfunctions of
 $B_{\beta}$, $\beta^{-1}=-1.5$, in the
 $xy$-plane for the unit ball.}\label{Fig:EigFunctionsPlanePrime}
  \end{figure}
\subsubsection{L-shape domain}In the second  numerical example we have chosen as domain $\Omega_\text{i}$ a so-called L-shape domain with $\Omega_\text{i}= (-1,1)^3\setminus([0,1]^2\times[-1,1])$ and we have set $\beta^{-1}=-0.75$.  
In the numerical experiments the ellipse 
$g(t)=c+a\cos(t)+i b \sin(t)$, $t\in [0,2\pi]$, with $c=-4.0$,  $a=3.99$  
and  $b=0.01$, is taken as contour for the contour integral method.  
We have got three eigenvalues of the discretized eigenvalue problem inside this contour,  namely $\lambda_h^{(0)}= -5.54$, $\lambda_h^{(1)}=-4.41$ and $\lambda_h^{(2)}=-2.94$ for the mesh-size $h=0.1$. Plots of the numerical approximations of the eigenfunctions in the $xy$-plane are given in Figure~\ref{Fig:EigFunctionsLShape}.
 \begin{figure}[t]
 \begin{center}
 \includegraphics[width=0.32\textwidth]{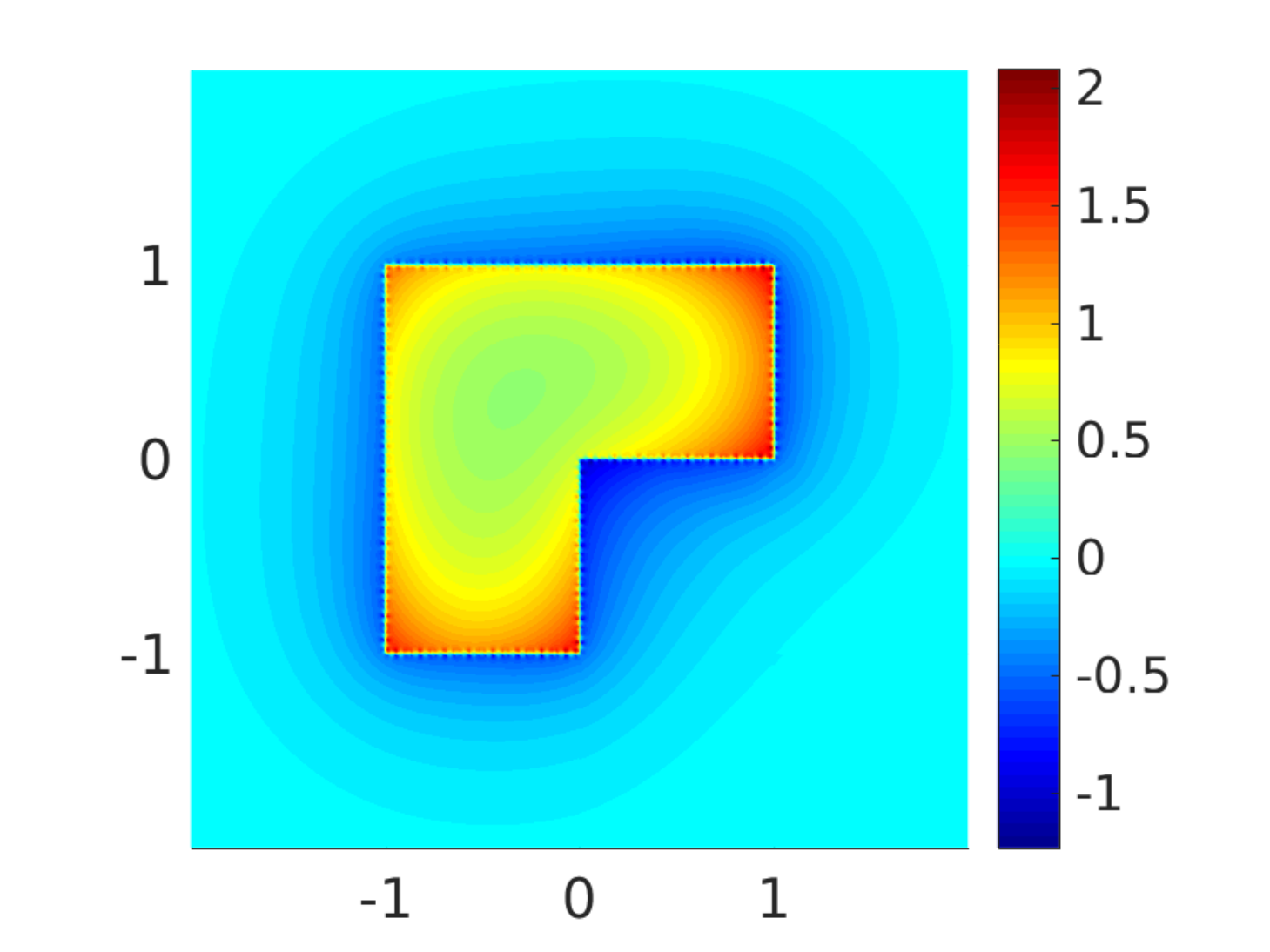}    
  \includegraphics[width=0.32\textwidth]{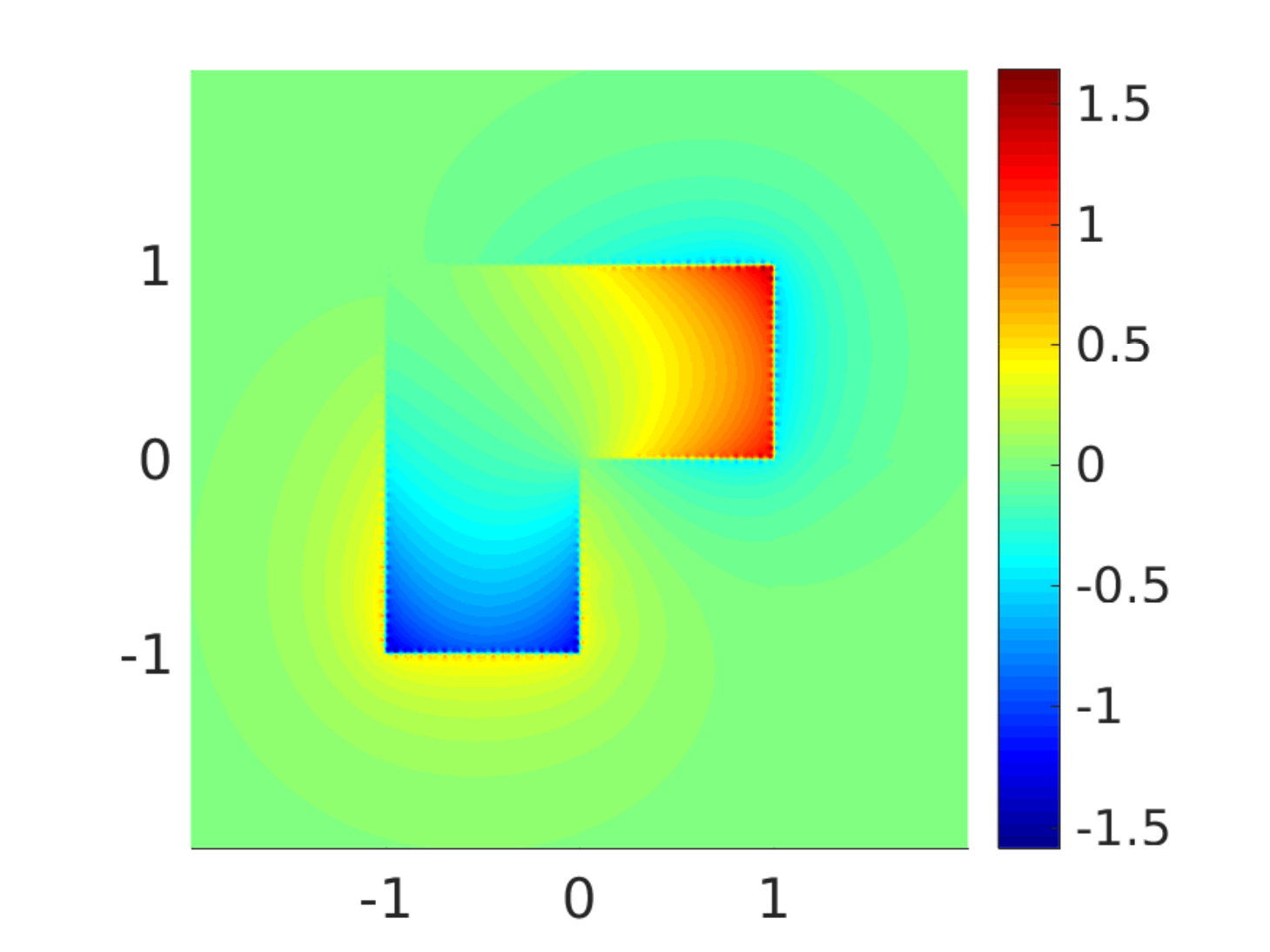}     
   \includegraphics[width=0.32\textwidth]{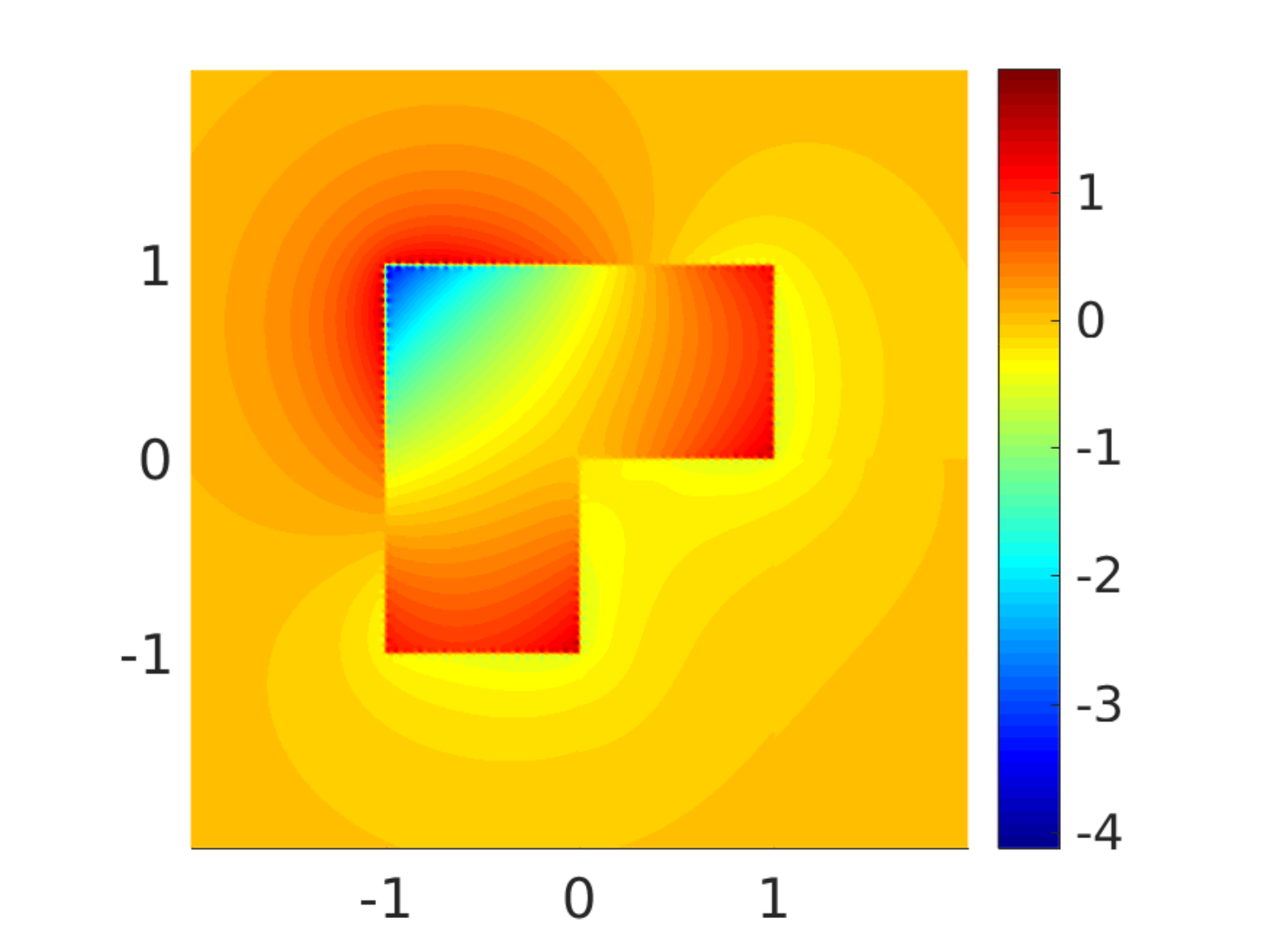}     
 \end{center}
 \caption{Computed eigenfunctions of
 $B_{\beta}$, $\beta^{-1}=-0.75$, for the  L-shape domain $\Omega_\text{i}= (-1,1)^3\setminus([0,1]^2\times[-1,1])$ in the
 $xy$-plane. }\label{Fig:EigFunctionsLShape}
 \end{figure}


\vglue 1cm 
{\small


\begin{thebibliography}{100}

\bibitem{AGHH05}
S.~Albeverio, F.~Gesztesy, R.~H{\o}egh-Krohn, and H.~Holden,
Solvable Models in Quantum Mechanics. With an Appendix by Pavel Exner,
2nd ed., Amer. Math. Soc. Chelsea Publishing, Providence, RI (2005).

\bibitem{AGS87} J.~Antoine, F.~Gesztesy and J. Shabani, {\it Exactly solvable models of sphere interactions in quantum mechanics}, J. Phys. A \textbf{20} (1987), 3687--3712.

\bibitem{BEHL17}
J.~Behrndt, P.~Exner, M.~Holzmann, and V.~Lotoreichik,
{\it Approximation of Schr\"{o}dinger operators with $\delta$-interactions supported on hypersurfaces}, 
Math. Nachr. {\bf 290}(8-9) (2017), 1215--1248.



\bibitem{BEHL19}
J.~Behrndt, P.~Exner, M.~Holzmann, and V.~Lotoreichik,
{\it The Landau Hamiltonian with \boldmath{$\delta$}-potentials supported on curves},
arxiv:1812.09145 (2018).

\bibitem{BEHL19QS}
J.~Behrndt, P.~Exner, M. Holzmann, V.~Lotoreichik,
{\it On Dirac operators in $\mathbb{R}^3$ with electrostatic and Lorentz scalar $\delta$-shell interactions},
To appear in Quantum Stud.: Math. Found., DOI: \url{https://doi.org/10.1007/s40509-019-00186-6}. 

\bibitem{BGLL15} J.~Behrndt, G.~Grubb, M.~Langer, and V.~Lotoreichik,
\textit{Spectral asymptotics for resolvent differences of elliptic
              operators with {$\delta$} and {$\delta'$}-interactions on
              hypersurfaces}, J. Spectr. Theory
\textbf{5}(4) (2015), 697--729.


\bibitem{BLL13b} J.~Behrndt, M.~Langer, and V.~Lotoreichik,
\textit{Schr\"{o}dinger operators with $\delta$ and $\delta^\prime$-potentials supported on hypersurfaces}, Ann. Henri Poincar\'e
\textbf{14} (2013), 385--423.

\bibitem{BR12}
J.~Behrndt and J.~Rohleder,
{\it An inverse problem of {C}alder\'{o}n type with partial data},
Comm. Partial Differential Equations {\bf 37}(6) (2012), 1141--1159.

\bibitem{BR15}
J.~Behrndt and J.~Rohleder,
{\it Spectral analysis of selfadjoint elliptic differential operators, 
{D}irichlet-to-{N}eumann maps, and abstract {W}eyl functions},
Adv. Math. {\bf 285} (2015), 1301--1338.

\bibitem{Beyn} W.~J.~Beyn,
{\it An integral method for solving nonlinear eigenvalue problems},
Linear Algebra Appl. \textbf{432} (2012), 3839--3863.

\bibitem{BEKS}
J.~Brasche, P.~Exner, Y.~Kuperin, and P.~\v{S}eba,
{\it Schr\"odinger operators with singular interactions},
J.\ Math.\ Anal.\ Appl. \textbf{184} (1994), 112--139.

\bibitem{BFT98}
J.~Brasche, R.~Figari, and A.~Teta,
{\it Singular {S}chr\"{o}dinger operators as limits of point  interaction {H}amiltonians},
Potential Anal. \textbf{8} (1998), 163--178.

\bibitem{BMNW08} M.~Brown, M.~Marletta, S.~Naboko, and I.~Wood, {\it Boundary triplets and {$M$}-functions for non-selfadjoint operators, with applications to elliptic {PDE}s and block
              operator matrices}, J. Lond. Math. Soc. (2)
\textbf{77}(3) (2008), 700--718.

\bibitem{BO07}
J.~Brasche and K.~O\v{z}anov\'{a},
{\it Convergence of Schr\"{o}dinger operators},
SIAM\ J.\ Math.\ Anal. \textbf{39} (2007), 281--297.



\bibitem{E08} P.~Exner,
Leaky quantum graphs: a review,
in: \textit{Analysis on graphs and its applications}.
Selected papers based on the Isaac Newton Institute for
Mathematical Sciences programme, Cambridge, UK, 2007.
\textit{Proc.\ Symp.\ Pure Math.} \textbf{77} (2008), 523--564.

\bibitem{EK15}
P.~Exner and H.~Kova{\v{r}}{\'{\i}}k,
Quantum Waveguides.
Theoretical and Mathematical Physics, Springer (2015).

\bibitem{EN03}
P.~Exner and K.~N\v{e}mcov\'{a},
{\it Leaky quantum graphs: approximations by point-interaction {H}amiltonians},
J. Phys. A \textbf{36} (2003), 10173--10193.

\bibitem{FK}
A.~Figotin and P.~Kuchment,
{\it Band-gap structure of spectra of periodic dielectric and acoustic media. {II}. {T}wo-dimensional photonic crystals},
SIAM J. Appl. Math. \textbf{56}(6) (1996), 1561--1620.


\bibitem{GohbergGoldberg1990}
I.~C.~Gohberg, S.~Goldberg, M.~A.~Kaashoek,
Classes of {L}inear {O}perators. {V}ol. {I}.
Birkh{\"a}user Verlag, Basel (1990).

\bibitem{GohbergSigal:1971}
I.~C.~Gohberg and E.~I.~Sigal,
{\it An operator generalization of the logarithmic residue theorem and {R}ouch\'e's theorem},
Math. USSR-Sb. \textbf{13} (1971), 603--625.


\bibitem{Karma1} O.~Karma, {\it Approximation in eigenvalue problems for
holomorphic {F}redholm operator functions. {I}}, Numer. Funct. Anal. Optim.
\textbf{17} (1996), 365--387.

\bibitem{Karma2} O.~Karma, {\it Approximation in eigenvalue problems for
holomorphic {F}redholm operator functions. {II}. ({C}onvergence rate) }, Numer.
Funct. Anal. Optim.
\textbf{17} (1996), 389--408.

\bibitem{K95}
T.~Kato,
Perturbation theory for linear operators.
Springer-Verlag, Berlin, reprint of the 1980 edition (1995).

\bibitem{KSU}
A.~Kimeswenger, O.~Steinbach, and G.~Unger,
{\it Coupled finite and boundary element methods for fluid-solid
              interaction eigenvalue problems},
SIAM J. Numer. Anal. \textbf{52}(5) (2014), 2400--2414.


\bibitem{Kleefeld:2013}
A.~Kleefeld,
{\it A numerical method to compute interior transmission eigenvalues},
Inverse Problems \textbf{29}(10) (2013), 104012, 20.



\bibitem{KozlovMazya:1999}
V.~Kozlov and V.~Maz$'$ya,
Differential Equations with Operator Coefficients with Applications to Boundary Value Problems for Partial Differential Equations.
Springer-Verlag, Berlin (1999).

\bibitem{KP31}
R. de L.~Kronig and W.~Penney, 
{\it Quantum mechanics of electrons in crystal lattices},
Proc.\ Roy.\ Soc.\ Lond. \textbf{130} (1931), 499--513.

\bibitem{M00}
W.~McLean,
Strongly Elliptic Systems and Boundary Integral Equations,
Cambridge University Press, Cambridge (2000).

\bibitem{MPS16} A.~Mantile, A.~Posilicano, and M.~Sini,
\textit{Self-adjoint elliptic operators with boundary conditions on
              not closed hypersurfaces}, J. Differential Equations
\textbf{261}(1) (2016), 1--55.

\bibitem{O06}
K.~O\v{z}anov\'{a},
{\it Approximation by point potentials in a magnetic field},
J. Phys. A \textbf{39} (2006), 3071--3083.


\bibitem{SauterSchwab} S.~A.~Sauter and C.~Schwab,
Boundary element methods.
Springer-Verlag, Berlin (2011).

\bibitem{BEMplusplus:2015}
W.~{\'S}migaj, S.~Arridge, T.~Betcke , J.~Phillips,  and  J.~Schweiger,
{\it Solving Boundary Integral Problems with {BEM}++},
ACM Trans. Math. Software \textbf{41}(6) (2015), 1--40.


\bibitem{Steinbach}
O.~Steinbach,
Numerical approximation methods for elliptic boundary value problems. Finite and boundary elements. 
Springer, New York (2008).

\bibitem{SU12} 0.~Steinbach and G.~Unger, {\it Convergence analysis of a {G}alerkin
boundary element method for the {D}irichlet {L}aplacian eigenvalue problem},
SIAM J. Numer. Anal.  \textbf{50} (2012), 710--728.

\bibitem{T35} 
L.\,H.~Thomas,
{\it The interaction between a neutron and a proton and the structure of $H^3$}, 
Phys. Rev., II. Ser. \textbf{47} (1935), 903--909.

\bibitem{Unger:2009}
G.~Unger, 
Analysis of Boundary Element Methods for Laplacian Eigenvalue Problems.
PhD thesis, Graz University of Technology, Graz, 2009.

\end{thebibliography}
\end{document}